\documentclass[11pt]{article}
\usepackage{amsmath,amsthm,amssymb,enumerate,mathscinet}
\usepackage{fullpage}
\usepackage{graphicx}

\usepackage{caption}
\usepackage{mathrsfs}  
\usepackage[backref,colorlinks]{hyperref} 

\usepackage[abbrev,msc-links]{amsrefs}   
\usepackage[utf8]{inputenc}
\usepackage{tikz}

\usepackage{enumerate}
\usepackage{cleveref}     

\newtheorem{theorem}{Theorem}[section]

\newtheorem{lemma}[theorem]{Lemma}

\newtheorem{proposition}[theorem]{Proposition}

\newtheorem{claim}[theorem]{Claim}
\newtheorem{definition}[theorem]{Definition}

\newtheorem{thm}[theorem]{Theorem}

\newtheorem{cor}[theorem]{Corollary}

\newtheorem{defin}[theorem]{Definition}

\usepackage{lineno}
\theoremstyle{plain}

\newtheorem*{clm*}{Claim}

\theoremstyle{definition}
\newtheorem{rem}[theorem]{Remark}

\numberwithin{equation}{section}

\numberwithin{equation}{section}

\def\eps{\varepsilon}

\let\eps=\epsilon
\let\ep=\epsilon
\let\theta=\vartheta
\let\rho=\varrho
\let\phi=\varphi

\def\NN{\mathbb N}

\def\cA{{\mathcal A}}
\def\cH{{\mathcal H}}

\def\cF{{\mathcal F}}

\def\cP{{\mathcal P}}

\def\cR{{\mathcal R}}
\def\cS{{\mathcal S}}

\def\cK{{\mathcal K}}
\def\cU{{\mathcal U}}

\def\be{\mathop{\text{\rm Be}}\nolimits}

\newcommand{\floor}[1]{\left\lfloor#1\right\rfloor}
\newcommand{\ceil}[1]{\left\lceil#1\right\rceil}

\usepackage{enumitem} 

\let\polishlcross=\l
\def\l{\ifmmode\ell\else\polishlcross\fi}

\usepackage{bm}

\def\COMMENT#1{}
\let\COMMENT=\footnote% COMMENT OUT for clean output

\title{Tilings in randomly perturbed graphs: bridging the gap between  Hajnal--Szemer\'edi and  Johansson--Kahn--Vu} 
\author{Jie Han\footnote{University of Rhode Island, Kingston, RI, USA, {\tt jie\_han@uri.edu}.}, Patrick Morris\footnote{Freie Universit\"at Berlin, Germany and Berlin Mathematical School, Germany, {\tt pm0041@mi.fu-berlin.de}. Research supported by a Leverhulme Trust Study Abroad
  Studentship (SAS-2017-052$\backslash$9).}  \, and
 Andrew Treglown\footnote{University of Birmingham, United Kingdom, {\tt a.c.treglown@bham.ac.uk}. Research supported by EPSRC grant EP/M016641/1.}}
\begin{document}
%\linenumbers

 \maketitle
\begin{abstract}
A perfect $K_r$-tiling in a graph $G$ is a collection of vertex-disjoint copies of $K_r$ covering all the vertices in $G$. 
In this paper we consider perfect $K_r$-tilings in the setting of randomly perturbed graphs;  a model introduced by Bohman, Frieze and Martin~\cite{bfm1} where one starts with a dense graph and then 
adds $m$ random edges to it. Specifically,  given any fixed $0< \alpha <1-1/r$ we determine how many random edges one must add to an $n$-vertex graph $G$ of minimum degree $\delta (G) \geq \alpha n$ to ensure that,
asymptotically almost surely, the resulting graph contains a perfect $K_r$-tiling. As one increases $\alpha$ we demonstrate that the number of random edges required `jumps' at regular intervals, and within these intervals  our
 result is best-possible. This work therefore closes the gap between the seminal work of Johansson, Kahn and Vu~\cite{jkv} (which resolves the purely random case, i.e., $\alpha =0$) and that of Hajnal and Szemer\'edi~\cite{hs} (showing that when $\alpha \geq 1-1/r$ the initial graph already houses the desired perfect $K_r$-tiling).
\end{abstract}
\section{Introduction}

A significant facet of both extremal graph theory and random graph theory is the study of embeddings. In the setting of random graphs, one is interested in the threshold for the property that $G(n,p)$ \emph{asymptotically almost surely} (a.a.s.) contains a fixed 
(spanning) subgraph $F$. Meanwhile, a classical line of inquiry in extremal graph theory is to determine the minimum degree threshold that ensures a graph $G$ contains a fixed (spanning) subgraph $F$.
A much studied problem in both the extremal and random settings concerns the case when $F$ is a so-called \emph{perfect $H$-tiling}. In this paper we bridge the gap between the random and extremal models for the problem of perfect clique tilings. 

\subsection{Perfect tilings in graphs} 
Given two graphs $H$  and $G$, an \emph{$H$-tiling} in $G$ 
is a collection of vertex-disjoint copies of $H$ in $G$. An
$H$-tiling is called \emph{perfect} if it covers all the vertices of $G$.
Perfect $H$-tilings are also referred to as \emph{$H$-factors}, \emph{$H$-matchings} or \emph{perfect $H$-packings}. 
Note that  a perfect $H$-tiling is a generalisation of the notion of a perfect matching; indeed, perfect matchings correspond to the case when $H$ is a single edge.

One of the cornerstone results in extremal graph theory is the  Hajnal--Szemer\'edi theorem~\cite{hs} which determines the minimum degree threshold that ensures a graph
contains a perfect $K_r$-tiling. 
\begin{thm}[Hajnal and Szemer\'edi~\cite{hs}]\label{hs}
Every graph $G$ whose order $n$
is divisible by $r$ and whose minimum degree satisfies $\delta (G) \geq (1-1/r)n$ contains a perfect $K_r$-tiling. 
Moreover, there is an $n$-vertex graph $G$ with $\delta (G) = (1-1/r)n-1$ that does not contain a perfect $K_r$-tiling.
\end{thm}
Earlier, Corr\'adi and Hajnal~\cite{corradi} proved Theorem~\ref{hs} in the
case when $r=3$. See~\cite{short} for a short proof of Theorem~\ref{hs}. 

Since the proof of Theorem~\ref{hs} there have been many generalisations obtained in several directions. For example,
 K\"uhn and Osthus~\cite{kuhn2}  characterised, up to an additive constant, the minimum degree which ensures that a graph $G$ 
contains a perfect $H$-tiling for an \emph{arbitrary} graph $H$.
Keevash and Mycroft~\cite{my} proved an analogue of the Hajnal--Szemer\'edi theorem in the setting of \emph{$r$-partite graphs}, whilst there 
are now several generalisations of Theorem~\ref{hs} in the setting of \emph{directed graphs} (see e.g.~\cites{cdkm, treg, forum}).
See~\cite{survey} for a survey including many of the results on graph tiling. 
There has also been significant interest in tiling problems in hypergraphs; the survey of Zhao~\cite{zsurvey} describes many of the results in the area.

\subsection{Perfect tilings in random graphs} 
Recall that the  random graph $G(n,p)$ consists of vertex set $[n]:=\{1,\dots,n\}$ where each edge is present with probability $p$, independently of all other choices.
In the early 1990s, the problem of determining the threshold for the property that $G(n,p)$ contains a perfect $H$-tiling attracted the attention of Erd\H{o}s (see the appendix of~\cite{prob1}). 
Indeed, as well as raising the general problem,
Erd\H{o}s particularly focused on the case when $H=K_3$ stating that `the correct answer will be probably about $n^{4/3}$ edges', though cautioned that `the lack of analogs to Tutte's theorem may cause serious trouble'.
This caution turned out to be well-founded as 
for a number of years even the case of triangles  remained quite stubborn, 
despite some  partial results towards it~\cite{kim, k1}. However, in 2008, spectacular work of Johannson, Kahn and Vu~\cite{jkv} not only resolved the problem for perfect $K_3$-tilings, but the general problem of perfect $H$-tilings for all 
so-called \emph{strictly balanced} graphs $H$. Below we state their result only in the case of perfect clique tilings.
\begin{thm}[Johansson, Kahn and Vu~\cite{jkv}]\label{thmjkv}
Let  $n\in \mathbb N$ be divisible by $r \in \mathbb N$ where $r\geq 3$.
\begin{itemize}
\item If $p =\omega( n^{-2/r}(\log n)^{2/(r^2-r)})$ then a.a.s.~$G(n,p)$ contains a perfect $K_r$-tiling.
\item If $p =o( n^{-2/r}(\log n)^{2/(r^2-r)})$ then a.a.s.~$G(n,p)$ does not contain a perfect $K_r$-tiling.
\end{itemize}
\end{thm}

\subsection{The model of randomly perturbed graphs}
Bohman, Frieze and Martin~\cite{bfm1} introduced a model which provides a connection between the extremal and random graph settings.
In their model one starts with a dense graph and then 
adds $m$ random edges to it. A natural problem in this setting is to determine \emph{how many} random edges are required to ensure that the resulting graph  a.a.s.~contains a given graph $F$ as a spanning subgraph.
For example, the main result in~\cite{bfm1} states that for every $\alpha>0$, there is a $c=c(\alpha)$ such that if we start with an arbitrary $n$-vertex graph $G$ of minimum degree $\delta(G)\geq \alpha n$ and
add $cn$ random edges to it, then a.a.s.~the resulting graph is Hamiltonian. 
This result characterises how many random edges we require for \emph{every} fixed $\alpha>0$. Indeed, if $\alpha \geq 1/2$ then Dirac's theorem implies that we do not require any random edges; that is any $n$-vertex graph $G$ of minimum degree $\delta(G)\geq \alpha n$
is already Hamiltonian. On the other hand, if $0<\alpha <1/2$ then the following example implies that we indeed require a linear number of random edges: Let $G'$ be the complete bipartite graph  with vertex classes of size $\alpha n,(1-\alpha) n$.
It is easy to see that if one adds fewer than $(1-2\alpha )n$ (random) edges to $G'$, the resulting graph is \emph{not} Hamiltonian.

In recent years, a range of results  have been obtained concerning embedding spanning subgraphs into a randomly perturbed graph, as well as other  properties of the model; see e.g.~\cite{ bwt2, benn, bfkm, bhkmpp, bmpp2, das, dudek, joos2, kks2, kst,antoniuk2020high,knierim2019k_r}.
The model has also been investigated in the setting of directed graphs and hypergraphs (see e.g.~\cite{bhkm, hanzhao, kks1, mm}). Much of this work has focused on the range where the minimum degree of the deterministic graph is linear but with respect to some arbitrarily small constant $\alpha$. In this range, one thinks of the deterministic graph as `helping' $G(n,p)$ to get a certain spanning structure and the observed phenomenon is usually a decrease in the probability threshold of a logarithmic factor, as is the case for Hamiltonicity as above. Recently, there has been interest in the other extreme, where one starts with a minimum degree slightly less than the extremal minimum degree threshold for a certain spanning structure and requires a small `sprinkling' of random edges to guarantee the existence of the spanning structure in the resulting graph, see e.g. \cite{dudek, nan}. 

Balogh, Treglown and Wagner~\cite{bwt2} considered the perfect $H$-tiling problem in the setting of randomly perturbed graphs. Indeed, for every fixed graph $H$ they determined how many random edges one must add to a graph $G$ of linear minimum degree to ensure that a.a.s. 
 $G \cup G(n,p)$ contains a perfect $H$-tiling. Again we only state their result in the case of perfect clique tilings.
\begin{thm}[Balogh, Treglown and Wagner~\cite{bwt2}]\label{btwthm}
Let  $r \geq 2$.
For every $\alpha >0$, there is a $C=C(\alpha,r)>0$ such that if $p\geq Cn^{-2/r}$ and $(G_n)_{n\in r\mathbb{N}}$ is a sequence of graphs with $|G_n|=n$
and minimum degree $\delta (G_n) \geq \alpha n$
then a.a.s.~$G_n\cup G(n,p)$ contains a perfect $K_r$-tiling.
\end{thm}
 Theorem~\ref{btwthm}, unlike Theorem~\ref{thmjkv}, does not involve a logarithmic term. Thus comparing  the randomly perturbed model with the  random graph model,
we see that starting with a graph of linear minimum degree instead of the empty graph
saves a logarithmic factor in terms of how many random edges one needs to ensure the resulting graph a.a.s.~contains a perfect $K_r$-tiling.
Further,
Theorem~\ref{btwthm} is best-possible in the sense that given any $0<\alpha <1/r$, there is a constant $c=c(\alpha,r)>0$ and sequence of graphs $(G_n)_{n\in r \mathbb N}$ where $G_n$ is an $n$-vertex graph with minimum degree at least $\alpha n$ 
so that a.a.s. $G_n \cup G(n,p)$ does not contain a  perfect $K_r$-tiling  when $p\leq cn^{-2/r} $ (see Section~2.1 in~\cite{bwt2}).
However, as suggested in~\cite{bwt2}, this still leaves open the question of how many random edges one requires if $\alpha >1/r$.

In this paper we give  a sharp answer to this question. Before we can state our result we introduce some notation.
\begin{defin} \label{def:threshold}[Perturbed perfect tiling threshold]
Given some $0 \leq \alpha \leq 1$, and a graph $H$ of order $h$, the \emph{perturbed perfect tiling threshold} $p(H, \alpha)$ satisfies the following.
\begin{itemize}
	\item[(i)] If $p = p(n) =\omega( p(H, \alpha))$, then for any sequence $(G_n)_{n \in h\mathbb{N}}$ of $n$-vertex graphs with $\delta (G_n) \geq \alpha n$, the graph $G_n \cup G(n,p)$ a.a.s. contains a perfect $H$-tiling.
	\item[(ii)] If $p = p(n) =o( p(H, \alpha))$, for some sequence $(G_n)_{n \in h\mathbb{N}}$ of $n$-vertex graphs with $\delta (G_n) \geq \alpha n$, the graph $G_n \cup G(n,p)$ a.a.s. does not contain a perfect $H$-tiling.
\end{itemize}
If it is the case that every sufficiently large $n$-vertex graph of minimum degree at least $\alpha n$ contains a perfect $H$-tiling we define $p(H, \alpha):=0$.  
We say the threshold $ p(H, \alpha)$ is \emph{sharp} if there are constants $C(H,\alpha),D(H,\alpha)>0$ such that (i) remains valid with $p\geq C  p(H, \alpha)$ and (ii) is satisfied for any $p\leq D  p(H,\alpha)$.
\end{defin}
Thus,  Theorem~\ref{hs} implies that $p(K_r, \alpha)=0$ for all $\alpha \geq 1-1/r$ whilst Theorem~\ref{thmjkv} precisely states that $p(K_r, 0)=n^{-2/r}(\log n)^{2/(r^2-r)}$ (actually Theorem~2.3 in \cite{jkv} and Theorem~3.22(ii) in \cite{jlr} imply this threshold is sharp).
Our main result deals with the intermediate cases (i.e. when $0<\alpha <1-1/r$).

\begin{thm} \label{thm:Main}
Let $2\leq k\leq r$ be integers.
%\COMMENT{AT: TO DO: the proof given is only for $k <r$. Can we include $k=r$ case so that it implies Theorem~\ref{btwthm}?}
 Then given any $1-\frac{k}{r} < \alpha < 1-\frac{k-1}{r}$, \[p(K_r, \alpha)=n^{-2/k}.\]
Moreover, the threshold $p(K_r, \alpha)$ is sharp.
%Moreover, there exists $c=c(\gamma,k,r)>0$ and a sequence $(G'_n)_{n\in \mathbb{N}}$ with $v(G'_n)=n$ and $\delta(G'_n)\geq (1-\frac{k}{r}+\gamma)n$, then a.a.s. $G'_n\cup G(n,p)$ does not contain a $K_r$-tiling. 
%covering all but at most $\alpha n$ vertices. 
\end{thm}
%Theorem~\ref{thm:Main} is best-possible in the sense that we cannot lower the value of $p$ here: more precisely, 
% there exists $c=c(\gamma,k,r)>0$ and a sequence $(G'_n)_{n\in \mathbb{N}}$ with $v(G'_n)=n$ and $\delta(G'_n)\geq (1-\frac{k}{r}+\gamma)n$ such that a.a.s. $G'_n\cup G(n,p)$ does not contain a perfect $K_r$-tiling. 
%covering all but at most $\alpha n$ vertices.
%Further, the minimum degree condition is also essentially best possible; indeed, one cannot replace the $(1-\frac{k}{r}+\gamma)$-term in Theorem~\ref{thm:Main} with any term smaller than $1-\frac{k}{r}$.
%We explain this more precisely in Section~\ref{}.

%Note that Theorem~\ref{thm:Main} provides a `bridge' between the Hajnal--Szemer\'edi theorem and the Johansson--Kahn--Vu theorem. Indeed, at one extreme Theorem~\ref{hs} deals with the case when $G_n$ has minimum degree at least $(1-\frac{1}{r})n$ and at the other extreme Theorem~\ref{thmjkv} deals with the case when $G_n$ has minimum degree $0$; Theorem~\ref{thm:Main} considers the intermediate ground.
Thus, Theorem~\ref{thm:Main} provides a bridge between the Hajnal--Szemer\'edi theorem and the Johansson--Kahn--Vu theorem.
Notice that the value of $p(K_r, \alpha)$ demonstrates a `jumping' phenomenon; given a fixed $k$ the value of $p(K_r, \alpha)$ is the \emph{same} for all $\alpha \in ((r-k)/r, (r-k+1)/r)$, however if $\alpha$ is just above this interval the
value of $p(K_r, \alpha)$ is significantly smaller.

Note in the case when $k=r$, Theorem~\ref{thm:Main} is implied by the results from~\cite{bwt2}; whilst finalising the paper we learned of a very recent result~\cite{nan} concerning powers of Hamilton cycles in randomly perturbed graphs which implies the case when $k=2$ and $r$ is even.  %\COMMENT{PM: I wonder if we should dwell on this work a bit more? And maybe mention the paper of Dudek, Reiher, Ruci{\'n}ski and Schacht \cite{dudek} which came before. The reason I say this is because it seems somewhat like this work, looking at a large min degree slightly below the extremal threshold for appearance,  is seen as significantly different from the perturbed model. In particular, in \cite{dudek}, they don't even mention the perturbed model, and instead opt for calling it a `randomly augmented model'. The fact that our paper contains both points of view and in fact everything in between, is maybe something we should emphasise? }. 
To help provide some intuition for Theorem~\ref{thm:Main}, note that $n^{-2/k}$ is the threshold for the property that $G(n,p)$ contains a copy of $K_k$ in every linear sized subset of vertices; this property will be exploited throughout the proof. 
Our proof uses the absorption method, and in particular the novel `absorption reservoir method' introduced by Montgomery \cite{M14a,M19}, where we use a robust sparse bipartite graph, which we call a template, in order to build an absorbing structure in our graph. We also use `reachability' arguments, introduced by Lo and Markst\"om \cite{LM1}, in order to build absorbing structures. We use various probabilistic techniques throughout, such as multi-round exposure, and we appeal to Szemer\'edi's regularity lemma in order to obtain an `almost tiling'. 
%\COMMENT{TO DO: AT: add paragraph describing what tools we use in proof?}

\smallskip

The paper is organised as follows. In the next section we introduce some fundamental tools that will be applied in the proof of Theorem~\ref{thm:Main}. 
Section~\ref{sec:lower} then contains the construction that  provides the lower bound on $p(K_r,\alpha)$ in Theorem~\ref{thm:Main}. In Section~\ref{sec:overview} we give an overview of our proof for the upper bound on $p(K_r,\alpha)$ in Theorem~\ref{thm:Main}, which is given in Section~\ref{sec:proof} after developing the necessary theory in Section \ref{sec:almost} and Section \ref{sec:absorption}. Finally some open problems are raised in the concluding remarks section (Section~\ref{sec:conc}). 

\vspace{2mm}

{\bf Additional Note:} Since the paper was first submitted there have been some related results proven. Indeed, Knierim and Su \cites{knierim2019k_r}, expanding on work of Nenadov and Pehova \cite{nenadov2018ramsey}, considered the so-called Ramsey-Tur\'an problem for clique tilings. They showed that for any $\alpha>1-\frac{2}{r}$, there exists an $\eta>0$ such that if  $G'$ is a graph with $\delta(G')\geq \alpha n$ and  independence number less than $\eta n$, then $G'$ contains a perfect $K_r$-tiling. This implies Theorem \ref{thm:Main} for $k=2$ and all $r$ as if $G$ has minimum degree $\alpha n$ with $\alpha$ as above and $p=\omega (n^{-1})$ then $G(n,p)$ (and hence $G'=G\cup G(n,p)$) has sublinear independence number. In a different direction, Antoniuk, Dudek, Reiher, Ruci\'nski and Schacht~\cite{antoniuk2020high} explored the appearance of powers of Hamilton cycles in randomly perturbed graphs, building on the work of Nenadov and Truji\'c~\cite{nan}. As the $k^{th}$ power of a Hamilton cycle is a supergraph of a perfect  $K_{k+1}$-tiling, their work gives bounds for the existence of clique tilings. They focus solely on $G$ with minimum degree $\alpha n$ where  $\alpha=j/(j+1)+ \eps$ for some $j\in \mathbb{N}$, $j\geq 1$ and they obtain tight results in certain  cases.  The only range where their implied results on clique tilings is tight with regards to the threshold obtained and the minimum degree condition is the case when $k=2$ and $r$ is even in Theorem \ref{thm:Main}, already implied by~\cite{nan} and~\cite{knierim2019k_r} as mentioned above.

%%%%%%%%%%%%%%%
\section{Notation and preliminaries}
%\section{Preliminaries}

%\subsection{Notation}
We use standard graph theory notation throughout. In particular we use $|G|$ to denote the number of vertices of a graph $G$. 
Sometimes we will also write $v_G$ and $e_G$ to denote the number of vertices and edges in $G$ respectively.
We write $N_G(v)$ to denote the neighbourhood of a vertex $v\in G$. 
For a subset of vertices $V'\subseteq V=V(G)$, $G[V']$ denotes the graph induced by $G$ on $V'$ and we use the shorthand $G\setminus V'$ to denote $G[V\setminus V']$. 
If $V'=\{x\}$ we simply write $G \setminus x$.
Further, for disjoint subsets of vertices $V',V''\subseteq V $, $G[V',V'']$ denotes the bipartite graph induced by $G$ on $V'\cup V''$ by  considering only the edges of $G$ with one endpoint in $V'$ and the other endpoint in $V''$. If $G'$ is a graph on the same vertex set as $G$ we write $G\cup G'$ to denote the
graph on vertex set $V(G)$ with edge set $E(G)\cup E(G')$. We write $G-E(G')$ to be the graph obtained from $G$ by deleting any edges that also lie in $G'$. One key exception to the use of standard notation is our use of $\overline{H}$, to denote the complement of $H$ with respect to a graph which is \emph{not} complete, see Definition \ref{def:Hcom}.

We write $K^{r}_{m_1,m_2,\ldots,m_r}$ to denote the complete $r$-partite graph with parts of size $m_1, \ldots, m_r.$ For a graph $J$ on $r$ vertices $\{v_1,\ldots,v_r\}$ and $m_1,\ldots,m_r \in \mathbb{N}$, we define the \emph{blow-up of $J$} to be the graph $J_{m_1,\ldots,m_r}$ with vertex set $P_1\sqcup P_2 \sqcup \ldots \sqcup P_r$, such that $|P_i|=m_i$ and for all $i,j \in [r]$ and $w\in P_i$, $w' \in P_j$ we have $ww'\in E( J_{m_1,\ldots,m_r})$ if and only if $v_iv_j \in E(J)$. Given a set $A$ and $k \in \mathbb N$ we denote by $A^k$ the set of all ordered $k$-tuples  of elements from $A$, while $\binom{A}{k}$ denotes the set of all (unordered) $k$-element subsets  of $A$. %\COMMENT{PM: Added some notational remarks} 
At times we have statements such as the following: ``Choose constants $0\ll c_1 \ll c_2 \ll \ldots \ll c_k$". This should be taken to mean that one can choose constants from right to left so that all the subsequent constraints are satisfied. That is, there exist increasing functions $f_i$ for $i\in [k]$ such that whenever $c_{i}\leq f_{i+1}(c_{i+1})$ for all $i\in[k-1]$, all constraints on these constants that are in the proof, are satisfied.   Finally, we omit the use of floors and ceilings unless it is necessary, so as not to clutter the arguments.

Throughout, we will deal exclusively with \emph{ordered} embeddings of graphs, which we also refer to as \emph{labelled} embeddings. Thus when we refer to an embedding of $H$ in $G$, we implicitly fix an ordering on $V(H)$, say $V(H):=\{h_1,\ldots, h_{v_H}\}$ and say that there is an embedding of $H$ onto an (ordered) vertex set $\{v_1,\ldots,v_{v_H}\}\subseteq V(G)$ if $v_iv_j\in E(G)$ for all $i$ and $j$ such that $h_ih_j\in E(H)$.

In what follows, we introduce the tools that we will use in our proofs. Most of these are well known and so are stated without proof. One  exception is Lemma~\ref{lm:gnpfindingsubgraph}, which is tailored to our purposes and slightly technical 
(but follows from well-known techniques nonetheless).

\subsection{A deterministic tiling result}  \label{sec:det-tiling}
Let $\chi_{cr}(H):=(\chi(H)-1)\frac{|H|}{|H|-\sigma(H)}$ where $\sigma(H)$ is the smallest size of a colour class over all colourings of $H$ with $\chi(H)$ colours.  
The following result of Koml\'os \cite{komlos} is a crucial tool in the proof of Theorem~\ref{thm:Main}. It determines the minimum degree threshold for the property of containing an `almost' perfect $H$-tiling. 

\begin{thm}\label{thm:komlos}
For every graph $H$ and every $\alpha>0$, there exists $n_0$ such that if $G$ is a graph on $n\geq n_0$ vertices with $\delta(G)\geq \left( 1-\frac{1}{\chi_{cr}(H)}\right) n$, then $G$ contains an $H$-tiling which covers all but at most $\alpha n$ vertices of $G$. 
\end{thm}

This was later improved to a constant number of uncovered vertices by Shokoufandeh and Zhao \cite{sz}, but Koml\'os' result suffices for our purposes. 
We will apply Koml\'os'  theorem to 
find an almost perfect $H$-tiling in a reduced graph $R$ of our (deterministic) graph $G$ from Theorem~\ref{thm:Main}; here $H$ will be a carefully chosen auxiliary graph (not $K_r$!). We discuss this further in the proof overview section. 
%However, 
 %it is worth noting already that for our application we do not require the exact tiling result of K\"uhn and Osthus \cite{kuhn2}.

\subsection{Regularity}

We will use the famous regularity lemma due to Szemer\'edi \cite{szemeredi}. The lemma and its consequences appeared in the form we give here, in a survey of Koml\'os and Simonovits \cite{ksregularitysurvey}, which we also recommend for further details on the subject. First we introduce some necessary terminology. Let $G$ be a bipartite graph with bipartition~$\{A,B\}$. For non-empty sets $X\subseteq A$, $Y\subseteq B$, we define the \emph{density} of $G[X,Y]$ to be $d_G(X, Y):=e(G[X,Y])/(|X||Y|)$.
%Let $\ep>0$.
We say that $G$ is \emph{$\ep$-regular} for some $\eps>0$ if for all sets $X \subseteq A$ and $Y \subseteq B$ with $|X|\geq \ep |A|$ and $|Y| \geq \ep |B|$ we have 
\[
|d_G(A,B) - d_G(X,Y)| < \ep.
\]
It is also common, when the underlying graph $G$ is clear, to refer to $(A,B)$ as an \emph{$\eps$-regular pair}. 

%Regularity is a pseudorandom condition. It requires the number of edges that lie between subsets of vertices in the bipartite graph to be roughly what we expect to see in a random graph of the same density, with the parameter $\eps$ controlling the discrepancy from this paradigm. As with many notions of pseudorandomness, this simple definition allows us to conclude many further properties that we would expect from a random graph. The following two lemmas are examples of this. 
We will use the following two well-known results in our proof.
The so-called slicing lemma shows that regularity is hereditary, with slightly weaker parameters.
% in the sense that subgraphs of regular graphs are also regular albeit with slightly weaker parameters. 

\begin{lemma}[Slicing lemma \cite{ksregularitysurvey}*{Fact 1.5}] \label{lem:slice}
Let $G$ be $\eps$-regular on parts $\{A,B\}$ with density $d$ and let $\alpha > \eps$.
Let $A' \subseteq A$ and $B' \subseteq B$ with  $|A'| \ge \alpha |A|$ and $|B'| \ge \alpha |B|$.
Then $G[A', B']$ is $(2\eps/\alpha)$-regular with density at least $d - \eps$.
\end{lemma}

The next lemma is an extremely useful tool, extending the control on the edge count in regular pairs to be able to count the number of embeddings of small subgraphs. 

\begin{lemma}[Counting lemma \cite{ksregularitysurvey}*{Theorem 2.1}] \label{lem:regularitycounting}
Given $d>\eps>0$, $m\in \NN$ and $H$ some fixed graph on $r$ vertices, let $G$ be a graph obtained by replacing every vertex $x_i$ of $H$ with an independent set $V_i$ of size $m$ and every edge of $H$ with an $\eps$-regular pair of density at least $d$ on the corresponding sets. If $\eps \leq \frac{d^r}{(2+r)2^r}=:d_0$, then there are at least $(d_0 m)^r$ embeddings of $H$ in $G$ so that each $x_i$ is embedded into the set $V_i$. %\COMMENT{AT: added this last condition as we really use that the vertices are mapped to different clusters e.g. in the proof of Theorem 5.2}
\end{lemma}

We now turn to the regularity lemma, which 
%allows us to deduce useful structural information about \emph{any} large enough graph. It 
tells us that there is a way to partition \emph{any} large enough graph in such a way that the graph induces $\eps$-regular pairs on almost all of the pairs of parts in this partition.
Actually, we apply a variant of the lemma which ensures that, ignoring a small number of edges and a small exceptional set of vertices,  \emph{all} such pairs are $\eps$-regular. %\COMMENT{AT: added sentence  to avoid confusion}
% Remarkably, for dense\footnote{We say a family of graphs is \emph{dense} if every graph in the family has positive density, in that $e(G)/\binom{v(G)}{2}$ is bounded away from $0$.} \COMMENT{AT: I'd say we don't need this footnote}
 %graphs the lemma asserts that there is such a partition  in which the edges lying in dense $\eps$-regular pairs comprise the majority of the edges in our original graph. 
%This allows one to form strong conclusions about \emph{any} large enough  dense graph by using the properties of $\eps$-regular pairs, such as those given above. 

\begin{lemma}[Degree form of the regularity lemma \cite{ksregularitysurvey}*{Theorem 1.10}]\label{lem:degreeform}
Let $0<\eps<1$ and $m_0 \in \NN$. %\COMMENT{AT: I'd prefer using $\NN$ instead of $\ZZ ^+$ as most people in our field use the former to not include zero}
 Then there is an $N = N(\ep,m_0)$ such that the following holds for every $0\leq d < 1$ and for every graph $G$ on $n \ge N$ vertices. There exists a partition $\{ V_0, V_1, \dots, V_m\}$ of $V(G)$ and a spanning subgraph $G'$ of $G$ satisfying the following:
\begin{enumerate}
	\item $m_0\leq m \leq N$;
	\item $|V_0|\leq  \ep n$ and
$|V_1| = \dots = |V_m|=:n'\leq \ep n$; %\COMMENT{AT: I've separated the conditions as technically we do not have that $V_0$ is smaller than the rest of the $V_i$}
	\item for each $v\in V(G)$, $d_{G'}(v) > d_G(v) - (d + \eps)n$;
	\item  for all 
	%but at most $\eps k^2$ 
	pairs $V_i,V_j$, where $1\leq i<j \leq k$, the graph $G'[V_i, V_j]$ is $\eps$-regular and has density either $0$ or greater than $d$.
	\end{enumerate}
\end{lemma}
The sets $V_1,\dots, V_m$ are called \emph{clusters}, $V_0$ the \emph{exceptional set} and the vertices in $V_0$ are \emph{exceptional vertices}.

The degree condition (\emph{3.}) in Lemma \ref{lem:degreeform} guarantees that the majority of the edges of $G$ lie in $G'$. To make this more transparent it is useful to focus on the dense $\eps$-regular pairs and define the following auxiliary graph. 
The \emph{$(\ep, d)$-reduced graph} $R$ is as follows:
The vertex set of $R$ is the set of clusters $\{V_i : i \in [m]\}$ and
for each $U,U' \in V(R)$, $U U'$ is an edge of $R$ if the subgraph $G'[U,U']$ is $\ep$-regular and has density greater than $d$. The following then follows easily from Lemma \ref{lem:degreeform}.

\begin{cor}

\label{cor:mindeg}
Suppose that $0<\eps \leq  d \leq c$ are constants.
Let $G$ be a graph on $n$ vertices and $\delta(G) \geq cn$.
Suppose that $G$ has a partition $\mathcal{P}=\{V_0, V_1, \dots, V_m\}$ and a subgraph $G'\subseteq G$ as given by Lemma~\ref{lem:degreeform} and corresponding $(\ep, d)$-reduced graph $R$. Then $\delta(R) \geq (c-d-2\ep)m$.
%\COMMENT{AT: changed the condition of $\delta (R)$ here slightly as what was written was not technically correct e.g. if $d=\eps$}
\end{cor}

\subsection{Supersaturation}

The following phenomenon was first noticed by Erd\H{o}s and Simonovits in their seminal paper~\cite{erdHos1983supersaturated}. It states that if there are many copies of a given small subgraph in some host graph, then we can also find many copies of a blow-up in the host graph. It can be proven easily e.g. by induction.

\begin{lemma} \label{lem:supersat}
Let $r, m_1,m_2,\ldots, m_r\in \mathbb{N}$, let $J$ be some graph on $r$ vertices $\{v_1,\ldots,v_r\}$ and $c>0$. Then there exists $c'=c'(r, m_1,m_2,\ldots, m_r,c)>0$ such that the following holds. Suppose $G$ is a graph on $n$ vertices with $n$ sufficiently large such that there are subsets $V_1,\ldots, V_r \subset V(G)$ and $G$ contains at least $cn^r$ labelled copies of $J$ with $v_i\in V_i$ for $i=1,\ldots, r$. Then $G$ contains at least $c'n^{m_1+\ldots+m_r}$ labelled copies of $J_{m_1,m_2,\ldots,m_r}$ with parts $P_1,\ldots, P_r$ such that $P_i\subset V_i$ and $|P_i|=m_i$.
\end{lemma}

\subsection{Subgraph counts in random graphs}

%In this section we introduce some results related to binomial random graphs.
%Similar ones were used in....
%Our main tools are Janson's inequality (see, e.g.,%...~\cite[Theorem 2.14]{JLR}) 
%and Chebyshev's inequality.
%In the following lemma,~\eqref{eq:1} follows from Chebyshev's inequality and~\eqref{eq:2} follows from Janson's inequality.
%Let $I_i\in \be(p_i)$, $i\in \mathcal{I}$, be a finite family of Bernoulli random variables having a dependency graph $L$, i.e., if $ij\in E(L)$ then $I_i$ and $I_j$ are dependent. Let $X=\sum_i I_i$ and $\lambda =\mathbb{E}X=\sum_i p_i$. Moreover, write $i\sim j$ if $ij\in E(L)$, and let $\Delta = \sum_{i\sim j}\mathbb{E}[I_iI_j]$ and $\delta =max_i \sum_{k\sim i}p_k$ and $\overline{\Delta}=\lambda+2\Delta$. 
%Then we have the following estimates on $X$.

%\begin{lemma}
%Let $X$ be a random variable. Then
%$$
%\mathbb{P}(X\ge 2\mathbb{E}[X)] \le \frac{\mathbb{E}[X)+\Delta}{\mathbb{E}[X)^2} .
%$$
%\end{lemma}

%Given a family of random variables $\{I_i\}_{i\in \mathcal{I}}$, a \emph{dependency graph $L$ for $\{I_i\}_{i\in \mathcal{I}}$} is a graph with vertex set $\mathcal I$ and $ij\in E(L)$ if and only if $I_i$ and $I_j$ are dependent. 
%We write $i\sim j$ if $ij\in E(L)$ and define $\Delta = \Delta(\{I_i\}_{i\in \mathcal{I}}) = \sum_{i\sim j}\mathbb{E}[I_iI_j]$.
%Let $X:=\sum_{S\in \cS} I_S$ be the sum of a family of Bernoulli random variables and let $\lambda = \mathbb E(X)$.
%Let $\Delta_{X} := \sum_{S\cap T\neq\emptyset}\mathbb{E}(I_S I_T)$, where the sum is over ordered pairs $S, T\in \cS$.

We first recall Janson's inequality (see e.g. \cite[Theorem 2.14]{jlr}).
Let $\Gamma$ be a finite set and let $\Gamma_p$ be a random subset of $\Gamma$ such that each element of $\Gamma$ is included independently with probability $p$.
Let $\mathcal{S}$ be a family of non-empty subsets of $\Gamma$ and for each $S\in \mathcal{S}$, let $I_S$ be the indicator random variable for the event $S\subseteq \Gamma_p$.
%$I_A = \textbf{1}[A\subseteq \Gamma_p]$.
Thus each $I_S$ is a Bernoulli random variable $\be(p^{|S|})$.
%Given a family of random variables $\{I_i\}_{i\in \mathcal{I}}$, for $i, j\in \mathcal I$, we write $i\sim j$ if and only if $I_i$ and $I_j$ are dependent. 
%Let $\Delta = \sum_{i\sim j}\mathbb{E}[I_iI_j]$, where the sum is over unordered pairs.
Let $X:=\sum_{S\in \cS} I_S$ and $\lambda := \mathbb E(X)$.
Let $\Delta_{X} := \sum_{S\cap T\neq\emptyset}\mathbb{E}(I_S I_T)$, where the sum is over not necessarily distinct ordered pairs $S, T\in \cS$.
Then Janson's inequality states that for any $0\le t\le \lambda$,

\begin{equation}
\mathbb{P}(X\leq \lambda -t)\leq \exp \left ( -\dfrac{t^2}{2\Delta_{X}}\right ). \label{eq:2}
\end{equation}

%Next note that $\var(X)=\mathbb E(X^2) - \mathbb E(X)^2 \le \Delta_X$.
%Then by Chebyshev's inequality,
%\begin{equation}
%\mathbb{P}(X\ge 2\lambda) \le \frac{\var(X)}{\lambda^2} \le \frac{\Delta_X}{\lambda^2}. \label{eq:1}
%\end{equation}

Consider the random graph $G(n, p)$ on an $n$-vertex set $V$.
Note that we can view $G(n, p)$ as $\Gamma_p$ with $\Gamma := \binom{V}2$.
Following \cite{jlr}, for a fixed graph $F$, we define $\Phi_F = \Phi_F(n, p):= \min \{n^{v_H} p^{e_H}: H\subseteq F, e_H>0\}$. This parameter helps to simplify calculations of $\Delta_X$ in the context of counting the number of embeddings of the graph $F$ in $G(n,p)$.  We will also be interested in the appearance of graphs in $G(n,p)$ where we require some subset of vertices to be already fixed in place. Therefore, for a graph $F$, and some independent\footnote{With respect to $F$ i.e. $E(F[W])=\emptyset.$} subset of vertices $W\subset V(F)$, we define \[\Phi_{F,W}=\Phi_{F,W}(n,p):=\min\{n^{v_{H}-v_{H[W]}}p^{e_H}:H\subseteq F, e_H>0\}.\]
Note that $\Phi_F=\Phi_{F,\emptyset}$ and $\Phi_{F\setminus W}\geq \Phi_{F,W}$ for any $F$ and independent set $W\subset V(F)$.   If $W=\{w\}$ for a single vertex $w \in V(F)$, we drop the brackets and simply write $\Phi_{F,w}$ and $\Phi_{F\setminus w}$.
%Also, for a given $F'$ and $W\subset V(F')$, defining $F=F'[V\setminus W]$, we have that $\Phi_F\geq \Phi_{F',W}$ as any subgraph of $F$ can be realised as a subgraph of $F'$.
%and the edges between $W$ and this subgraph can only lower the value that this subgraph gives when calculating the minimum. 
Let us collect some more simple observations concerning $\Phi_F$ and $\Phi_{F,W}$ which will be useful later.

\begin{lemma} \label{PhiFobservations}
The following hold:
\begin{enumerate}
\item Let $C>1$ be some constant, $k\in \NN\setminus\{1\}$ and $p=p(n)\geq C n^{-\frac{2}{k}}$. Let $k'\le k$ and $F_1:=K_{k'}$, then we have that $\Phi_{F_1}\geq Cn$.
\item As above, let  $C>1$ be some constant, $k\in \NN\setminus\{1\}$ and $p=p(n)\geq C n^{-\frac{2}{k}}$. Suppose now that $3 \le k'\le k$ and let $F_2:=K_{k'}^-$ be the complete graph on $k'$ vertices with one edge missing and let  %$W_2:=\{w\}$, where 
$w\in V(F_2)$ be one of the endpoints of the missing edge. Then  $\Phi_{F_2}\geq Cn$ and $\Phi_{F_2,w}\geq \min\{Cn^{1-\frac{2}{k}}, Cn^{\frac{2}{k}}\}\geq Cn^{\frac{1}{k}}$.
\item Let $F_3$, $F_4$ be graphs with vertex subsets $W_3\subset V(F_3)$, $W_4\subset V(F_4)$, let $\Phi_{3}:=\Phi_{F_3,W_3}$ and $\Phi_4:=\Phi_{F_4,W_4}$ and suppose that $\Phi_3,\Phi_4\ge 1$. %such that $\Phi_{F_3,W_3},\Phi_{F_4,W_4} \geq Cn$. 
Let $F_5$ be the graph formed by the union of $F_3$ and $F_4$ meeting in exactly one vertex $x\in (V(F_3)\setminus W_3) \cap (V(F_4)\setminus W_4)$, and let $F_6$ be the graph obtained by taking a disjoint union of $F_3$ and $F_4$. Then letting $W_5:=W_3\sqcup W_4$, we have that $\Phi_{F_5,W_5}\geq \min\{\Phi_{3},\Phi_{4}, \Phi_{3}\Phi_{4}n^{-1}\}$ and $\Phi_{F_6,W_5}= \min\{\Phi_3,\Phi_4\}$. 
\end{enumerate}
\end{lemma}
\begin{proof}
For parts 1 and 2, it suffices to consider the case $k'=k$.
For part 1, we have a simple calculation. 
%It suffices to consider the case $k'=k$. % as $K_k$ contains all possible $F_1$. 
Let $H$ be a subgraph of $K_k$ with $v_H$ vertices and $e_H$ edges. As $v_H\leq k$, we obtain
\[
n^{v_H}p^{e_H}\geq n^{v_H}(Cn^{-\frac{2}{k}})^{\frac{v_H(v_H-1)}{2}}\geq Cn^{v_H-(v_H-1)}= C n.
\]
%minimising $v_H(1-\frac{v_H-1}{k})$ in the range $2\leq v_H \leq k$ at $v_H=k$.

For part 2 first note that as $F_2 \subseteq K_{k} =F_1$ we have that $\Phi _{F_2}\geq \Phi _{F_1}\geq Cn$. %For the second inequality in part 2,
 %it suffices to consider $k'=k$.
 Let $H$ be a subgraph of $K_k^-$. If $w \notin H$, the calculation from part 1 gives that $n^{v_H}p^{e_H}\geq Cn$. So suppose $w\in H$. Now let us distinguish two cases, depending on whether the vertex $u$ is in $H$, where $u$ is the vertex in $K_k^-$ such that $uw$ is a \emph{non}-edge. If $u\in H$, we have that 
\[
n^{v_H-1}p^{e_H}\geq n^{v_H-1}Cn^{-\frac{2}{k}\left(\frac{v_H(v_H-1)}{2}-1\right)}\geq Cn^{v_H-1-(v_H-1-\frac{2}{k})}\geq Cn^{\frac{2}{k}},
\]
%minimising $v_H-\frac{(v_H+2)(v_H-1)}{k}$ in the range $3\leq v_H \leq k$ at $v_H=k$.
again using that $v_H\leq k$.
Likewise, if $u\notin H$, we have that 
\[
n^{v_H-1}p^{e_H}\geq n^{v_H-1}p^{\binom{v_H}{2}}\geq Cn^{(v_H-1)\left(1-\frac{v_H}{k}\right)}\geq Cn^{1-\frac{2}{k}},
\]
where the last inequality follows as
 $(v_H-1)\left(1-\frac{v_H}{k}\right)$ is minimised in the range $2\le v_H\le k-1$ at $v_H=2$ and $v_H=k-1$.
This shows that $\Phi _{F_2,w}$ is bounded as desired.

Part 3 also follows from the definition. Indeed, note that one subgraph $H$ of $F_6$ that is a minimiser of the term in the definition of $\Phi_{F_6,W_5}$ must be a subgraph of $F _3$ or a subgraph of $F _4$. This ensures 
$\Phi_{F_6,W_5}= \min\{\Phi_3,\Phi_4\}$.
Similarly,  one subgraph $H$ of $F_5$ that is a minimiser of the term in the definition of $\Phi_{F_5,W_5}$ must be a subgraph of $F _3$, a subgraph of $F _4$, or a subgraph of $F_5$ that contains $x$. This ensures 
$\Phi_{F_5,W_5}\geq \min\{\Phi_{3},\Phi_{4}, \Phi_{3}\Phi_{4}n^{-1}\}$.
\end{proof}

We now apply Janson's inequality in order to give a general result about embedding  constant sized graphs into $G(n,p)$. 
The following lemma provides the basis for a greedy process in which we find some larger (linear size) graph in $G(n,p)$. 
We will require that the embedding of our larger graph has certain vertices already prescribed and repeated applications of Lemma \ref{lm:gnpfindingsubgraph} will then allow us to embed the remaining vertices of the graph in a greedy manner. 
%This explains the need for the flexibility in the result allowing us to apply it to \emph{any} subset of $s$ (remaining) indices and forbidding \emph{any} small enough set of (previously used) vertices from being used. 
So it is crucial that we can apply the lemma to \emph{any} subset of $s$ (remaining) indices while avoiding \emph{any} small enough set of (previously used) vertices from being used. 

For future applications, we state and prove the following lemma in the context of \emph{$r$-uniform hypergraphs} ($r$-graphs), and the definition of $\Phi_F$ extends naturally to $r$-graphs $F$ and $G^{(r)}(n,p)$.
Recall that $G^{(r)}(n,p)$ is an  $r$-graph on $n$ vertices where each $r$-tuple of vertices forms an edge with probability $p$, independent of all other $r$-tuples.

\begin{lemma} \label{lm:gnpfindingsubgraph}
Let $n,t(n),s(n)\in\NN$,  $0<\beta<1/2$ and 
%$t=t(n)\in \mathbb{N}$ and 
$L, v, w, e, r\in \mathbb{N}$ such that $r\ge 2$,
%$v_i\le v$ for all $i\in [t]$ and 
$L t, sw \leq \frac{\beta n}{4v}$
%\COMMENT{AT: define $\ll$ notation somewhere? In fact here I guess for what we need$\leq$ instead of $\ll$ suffices?!}
 and $\binom{t}{s}\le 2^n$. %such that for $i\in[t]$, $0\le w_i<v_i$.
Let $F_1,\ldots,F_t$ be labelled $r$-graphs with distinguished vertex subsets $W_i\subset V(F_i)$ such that $|W_i|\leq w$, $|V(F_i\setminus W_i)|=v$, $e(F_i)=e$ and $e(F_i[W_i])=0$ for all $i\in[t]$. 
%Further, suppose that $e(F_1)=e(F_2)=\cdots=e(F_t):=e$ and $e(F_i[W_i])=0$ for all $i\in[t]$. 
Now %let $n$ be sufficiently large %be a labelled graph with $b$ vertices and $a$ edges and $C>0$, $s\in \mathbb{N}$ be some constants.
%Suppose $1/n\ll \gamma, 1/C, 1/b, 1/a, 1/s$.
%Write $a:=g(b)$.
let $V$ be an $n$-vertex set 
%with $n$ sufficiently large (in particular $n\geq t$), 
and let $U_1,\ldots, U_t\subset V$ be labelled vertex subsets with $|U_i|=|W_i|$ for all $i\in[t]$. 
%Suppose also that 
Finally, suppose there are families $\cF_1, \dots, \cF_t\subset \binom{V}{v}$ of labelled vertex sets such that for each $i\in [t]$, $|\cF_i|\geq \beta n^v$.

Now suppose that $1\le s(n) \le t(n)$ and $p=p(n)$ are such that  %$\Phi\cdot s \ge 2^{v+4}v!t/\beta^2$ and  
\begin{align}\label{eq1}
 s\cdot\Phi \ge \left(\frac{2^{v+7}v!}{\beta^2}\right)\min\{Lt\log n,n\}\mbox{ and }\Phi' \ge \left(\frac{2^{v+7}v!}{\beta^2}\right)n,
\end{align}
%\COMMENT{AT: replaced 5 with 7 in two places in~(\ref{eq1})}
where $\Phi:=\min\{\Phi_{F_i,W_i}:i\in [t]\}$ and $\Phi':=\min\{\Phi_{F_i\setminus W_i}:i \in [t]\}$ with respect to $p=p(n)$.  
Then, a.a.s., for any $V'\subseteq V$, with $|V'|\geq n-Lt$ and any subset $S\subseteq [t]$ such that $|S|=s$ and $U_i\cap U_j=\emptyset$ for $i\neq j\in[s]$, there exists \emph{some} $i\in S$ such that there is an embedding (which respects labelling) of $F_i$ in $G^{(r)}(n,p)$ on $V$ which maps $W_i$ to $U_i$ and $V(F_i)\setminus W_i$ to a labelled set in $\cF_i$ which lies in $V'$.
\end{lemma}
Note by `labelled' here we mean   that for all $j$, the $j^{th}$ vertex in $W_i$ is mapped to the $j^{th}$ vertex in $U_i$; moreover the $j^{th}$ vertex in $V(F_i)\setminus W_i$ is mapped to the $j^{th}$ vertex in some labelled set from $\cF_i$.

\begin{proof}
%Firstly, let $v=\max_{i\in [t]}v_i$ and note that we can assume $v_i=v$ for each $i\in [t]$. Indeed, we extend every $F_i$ to a $v$-vertex graph by adjoining isolated vertices and then add arbitrary vertices from $V\setminus (X\cup_{i\in[t]}U_i)$ to each $X\in \cF_i$ so that $X$ has size $v$. 
Let us fix $S\subset [t]$ with $|S|=s$ and a vertex subset $V'\subset V$ as in the statement of the lemma. Let $U:=\cup_{i\in S}U_i$ and fix $V'':=V'\setminus U$. Note that $(V\setminus V')\cup U$ intersects at most $\frac{\beta}{2} n^v$ of the elements of $\cF_i$ for each $i$
%. Thus, 
%as $\eta v\ll \beta$, 
 and so we can focus on a subset $\cF'_i$ of each $\cF_i$ of at least $\frac{\beta}{2} n^v$ sets which are all contained in $V''$.
 %disjoint from $V\setminus V'$.
For each $i\in S$  and each labelled subset $ X\in \cF'_i$, %which lies in $V'$,
let $I_{X,i}$ denote the indicator random variable that $X\cup U_i$ hosts a labelled copy of $F_i$ where $W_i$ is mapped to $U_i$.
 %assuming here that $X\in \cF'_i$. 
To ease notation sometimes we write $I_{X}$ instead of $I_{X,i}$.
Note that $Z:=\sum\{I_{X,i}:X\in\cup_{i\in S}\cF'_i\}$ counts the number of suitable embeddings in $G^{(r)}(n,p)$. (So here if $X$ is in $a$ of the collections $\mathcal F'_i$, then there are $a$ indicator random variables in this sum corresponding to $X$.) 
%We aim to bound the probability that $Z< 1$ by proving concentration for the variable $Z$.
%Indeed, an easy calculation  (using the first part of (\ref{eq1})) gives that $\mathbb{E}[Z]> 2$ for large enough $n$ and so it suffices to bound $\mathbb{P}[Z\le \mathbb{E}[Z]/2]$. 

%First we calculate a lower bound on the expected number of suitable embeddings of $F_i$ for some $i\in S$. In order to get this lower bound, it suffices to consider one index $i^*\in S$ as follows

%\begin{align*}\mathbb{E}[Z]&=\sum \{ \mathbb{E}[I_X]:X\in\cup_{i\in S}\cF_i\} \\ &\ge \sum \{\mathbb{E}[I_X]:X\in\cF_{i^*}\}\\ &\ge \frac{\beta}{2} n^{v}p^{e}\\ &\ge \frac{\beta}{2} \Phi_{F_{i^*},W_{i^*}} \\ &\ge 2t/s\ge 2.\end{align*}

%We now turn to showing concentration for the random variable $Z$. 
%If $t\log n=o(n)$, we are done if we can show that $\mathbb{P}[Z\le \mathbb{E}[Z]/2]\le \exp(-2Lt\log n)$, as taking a union bound over the (at most $2^t$) possible sets $S$ and the $\binom{n}{L t}\le \exp(L t(1+\log n))$ possible $V'$, we have that a.a.s., $Z\ge 1$ for all such $S$ and $V'$. 
%If $t\log n=\Omega(n)$, we instead bound the number of possible $V'$ by $2^n$ and show that $\mathbb{P}[Z\le \mathbb{E}[Z]/2]\le \exp(-2n)$. 

An easy calculation  (using the first part of (\ref{eq1})) gives that $\mathbb{E}[Z]> 2$ for large enough $n$.
%\smallskip
We will show that
\begin{equation}\label{eq:deltaZ}
\Delta_Z\le \frac{\mathbb{E}[Z]^2}{16\min\{Lt\log n,n\}} 
%\max\{\frac{1}{Lt\log n},\frac{1}{n}\}
\end{equation}
 and thus by Janson's inequality (\ref{eq:2}),
$\mathbb{P}[Z\le \mathbb{E}[Z]/2]\le \exp(-2\min\{Lt\log n,n\})$. If $Lt\log n\le n$, taking a union bound over the (at most $2^t$) possible sets $S$ and the $\binom{n}{L t}\le \exp(L t(1+\log n))$ possible $V'$, we have that a.a.s., $Z\ge 1$ for all such $S$ and $V'$; if $Lt\log n>n$, we instead bound both the number of $V'$ and the number of $S$ by $2^n$ and draw the same conclusion. 
So in both cases $Z\ge 1$ a.a.s. for all such $S$ and $V'$ and we are done.

%In order to show concentration, we turn to Janson's inequality, as discussed at the beginning of this subsection. 
Now it remains to verify~\eqref{eq:deltaZ}.
Firstly let $Z_i:=\sum_{X\in \cF'_i}I_{X,i}$. Then 
\begin{equation} \label{eqn:expectation squared}
\mathbb{E}[Z]^2=\left(\sum_{i\in S}\mathbb{E}[Z_i]\right)^2=\sum_{i,j\in S}\mathbb{E}[Z_i]\mathbb{E}[Z_j].
\end{equation}
To ease notation, let $\cF:=\cup_{i\in S}\cF'_i$ and for $X,X'\in \cF$, we write $X\sim X'$ if, assuming $X\in \cF_i$, $X'\in \cF_j$, the labelled copies of $F_i$ on $X\cup U_i$ and $F_j$ on $X'\cup U_j$ %\COMMENT{AT:added $U_i$ and $U_j$ here}
intersect in at least one edge. We  split $\Delta_Z$ as follows:

\begin{align}\Delta_Z&=\sum_{\{(X,X')\in\cF^2:X \sim X'\}}\mathbb{E}[I_XI_{X'}] \nonumber
\\ \label{eqn:Delta}&=\sum_{i\in S} \Delta_{Z_i} + \sum_{\{(i,j)\in S^2: \, i\neq j\}}\sum_{\{(X,X')\in\cF'_i\times\cF'_j:X \sim X'\}}\mathbb{E}[I_XI_{X'}],\end{align}
where $\Delta_{Z_i}$ is defined analogously to $\Delta_Z$ for the random variable $Z_i$.

For integers $a$ and $b$, write $(a)_b := a(a-1)\cdots (a-b+1)$.
Fix $1\le k\le v$. There are $\binom{v}{k} (v)_k\le \binom vk v!$ ways that two labelled $v$-sets share exactly $k$ vertices. Fixing two such $v$-sets, there are at most $(n)_{2v - k}\le n^{2v-k}$ ways of mapping their $2v-k$ vertices into  
$V$.
%$V\setminus (\cup_{i\in[t]}U_i)$.
Let $f_k$ denote the maximum number of edges of a $k$-vertex subgraph of $F_i\setminus W_i$, taken over all $i\in [t]$.  As we explain in the next paragraph, we have that for $i\neq j$,  
%\begin{align*}
%\sum_{\{(X,X')\in\cF'_i\times\cF'_j:X \sim X'\}}\mathbb{E}[I_XI_{X'}] &\le \sum_{k=1}^v \binom{v}{k} (v)_k (n)_{2v - k} p^{2e - f_k}\\ &\le \sum_{k=1}^v \binom vk v! \, n^{2v-k} p^{2e - f_k} \\
%&\le \frac{2^{v} v! \, n^{2v}p^{2e}}{\Phi'}\\ &\le \frac{2^{v+2} v!\mathbb{E}[Z_i]\mathbb{E}[Z_j]}{\beta^2\Phi'} .
%\end{align*}
\[
\sum_{\{(X,X')\in\cF'_i\times\cF'_j:X \sim X'\}}\mathbb{E}[I_XI_{X'}] \le \sum_{k=1}^v \binom vk v! \, n^{2v-k} p^{2e - f_k}\le \frac{2^{v} v! \, n^{2v}p^{2e}}{\Phi'}\le \frac{2^{v+2} v!\mathbb{E}[Z_i]\mathbb{E}[Z_j]}{\beta^2\Phi'} .
\]
%\COMMENT{AT: corrected last inequality: $2^{v+2}$}

Here, we crucially used that any copy of $F_i$ on $X\in \cF'_i$ does not have edges intersecting $U_j$ for $j\neq i$. 
Note that the penultimate inequality follows by definition of $\Phi '$. The last inequality follows as $\beta n^v p^e/2 \leq \mathbb{E}[Z_i]$ for all $i \in S$.

Using the above calculation (and the second part of (\ref{eq1})) to compare (\ref{eqn:Delta}) and (\ref{eqn:expectation squared}), we see that the right hand summand of $(\ref{eqn:Delta})$ is less than $\mathbb{E}[Z]^2/ (32 n)$.
 We now estimate the left hand summand of (\ref{eqn:Delta}) in a similar fashion. For a fixed $i\in S$, let $1\le k\le v$. We let $g_k$ denote the maximum number of edges of a subgraph of $F_i$ which has $k$ vertices disjoint from $W_i$.  We have, similarly to before, that
\[
\Delta_{Z_i} \le \sum_{k=1}^v \binom{v}{k} (v)_k (n)_{2v - k} p^{2e - g_k} \le \sum_{k=1}^v \binom vk v! \, n^{2v-k} p^{2e - g_k} 
\le \frac{2^{v} v! \, n^{2v}p^{2e}}{\Phi}.
\]
Thus, 
\[\sum_{i\in S}\Delta_{Z_i}\le \frac{s2^v v!n^{2v}p^{2e}}{\Phi}
\stackrel{(\ref{eq1})}{\leq}
\frac{(s\beta n^{v}p^{e}/2)^2}{32\min\{Lt\log n,n\}}\le \frac{\left(\sum_{i\in S}\mathbb{E}[Z_i]\right)^2}{32\min\{Lt\log n,n\}}=\frac{\mathbb{E}[Z]^2}{32\min\{Lt\log n,n\}}.\]
So bringing both summands together,~\eqref{eq:deltaZ} holds and we are done. 
%$ \Delta_Z\le \frac{\mathbb{E}[Z]^2}{16\min\{Lt\log n,n\}} 
%\max\{\frac{1}{Lt\log n},\frac{1}{n}\}
%$ and thus by Janson's inequality (\ref{eq:2}),
%$\mathbb{P}[Z\le \mathbb{E}[Z]/2]\le \exp(-2\min\{Lt\log n,n\})$. If $Lt\log n\le n$, taking a union bound over the (at most $2^t$) possible sets $S$ and the $\binom{n}{L t}\le \exp(L t(1+\log n))$ possible $V'$, we have that a.a.s., $Z\ge 1$ for all such $S$ and $V'$.
%If $Lt\log n>n$, we instead bound both the number of $V'$ and the number of $S$ by $2^n$ and draw the same conclusion. 
%if $t\log n=o(n)$, we have that $\Delta_Z\le \mathbb{E}[Z]^2/(16Lt\log n)$ which by Jansen's inequality, implies that 
%$\mathbb{P}[Z\le \mathbb{E}[Z]/2]\le \exp(-2Lt\log n)$ as required. 
%Similarly, if $t\log n=\Omega(n)$, we have that $\Delta_Z\le \mathbb{E}[Z]^2/(16n)$ and thus $\mathbb{P}[Z\le \mathbb{E}[Z]/2]\le \exp(-2n)$.
\end{proof}

%This lemma provides the basis for a greedy process in which we find some larger (linear size) graph in $G(n,p)$. We will require that the embedding of our larger graph has certain vertices already prescribed and repeated applications of Lemma \ref{lm:gnpfindingsubgraph} will then allow us to embed the remaining vertices of the graph in many small steps. This explains the need for the flexibility in the result allowing us to apply it to \emph{any} subset of $s$ (remaining) indices and forbidding \emph{any} small enough set of (previously used) vertices from being used. 

In its full generality, Lemma \ref{lm:gnpfindingsubgraph} will be a valuable tool in our proof. However, we will also have instances where we do not need to use the full power of the lemma. 
For instance, setting $r=2$, $s=1$ and $W_i=U_i=\emptyset$ for all $i\in[t]$, we recover a more standard application of Janson's inequality to subgraph containment which we state below for convenience.

\begin{cor} \label{cor:simplegraphcontainment}
%\COMMENT{Made this about labelled embeddings}
Let $\beta>0$, $1\leq t\leq2^n$ and $F$ some fixed labelled graph on $v$  vertices. Then there exists $C>0$ such that the following holds. If $V$ is a set of $n$ vertices, $\cF_1,\ldots, \cF_t\subset\binom{V}{v}$ are families of  labelled subsets such that $|\cF_i|\geq \beta n^v$ and $p=p(n)$ is such that $\Phi_F\geq Cn$, then a.a.s., for each $i\in [t]$, there is an embedding  of $F$ onto a set in $\cF_i$, which respects labellings.
\end{cor}

\section{Lower bound construction for  the proof of Theorem~\ref{thm:Main}}\label{sec:lower}
In this section we give a construction that provides the lower bound in the  proof of Theorem~\ref{thm:Main}. Our construction is a generalisation of that used for the lower bound in Theorem~\ref{btwthm}
(see Section~2.1 of~\cite{bwt2}). We will make use of the following result.

\begin{thm}[\cite{jlr}, part of Theorem $4.9$]\label{jansonthm3}
For every $k \geq 2$ and for every $0<\eps<1$ there is a positive constant $c=c(r,\eps)$ such that if $p\leq cn^{-2/k}$,
\[\lim_{n\rightarrow \infty} \mathbb{P}(G({n,p}) \ \text{contains a $K_k$-tiling of size $\eps n$})=0.\]
\end{thm}

Let $k$ and $r$ be in the statement of Theorem~\ref{thm:Main}.
Consider any $1-\frac{k}{r} <\alpha <1-\frac{k-1}{r}$ and let $\gamma >0$ such that $(1-\gamma)(1-\frac{k-1}{r}) = \alpha$. 
Let $n\in \mathbb N$ be divisible by $r$.
Suppose $G$ is an $n$-vertex graph with vertex classes $A$ and $B$ such that $|B|=(1-\gamma)(1- \frac{k-1}{r})n$ and $|A|=n-|B|$, where there are all possible edges in $G$ except that $A$ is an independent set.
%(Note here for the sizes of $A$ and $B$ we ignore floors and ceilings as it does not affect the argument.)
So $\delta (G) \geq \alpha n$. 

Choose $c'=c'(\gamma, k,r)=c'(\alpha, k,r)$ sufficiently small so that, if $p=c'n^{-2/k}$ a.a.s. $G(n,p)[A] \cong G(|A|,p)$ does not contain a 
$K_k$-tiling of size $\gamma |A|/r\leq \gamma n/r$. The existence of $c'$ is guaranteed by Theorem~\ref{jansonthm3} since $A$ has size linear in $n$.

Observe that any copy of $K_r$ in $G \cup G(n,p)$ either contains a $K_k$ in $A$, or uses at least $r-(k-1)$ vertices in $B$.
Thus, a.a.s., the largest $K_r$-tiling in $G \cup G(n,p)$ has size less than $|B|/(r-k+1)+ \gamma n/r = n/r$ and we are done.
%We claim that a.a.s. $G \cup G(n,p)$ does not contain a perfect $K_r$-tiling. Indeed, in any perfect $K_r$-tiling in $G\cup G(n,p)$ one would require that at least $\gamma n/r$ of the copies of $K_r$ each contain at least $k$ vertices in $A$; this observation is immediate from the size of $A$.
%Thus, we would require a $K_k$-tiling in $G(n,p)[A]$ that covers at least $\gamma k n/r$ vertices. Since a.a.s. no such tiling exists we have shown that a.a.s. $G \cup G(n,p)$ does not that contain a perfect $K_r$-tiling, as desired.

%%%%%%%%%%%%%%%%%%%
\section{Overview of the proof of the upper bound of Theorem~\ref{thm:Main}} \label{sec:overview}
%\COMMENT{AT: I would say it is definitely preferable to prove Theorem~\ref{thm:Main} in full generality i.e. to include the case when $k=r$. That way our result implies Theorem~\ref{btwthm}. So we would need to extends subresults, definitions etc. For example,
%in the case $k=r$, $H_0$ becomes an independent set of size $r$}
In this section we sketch some of the ideas in the remainder of our proof of Theorem~\ref{thm:Main}.
We use the by now well-known absorbing method, which reduces the problem into finding a small absorbing structure on some vertex subset $A$ and finding a $K_r$-tiling that leaves a set $U$ of $o(n)$ vertices uncovered.
The property of the absorbing structure on $A$ is that for any small set $U$ with $|U|\in r\mathbb N$, one can find a perfect $K_r$-tiling in $(G\cup G(n,p))[A\cup U]$, which will finish the proof.

Let $p\geq Cn^{-2/k}$ and let $G$ be an $n$-vertex graph with $\delta(G)\geq (1-\frac{k}{r}+\gamma)n$.
Note that it might be true that both $G$ and $G(n,p)$ are $K_r$-free (a.a.s. for $G(n,p)$).
Thus, to build even a single copy of $K_r$, we may have to use both deterministic edges (from $G$) and random edges (from $G(n,p)$).
We will use the following partition of the edge set of $K_r$.

\begin{definition} \label{Hzerominus}
For $r \in \mathbb{N}$ and $k \in \mathbb{N}$ such that $2\leq k \leq r$, let $r^{\ast},q \in \mathbb{N}$ be such that $k(r^{\ast}-1)+q=r$ and $0< q\leq k$. Then $H_{\text{det}}=H_{\text{det}}(r,k)$ is the complete $r^*$-partite graph with parts $V_1,V_2, \ldots, V_{r^*}$ such that $|V_1|=|V_2|=\ldots=|V_{r^*-1}|=k$ and $|V_{r^*}|=q$, i.e. $H_{\text{det}}:=K^{r^*}_{k,\ldots,k,q}$. We also define $\overline{H_{\text{det}}}$ to be $K_r-E( H_{\text{det}})$, 
%\COMMENT{AT: changed notation here to $K_r-E( H_{\text{det}})$ to avoid conflict with other uses of $G\setminus V$ notation}
i.e. the complement of $H_{\text{det}}$ on the same vertex set. 
%We also define $H_0=K^{r^*}_{k,\ldots,k,q+1}$, and identify two vertices $w_{01}$ and $w_{02}$ of $H_0$ in the part of size $q+1$.
\end{definition}

Some examples are given in Figure \ref{fig:Hdet}. Note that when $k=r$, $H_{\text{det}}$ is simply an independent set of size $k=r$ and $\overline{H_{\text{det}}}$ is an $r$-clique. The motivation for this partition comes from the following observation. We can build a copy of $K_r$ in $G\cup G(n,p)$ by taking $\Omega(n^r)$ copies of $H_{\text{det}}$ in $G$ and then applying Janson's inequality to conclude that we can `fill up' the independent sets in some copy of $H_{\text{det}}$ by $K_k$s and a $K_q$ and obtain a copy of $K_r$.
With a few more ideas, one can repeatedly apply this naive idea to greedily obtain an almost perfect $K_r$-tiling (see Theorem~\ref{almost tiling thm}).

To build the absorbing set, we use the reachability arguments introduced by Lo and Markstr\"om~\cite{LM1}.
The main part of the reachability arguments rely on the following notion of reachable paths.
Given two vertices $u, v$, a set $P$ of constant size is called a reachable path for $u,v$ if both $P\cup \{u\}$ and $P\cup \{v\}$ contain perfect $K_r$-tilings.
%However, we cannot use arbitrary types of reachable paths because we may not be able to cover the missing 
Then we meet the same problem as above, and thus need to build certain structures by deterministic edges and `fill up the gaps' by random edges.
We need much more involved arguments, including building copies of $K_r$ in a few different ways and making sure that we can recover the missing edges by $G(n, p)$.
Moreover, when $k>r/2$ we cannot prove the reachability between \emph{every} two vertices and have to pursue a weaker property, namely, building a partition of $V(G)$ such that the reachability can be established within each part.

Once we have established the existence of reachable paths, we piece these together to form what we call `absorbing gadgets' (Definition \ref{absorbing gadgets}) and then further combine these absorbing gadgets to define our full absorbing structure in $G\cup G(n,p)$. 
We use an idea of Montgomery \cite{M14a,M19} in order to define our absorbing structure, using an auxiliary `template' to dictate how we interweave our absorbing gadgets, which will ensure that the resulting absorbing structure has a strong absorbing property, in that it can contribute to a $K_r$-tiling in many ways. We will introduce the random edges of $G(n,p)$ only in the last stage, when proving the existence of the full absorbing structure in $G\cup G(n,p)$. Thus, we will first be occupied with finding many reachable paths and absorbing gadgets which use these reachable paths, restricting our attention \emph{only} to the deterministic edges which will contribute to our eventual absorbing structure.

Our analysis splits into three cases depending on the structure of $H_{\text{det}}$ or equivalently, the values of $r$ and $k$. The cases are as follows: 
\begin{enumerate}
\item $H_{\text{det}}$ is balanced i.e. $q=k$ and so ${r}/{k}\in \mathbb{N}$,
\item $\chi (H_{\text{det}})=r^*\geq 3$ and $H_{\text{det}}$ is not balanced i.e. $k<r/2$ and $r/k\notin \mathbb{N}$,
\item $\chi (H_{\text{det}})=r^*=2$ and $H_{\text{det}}$ is not balanced i.e. $r/2<k<r$.
\end{enumerate}
Examples of each case can be seen in Figure \ref{fig:Hdet}. 

%\COMMENT{AT: again should extend definition~\ref{Hzerominus} to $k=r$ case. PM: done}

%\COMMENT{AT: in this overview explain where $H^-_0$ is coming in}
%\COMMENT{TO DO: add an overview of proof}

\begin{figure}   
    \centering
  \includegraphics[scale=1.0]{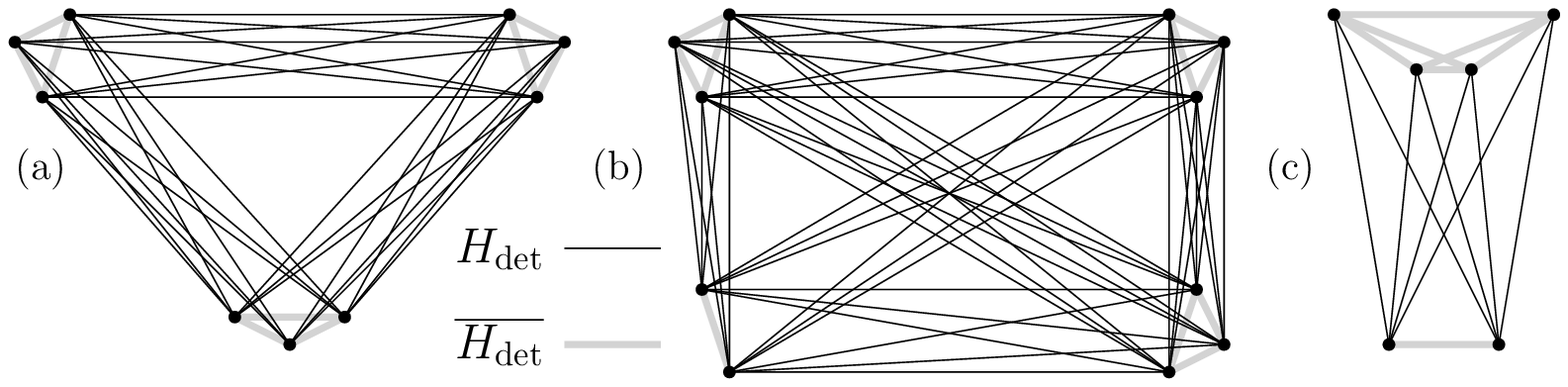}
  \captionsetup{singlelinecheck=off}
    \caption[bla]{%  
    \label{fig:Hdet}
    Some examples of the graphs $H_{\text{det}}$ and $\overline{H_{\text{det}}}$ from Definition \ref{Hzerominus} for each case: 
    \begin{enumerate}
  \item    (a) $r=9,k=q=r^*=3$,  
   \item    (b) $r=11, k=3, q=2, r^*=4,$ 
 \item      (c) $r=6, k=4, q=r^*=2$. 
 \end{enumerate}
 }
 \end{figure}

%%%%%%%%%

%%%%%%%%%%%

\section{An almost perfect tiling} \label{sec:almost}

%Let $2\leq k\leq r-1$ be integers and let $H_{\text{det}}$ be the complete $\ceil{r/k}$-partite graph on $r$ vertices having exactly $\floor{r/k}$ parts of size $k$ and $\ceil{r/k}-\floor{r/k}$ (i.e. none or one) parts of size $r-k\floor{r/k}$ (i.e. the remainder when $r$ is divided by $k$). 
%We begin by showing that in the setting of Theorem \ref{thm:Upper},  $G\cup G(n,p)$ a.a.s. has a $K_r$-tiling covering almost all of the vertices. 
In this section we study almost perfect tilings and prove Theorem~\ref{almost tiling thm} below.
As is the case throughout, in this almost perfect tiling, the edges of $G$ which contribute to the copies of $K_r$ will be copies of $H_{\text{det}}$ as defined in Definition \ref{Hzerominus}. We will rely on $G(n,p)$ to then `fill in the gaps', providing the missing edges i.e. $\overline{H_{\text{det}}}$, to guarantee that each copy of $H_{\text{det}}$ is in fact part of a copy of $K_r$
in $G\cup G(n,p)$. Note that $\chi(H_{\text{det}})=\ceil{r/k}$ and recall the definition of $\chi_{cr}$ discussed in Section \ref{sec:det-tiling}.  When $k$ divides $r$, we have $\chi_{cr}(H_{\text{det}})=r/k=\chi(H_{\text{det}})$ and when $k$ does not divide $r$, we have $\chi_{cr}(H_{\text{det}})=\floor{r/k}\frac{r}{r-(r-k\floor{r/k})}=r/k.$

Thus, the almost perfect tiling result of Koml\'os, Theorem \ref{thm:komlos}, guarantees the existence of an $H_{\text{det}}$-tiling in $G$ which covers almost all the vertices. However, given such a tiling we cannot guarantee that the correct edges appear in $G(n,p)$ in order to 
extend each copy of $H_{\text{det}}$ in the tiling to a copy of $K_r$.
%Instead, we need to show that there are \emph{many} $H_{\text{det}}$-tilings in $G$ and further, that we can greedily build an $H_{\text{det}}$-tiling in such a way that we have many options for the next copy in our tiling.
We aim instead to greedily build a $K_r$-tiling and guarantee that at each step there are $\Omega(n^r)$ copies of $H_{\text{det}}$.
%This will guarantee that one of these options will form a copy of $K_r$ when random edges are introduced and thus our $K_r$-tiling can be greedily formed also. 
%In order to find these copies of $H_{\text{det}}$, we instead apply Theorem \ref{thm:komlos} to the reduced graph of $G$. The pseudorandom condition of regularity then allows us to find many well distributed copies of $H_{\text{det}}$ as necessary. Precisely we show the following.
To achieve this, we use the regularity lemma and apply Theorem \ref{thm:komlos} to the reduced graph of $G$.
Then by the counting lemma, each copy of $H_{\text{det}}$ in the reduced graph will provide many copies of $H_{\text{det}}$ in $G$.

\begin{thm} \label{almost tiling thm}
Let $2\leq k\leq r$ and $\alpha, \gamma>0$.  Then there exists $C=C(\alpha,\gamma,r,k)>0$ %\COMMENT{AT: added $k,r$ into the definition of $C$ as if you follow the `trail' of constants certainly it depends on $r$}
%$0<\frac{1}{n_0}\ll \eta\ll \ep\ll \alpha$ and 
such that if $p\geq Cn^{-2/k}$  and $G$ is an $n$-vertex graph  with $\delta(G)\geq (1-\frac{k}{r}+\gamma)n$, then $G\cup G(n,p)$ a.a.s. contains a $K_r$-tiling covering all but at most $\alpha n$ vertices. 
\end{thm}

\begin{proof}
Apply Lemma~\ref{lem:degreeform} to $G$ with $0< \eps\ll d\ll \gamma/4,\alpha/4,1/r$ and $m_0\in \NN$ large, such that $d_0:=\frac{(d-\eps)^r}{(2+r)2^r}\geq 4\eps/\alpha$. 
We may assume that $n$ is sufficiently large.
Note that by Corollary~\ref{cor:mindeg}, the resulting $(\ep, d)$-reduced graph $R$ has $m\geq m_0$ vertices 
%where $ \ep' \leq n'\leq \ep n$ 
and satisfies $\delta(R)\geq (1-\frac{k}{r}+\gamma/2)m$. Let the size of the clusters in the regularity partition be $n'$ and note that $n/N(\eps,m_0)\leq n'\leq \ep n$. 
Now by Theorem~\ref{thm:komlos}, as $m\geq m_0$ is sufficiently large, there exists an $H_{\text{det}}$-tiling $\mathcal H$ covering all but at most $\alpha m/4$ vertices of $R$.
%Let $\cU_1, \dots, \cU_{t}\in \binom{V(R)}{r}$  such that each $r$-set $\cU_j$ of clusters of $R$ spans a unique copy of
 %$H_{\text{det}}$ in $\mathcal H$.  
Let $\cU_1, \dots, \cU_{t}\in \binom{V(R)}{r}$  such that the $\cU_j$ span disjoint copies of
 $H_{\text{det}}$ in $\mathcal H$.  

Next, let $\cF$ be the collection of subsets $W\subseteq V(G)$ such that there exists some $j\in [t]$ 
%of the form $\bigcup_{U\in \cU_i}U$ 
for which $W$ intersects each $U\in \cU_j$ in at least $\alpha n'/2$ elements (and $W$ contains no vertices of clusters from outside of $\cU_j$).
Here we say that $\cU_j$ \emph{corresponds} to $W$.
%, for some $i\in [t]$. 
Moreover, we call a copy of $K_r$ in $W$ \emph{crossing} if it contains precisely one vertex from each cluster in the class $\cU_j$.

We claim that a.a.s., \emph{every} $W$ in $\cF$ contains a crossing copy of $K_r$ in $G\cup G(n,p)$. 
%so that this copy of $K_r$ contains precisely one vertex from each cluster in the class $\cU_j$ that corresponds to $W$.
  Indeed, fix some $W\in \cF$ and suppose $\cU_j$ corresponds to $W$. Then there are subsets $W_1,\ldots, W_r\subset V(G)$ and  clusters $\{U_1,\ldots, U_r\}= \cU_j$ such that $W_i\subseteq U_i$, 
\[|W_i|= \frac{\alpha n'}{2}\geq \frac{\alpha n}{2N(\eps,m_0)},\]
$\cup_{i\in{r}}W_i\subseteq W$ and $U_1,U_2,\ldots,U_r$ form a copy $H$ of $H_{\text{det}}$ in $R$. 
By Lemma~\ref{lem:slice}, for every $U_iU_j\in E(H)$, we have that $G[W_i,W_j]$ is a $(4\eps/\alpha)$-regular pair with density at least $d-\eps$. 
Thus by Lemma~\ref{lem:regularitycounting} $G[W]$ contains at least $(d_0\alpha n'/2)^r=\Omega(n^r)$ copies of $H_{\text{det}}$ where in each such copy of $H_{\text{det}}$ precisely one vertex lies in each of $W_1,\dots, W_r$; let $\mathcal C_W$ denote this collection. 
Now noting that $F:=\overline{H_{\text{det}}}=K_r- E(H_{\text{det}})$ is a collection of disjoint cliques of size at most $k$,
Lemma~\ref{PhiFobservations} (part 1 and 3) implies that
 $\Phi_F\geq Cn$. Also, we have that $|\mathcal F| \leq 2^n$.
Thus for $C>0$ sufficiently large, Corollary~\ref{cor:simplegraphcontainment} %\COMMENT{AT: note the corollary doesn't talk about labelled copies of graphs being mapped onto labelled sets of vertices, so perhaps we should tweak the statement of the corollary slightly? PM: done}
gives that for every $W \in \mathcal F$ there is a copy of $H_{\text{det}}$ from $\mathcal C_W$ which hosts a labelled copy of $\overline{H_{\text{det}}}$ in $G(n,p)$;  thus the claim is satisfied. 

\smallskip

One can now use the claim to greedily build the almost perfect $K_r$-tiling in $G\cup G(n,p)$. Indeed, initially set $\mathcal K:=\emptyset$. At each step we will add a copy of $K_r$ to $\mathcal K$ whilst ensuring $\mathcal K$ is a $K_r$-tiling
in $G\cup G(n,p)$. Further, at every step we only add a copy $K$ of $K_r$ if there is some $j \in [t]$ such that each vertex in $K$ lies in a different cluster in $\mathcal U_j$ (recall each $\mathcal U_j$ consists of $r$ clusters).

Suppose we are at a given step in this process such that there exists some cluster $U_i \in \mathcal U_j$ (for some $j$)  that still has at least $\alpha n'/2$ vertices uncovered by $\mathcal K$. 
This in fact implies that every cluster in $\cU_j$ contains at least $\alpha n'/2$ vertices uncovered by $\mathcal K$; these uncovered vertices correspond precisely to a set $W \in \mathcal F$. Hence by the above claim there is a crossing copy of $K_r$ 
in $(G\cup G(n,p))[W]$.
% so that this copy of $K_r$ contains precisely one vertex from each cluster in the class $\cU_j$. 
Add this to $\mathcal K$. 
%Clearly $\mathcal K$ is a $K_r$-tiling in $G\cup G(n,p)$.
Thus, we can repeat this process, increasing the size of $\mathcal K$ at every step, until we find that for every $j \in [t]$, all the clusters in $\cU_j$ have at least $(1-\alpha /2)n'$ vertices covered by $\mathcal K$.

%Note that by Lemma \ref{lem:slice}, every element of $\cF$ is a $2\ep/\alpha$-regular pair of density at least $d-\ep$.  Also note that $|\cF|\leq t2^{rm}\leq 2^n$.  Now consider $G\cup G(n,p)$ where $p\geq Cn^{-2/k}$.  By Lemma \ref{lem:JansonCount} and the union bound, the probability that every element of $\cF$ contains a copy of $K_r$ in $G\cup G(n,p)$ is at least $1-2^ne^{-2n}\geq 1-1/2^n$.  
That is, a.a.s. there is a $K_r$-tiling in $G \cup G(n,p)$ covering all but at most
\[
(\alpha n'/2 \times m)+ (\alpha m/4 \times n')+|V_0|\leq \alpha n
\]
vertices, as desired.
Note that the first term in the above expression comes from the vertices in clusters from the classes $\cU_j$; the second term comes from those vertices in clusters that were uncovered by $\mathcal H$.
\end{proof}

%\COMMENT{PM:added}Note that in the case where $k=r$, one can simplify the above proof significantly. This is because the copies of $K_r$ that we look for are completely provided by $G(n,p)$. Indeed, as shown in \cite[Thorem 4.9]{jlr}, for any $\alpha>0$, there exists $C=C(\alpha,r)>0$ such that if $p\geq Cn^{-2/r}$, then $G(n,p)$ a.a.s. contains a $K_r$-tiling covering all but $\alpha n$ vertices.
%\COMMENT{PM:added. JH: simplified}
Note that  one can in fact establish  the case $k=r$ in a  much simpler way because the copies of $K_r$ that we look for can be completely provided by $G(n,p)$, see e.g.~\cite[Thorem 4.9]{jlr}.
%\COMMENT{AT: note that there were important details missing from the above proof, so I have added to it quite a lot.}

\section{The absorption} \label{sec:absorption}
The aim of this section is to prove the existence of an \emph{absorbing structure} $\cA$ in $G':=G\cup G(n,p)$. % under the conditions of Theorem \ref{thm:Upper}. 
The main outcomes are Corollaries~\ref{cor:case1+2absorbingstructure}, ~\ref{cor:case3absorbingstructure} and~\ref{cor:embeddingK_rs}, which will be used in next section to prove our main result.

The key component of the absorbing structure will be some absorbing subgraph $F\subset G'$. We will define $F$ so that it can contribute to a $K_r$-tiling in many ways. In fact we will define $F$ so that if we remove $F$ from $G'$ and we tile almost all of what remains (Theorem \ref{almost tiling thm}), then no matter which small set of vertices remains, the properties of $F$ allow us to complete this tiling to a full tiling of $G'$. 
%In order to describe our absorbing structure $\cA$ we introduce some definitions that will eventually culminate in the structure we will be interested in. 
There are some complications, and the absorbing structure will have 
%To show the existence of this structure requires 
different features depending on the exact values of minimum degree and the size of the cliques that we look to tile with. 
%However the key ideas involved in this section apply to all cases and so we will treat them simultaneously. 
%\COMMENT{JH: Commented out a sentence, which I found unnecessary}

Our absorbing subgraph will be comprised of two sets of edges, namely the deterministic edges in $G$ and the random edges in $G(n,p)$. Initially, we will be concerned with finding (parts of) the appropriate subgraph in $G$ (Section 6.1). In fact, we will need to prove the existence of  many copies of the  deterministic subgraphs we want, as we will rely on there being enough of these to guarantee that one of them will match up with random edges in $G(n,p)$ (Section 6.2) to give the desired subgraph. Therefore it is useful throughout to consider, with foresight, the random edges that we will be looking for to complete our desired structure, as this also motivates the form of our deterministic subgraphs. 

%%%%%%%%%%%
\subsection{The absorbing structure - deterministic edges}

The smallest building block in our absorbing graph will be $K_{r+1}^-$, the complete graph on $r+1$ vertices with one edge missing, say between $w_1$ and $w_2$. This is useful for the simple reason that it can contribute to a $K_r$-tiling in two ways, namely $K_{r+1}^-\setminus\{w_i\}$ for $i=1,2$. We introduce the following notation to keep track of the partition of the edges between the deterministic graph and the random graph. 

\begin{definition} \label{def:H}
%Let $K_{r+1}^-$ be the complete graph with one edge missing between two vertices $w_1,w_2 \in V(K_{r+1})$. 
Suppose $t,r,r_1,r_2,\ldots,r_t\in \NN$ such that $\sum_{i=1}^tr_i=r+1$. We use the notation 
\[H:=(K_{r_1,r_2,\ldots,r_t}^t,i,j),\] 
for not necessarily distinct $i,j\in[t]$, to denote the $(r+1)$-vertex graph  $K_{r_1,r_2,\ldots,r_t}^t$ with two distinct distinguished  vertices: $w_1$ in the $i^{th}$ part (which has size $r_i$) and $w_2$ in the $j^{th}$ part (which has size $r_j$). 
\end{definition}
\begin{definition} \label{def:Hcom}
%Let $K_{r+1}^-$ be the complete graph with one edge missing between two vertices $w_1,w_2 \in V(K_{r+1})$. 
Let $r \in \mathbb N$ and consider an $(r+1)$-vertex graph $F$ with two distinguished  vertices $w_1$ and $w_2$. (Typically we will take $F=H$ as in Definition~\ref{def:H}.)
We then write\footnote{Note that our use of the notation $\overline{F}$ is non-standard here.} $\overline{F}$ to denote the graph on the same vertex set $V(K_{r+1}^-)=V(F)$ such that $E(\overline{F}):=E(K_{r+1}^-)\setminus E(F)$, where we take the non-edge of $K_{r+1}^-$ to be $w_1w_2$. Thus $K_{r+1}^-\subseteq F\cup \overline{F}$. 
\end{definition}
%\COMMENT{AT: I've separated out the definition of $\overline{F}$ and made it more general. Indeed, this is because as stated, the more general version of this definition is used in Definition~\ref{Hzero}}
%This naturally leads to considering three different cases. First let us formally state the objective.     

%\begin{lemma} \label{Absorbing Lemma} Let $r \in \mathbb{N}$ and $k \in \mathbb{N}$ such that $2\leq k \leq r-1$ and $\eta>0$. Suppose as in ... that $G$ is a graph of minimum degree $\delta(G)\geq (\frac{k-1}{r}+\eta)n$ with $|G|=n$ sufficiently large such that $r|n$. Let $p\gg n^{-2/m'}$ where $$m':=\left\{\begin{array}{ll}r-k+2 &\mbox{ if } r\equiv r-k \mod r-k+1; \\ r-k+1 &\mbox{otherwise}.\end{array}\right.$$ Let $G':=G \cup G(n,p)$. Then there exists constants $c_1 \gg c_2$ such that a.a.s. there exists a subgraph $F$ of $G'$ such that $|F|\leq c_1 n$ and $F$ satisfies the following condition. If we denote $V' :=V(G')\setminus V(F)$, we have that for any $K_r$ tiling of $G'[V']$ covering all but at most $c_2n$ vertices there is a perfect $K_r$ tiling of $G'$ which extends this tiling.  
%\end{lemma}

%The idea here is to build $F$ on a small subset of vertices, such that $F$ has a strong absorbing property. This will be characterised by the fact that $F$ can contribute to a $K_r$ tiling in many ways. In fact, $F$ is constructed by considering many smaller structures which also provide this flexibility. The basic building blocks are what we call reachable sets. 

We think of $H,\overline{H}$ and $K_{r+1}^-$ as all lying on the same vertex set throughout with the two distinguished vertices $w_1,w_2$ being defined for all three. The following graph gives the paradigm for how we split the edges of $K_{r+1}^-$ between the deterministic and the random graph.

\begin{definition} \label{Hzero}
For $r \in \mathbb{N}$ and $k \in \mathbb{N}$ such that $2\leq k \leq r$, let $r^{\ast},q \in \mathbb{N}$ be such that $k(r^{\ast}-1)+q=r$ and $0< q\leq k$. 
%Then $H_{\text{det}}=H_{\text{det}}(r,k)$ is the complete multipartite graph with parts $V_1,V_2, \ldots, V_{r^*}$ such that $|V_1|=|V_2|=\ldots=|V_{r^*-1}|=k$ and $|V_{r^*}|=q$, i.e. $H_{\text{det}}=K^{r^*}_{k,\ldots,k,q}$. 
Then $H_0:=(K^{r^*}_{k,\ldots,k,q+1},r^*,r^*)$.
\end{definition}
Some examples of $H_0$ and $\overline{H_0}$ can be seen in Figures \ref{fig:H2} and \ref{fig:H3}. 
Note that if $w_1$ and $w_2$ are the distinguished vertices of $H_0$, then $H_0\setminus{w_i}$ for $i=1,2$ are both copies of the graph $H_{\text{det}}$ from Definition \ref{Hzerominus}. Also note that $\overline{H_0}$ is a disjoint union of $k$-cliques as well as a disjoint copy of $K_{q+1}^-$. Thus, when $q\leq k-1$, it follows from Lemma \ref{PhiFobservations} and Corollary \ref{cor:simplegraphcontainment} that the graph $\overline{H_0}$ is abundant\footnote{Specifically, one can see that any linear sized set in $V(G(n,p))$ contains a copy of $\overline{H_0}$ a.a.s..} in $G(n,p)$ when $p\geq C n^{-2/k}$ for some large enough $C$. Furthermore, as we will see, the minimum degree condition for $G$ along with Lemma \ref{lem:supersat} will imply that there are $\Omega(n^{r+1})$ copies %\footnote{Here many copies means $\Omega(n^{r+1})$. JH: I suggest to put the number in the text and omit this footnote}
 of $H_0$ in $G$. This suggests the suitability of this definition as a candidate for how to partition the edge set of $K_{r+1}^-$ between deterministic and random edges. We remark that the case when $q=k$ is slightly more subtle and we have to adjust our decomposition accordingly. We will discuss this is more detail in the next section.

\subsubsection{Reachability}

%Throughout this section, $\gamma>0$ and $G$ is a fixed graph with $n$ vertices and of minimum degree $\delta(G)\ge (1-\frac{k}{r}+\gamma)n$.
In this subsection, we define reachable paths and show that we can find many of these in our deterministic graph $G$, when the graphs used to define such paths are chosen appropriately. The main results are Proposition \ref{case1reachability}, Proposition~\ref{case2reachability} and Proposition~\ref{case3reachability} which  deal with Case~1, 2 and 3 respectively. 
We first define a reachable path which is a graph which connects together $(r+1)$-vertex graphs as follows.

\begin{definition} \label{reachable paths}
Let $t, r\in \mathbb{N}$ and let $\bm{H}=(H^1,H^2,\ldots,H^{t})$ be a vector of ${(r+1)}$-vertex graphs $H^i$ such that each $H^i$ has two distinguished vertices, $w_{1}^i$ and $w_2^i$. Then an $\bm{H}$-path is the graph $P$ obtained by taking one copy of each $H^i$ and identifying $w_{2}^i$ with $w_{1}^{i+1}$, for $i\in[t-1]$. We call $w_1^1$ and $w_2^t$ the \emph{endpoints} of $P$.

In the case where $H^1=H^2=\ldots=H^t=H$ for some $(r+1)$-vertex graph $H$, we use the notation $\bm{H}=(H,t)$ and thus refer to $(H,t)$-paths.  For $\bm{H}=(H^1,H^2,\ldots,H^{t})$, we also define $\overline{\bm{H}}:=(\overline{H^1},\overline{H^2},\ldots,\overline{H^{t}})$ where $\overline{H^i}$ is as defined in Definition \ref{def:Hcom}. 
\end{definition}

%\COMMENT{AT: implicitly in proofs and when doing calculations, when counting the number of $\bm{H}$-paths in a graph $G$, we count some that appear the same vertex set (i.e. order of vertices matters)... perhaps we should say that explicitly? PM:added} 
We  give some explicit examples of $\bm{H}$-paths  later in Figure \ref{fig:Hpaths}.
In the following, as we look to find  embeddings of $\bm{H}$-paths and larger subgraphs in $G$ and $G(n,p)$, we will always be considering \emph{labelled} embeddings. Therefore, implicitly, when we define graphs such as the $\bm{H}$-paths above, we think of these graphs as having some fixed labelling of their vertices.  
Again, the motivation for the definition of $\bm{H}$-paths comes from considering $K_{r+1}^-$, with vertices $w_1,w_2$ such that $w_1w_2\notin E(K_{r+1}^-)$. 
Indeed, then a $(K_{r+1}^-,t)$-path $P$ has two $K_r$-tilings missing a single vertex; one on the vertices of $V(P)\setminus{w_1^1}$, and one on the vertices of $V(P)\setminus{w_2^t}$. 
Our first step is to find many $\bm{H}$-paths in the deterministic graph $G$, for an appropriately defined $\bm{H}$. In particular, we are interested in the images of the endpoints of the paths.

\begin{definition}
Let $\beta>0$, $t, r\in \mathbb{N}$ and $\bm{H}=(H^1,\ldots,H^t)$ be a vector of $(r+1)$-vertex graphs (each of which is endowed with a tuple of distinguished vertices). We say that two vertices $x,y\in V(G)$
in an $n$-vertex graph $G$ 
 are \emph{$(\bm{H};\beta)$-reachable} (or $(H,t;\beta)$-reachable if $\bm{H}=(H,\ldots,H)=(H,t)$) if there are at least $\beta n^{tr-1}$ distinct labelled embeddings of the ${\bm{H}}$-path $P$ in $G$ such that the endpoints of $P$ are mapped to  $\{x ,y\}$. 

%For $\beta>0$, we say that two vertices $x,y\in V(G')$ are \emph{$({\bm{H}},\beta,t)$-reachable} if there are at least $\beta n^{tr-1}$ distinct\footnote{One reachable set may have several labellings but we are concerned here with the number of distinct sets.} $({\bm{H}},t)$-reachable sets between $x$ and $y$. 
\end{definition}

As discussed before, the graph $H_0$ from Definition \ref{Hzero} will be used to provide deterministic edges for our absorbing structure.
 That is, we look for $(H_0,t)$-paths in $G$ for some appropriate $t$. However, for various reasons there are complications with this approach. 
%as we do not wish to rely on random edges giving copies of $K_{r-k+1}$ incident to fixed vertices if we can avoid this 
Sometimes using a slightly different graph $H$ will allow more vertices to be reachable to each other. Also, as is the case below when $r/k\in\mathbb{N}$, it is possible that $\overline{H_0}$ is not sufficiently common in the random graph $G(n,p)$. Therefore, we have to tweak the graph $H_0$ in order to accommodate these subtleties. This is the reason for using a vector of graphs $\bm{H}$ as we will see.  We will look first at Case 1, when $r/k\in\mathbb{N}$ and so $\overline{H_0}$ contains a copy of $K_{k+1}^-$. This is too dense to appear in the random graph $G(n,p)$ with the frequency that we require and thus we define $H_1$ as in the following proposition.

\begin{figure}
    \centering
  \includegraphics[scale=1.2]{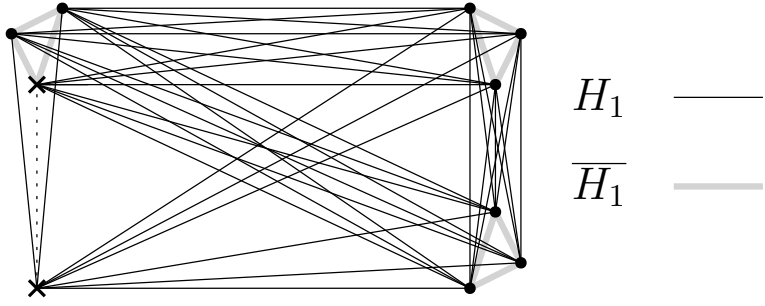}
    \caption{   \label{fig:H1} An example of $H_1$ and $\overline{H_1}$ for Case 1 (a) $r=9,k=q=r^*=3$ (see Definitions \ref{def:H}--\ref{def:Hcom} and Proposition \ref{case1reachability} for the definitions).} 
  \end{figure}

\begin{proposition} \label{case1reachability}
Let $\gamma>0$, $n,r,k\in \mathbb{N}$ such that ${r}/{k}=:r^*\in \mathbb{N}$, $2\le k\le r$ and $n$ is sufficiently large. Let $H_1:=J_1\cap K_{r+1}^-$  where $J_1:=(K^{(r^*+1)}_{k,k,\ldots,k,1},r^*+1,1)$, as defined in Definition \ref{def:H}, and we consider $K_{r+1}^-$ to be on the same vertex set as $J_1$ with a non-edge between the distinguished vertices of $J_1$. Likewise, let  $H_1':=(K^{(r^*+1)}_{k,k,\ldots,k,1},1,r^*+1)\cap K_{r+1}^-$ be the same graph with the labels of the distinguished vertices switched. See Figure \ref{fig:H1} for an example of $H_1$ (and $H'_1$ which is identical). 

%Then $G$ be a fixed graph with $n$ vertices of minimum degree $\delta(G)\ge (1-\frac{k}{r}+\gamma)n$.
 %Consider the graph $H_0$ as defined in \Cref{Hzero} and label the vertices in the part of size $q+1=m+1$ by $\{w_{01},w_{02},v_1,\ldots,v_{q-1}\}$. Now we define $H_1:=H_0\cup\{w_{01}v_j:1\leq j\leq q-1\}$. 
 Then there exists a $\beta_1=\beta_1(r,k,\gamma)>0$ such that for any $n$-vertex graph $G$ of minimum degree $\delta(G)\ge (1-\frac{k}{r}+\gamma)n$, any pair of distinct vertices in $V(G)$ are $(\bm{H_1};\beta_1)$-reachable where $\bm{H_1}:=(H_1,H_1')$.
\end{proposition}
%\COMMENT{AT: rewritten quite a bit of this proposition} 
\begin{proof}
  %We are in the case where $H_{\text{det}}$ is balanced and so of the form $H_{\text{det}}=K^{r^*}_{m,\ldots,m}$ with $r^*$ parts of size $m=r-k+1$. 
%We will use two distinct labelled graphs to give our copies of $H_{\text{det}}$, namely $\overline{H_i^*}=K^{r^*}_{m+1,m,\ldots,m}\cup \{w_iv_j:1\leq j \leq m-1\}$ for $i=1,2$ and denoting by $P:=\{w_1,v_1,\ldots,v_{m-1},w_2\}$, the one part of size $m+1$. 
Let $w_1,w_2$ be the distinguished vertices of $H_1=(K^{(r^*+1)}_{k,k,\ldots,k,1},r^*+1,1)\cap K_{r+1}^-$ as defined in Definition \ref{def:H}.
Fix a pair $x,y \in V(G)$.

We will show that for \emph{any} $z\in V(G)\setminus \{x,y\}$, there are at least $\beta_1'n^{r-1}$ labelled embeddings of $H_1$ which map $w_{1}$ to $x$ and $w_{2}$ to $z$, for some $\beta_1'=\beta' _1(r,k, \gamma)>0$. 
Once we have established this property, this implies the proposition. Indeed, by 
 symmetry, we can also find $\beta_1'n^{r-1}$ embeddings of $H_1$ which map $w_{1}$ to $y$ and $w_{2}$ to $z$. Set $\beta_1:={\beta_1'^2}/{2}$.
Thus there are at least
\[(n-2)\times \beta_1'n^{r-1} \times (\beta_1'n^{r-1} - r^2n^{r-2}) \geq \beta_1 n^{2r-1}\]
distinct embeddings of the ${\bm{H_1}}$-path in $G$ such that the endpoints  are mapped to  $\{x ,y\}$, as desired.
This follows as there are $n-2$ choices for $z$; at least $\beta_1'n^{r-1}$ choices for the copy of $H_1$ containing $x$ and $z$; 
at least $(\beta_1'n^{r-1} - r^2n^{r-2})$ choices for the copy of $H_1'$ containing $z$ and $y$ that are \emph{disjoint} from the choice of $H_1$ (except for the vertex $z$).

\smallskip

So let us fix $x,z \in V(G)$. The proof now follows easily from \Cref{lem:supersat}. As $kr^*=r$, we can express the minimum degree as $\delta(G) \geq (1-\frac{1}{r^*}+\gamma )n$. 
Thus any set of at most $r^*$ vertices has at least $\gamma n$ common neighbours. 
Therefore we have at least $(\gamma n)^{r^*}$ labelled copies $K$ of $K_{r^*}$ where 
$V(K)=\{x_1,\ldots,x_{r^{*}}\}\subset N_G(x)$ and $\{x_2,\ldots,x_{r^*}\}\subset N_G(x) \cap N_G(z)$.
This follows by first choosing $\{x_2,\ldots,x_{r^*}\}$ and then $x_1$ with the right adjacencies. %and noting that we divide by $({r^*})!$ as the order we select the vertices does not matter.
 Thus, by \Cref{lem:supersat} we have $\beta_1'n^{r-1}$ labelled embeddings of the blow-up, $H_1\setminus\{w_{1},w_{2}\}=K^{r^*}_{k,\ldots,k,k-1}$, of these cliques, 
crucially within the correct neighbourhoods ($N_G(x)$ and $N_G(x)\cap N_G(z)$) to ensure that
 together with $\{x,z\}$ they give us the required embeddings of $H_1$. %if we let $U:=N(x)\cap N(z)$, we have that $|U|\geq (1-\frac{2}{r^*}+2\eta)n$. Now we claim that there are $(\eta n)^{r^*-1}\geq (\eta |U|)^{r^*-1}$ copies $K_{r^*-1}$ in $G[U]$. Indeed, we can form these greedily, as our minimum degree condition gives that every set of size at most $r^*$ has at least $\eta n$ common neighbours. Therefore given $x,z,u_1,\ldots,u_i \in V(G)$ we have at least $\eta n$ choices for $u_{i+1} \in U \cap_{j=1}^i N(u_i)$, for each $i=1,\ldots, r^*-2$. 
\end{proof}

%\begin{remark}
%Note that the proof of \Cref{case1reachability} shows more than the statement of the proposition. Indeed, we build a reachable set between any $x$ and $y$ via any $z$ and our choices of $\overline{H_{\text{det}}_i}$ ensure that $\overline{H_{\text{det}}_1} \setminus \{x\}, \overline{H_{\text{det}}_1}\setminus \{z\}, \overline{H_{\text{det}}_2}\setminus \{z\}, \overline{H_{\text{det}}_2}\setminus \{y\}$ all give copies of $H_{\text{det}}$. However the copy of $H_{\text{det}}$ containing $x$ also has all vertices incident to $x$, and likewise with $y$. In particular, this implies that in order to complete a given $H_{\text{det}}$ reachable set into a $K_{r^*}$ reachable set requires the appearance of disjoint copies of $K_m$ in $G(n,p)$, none of which are incident to the fixed vertices $x$ and $y$.
%\end{remark}

Note that an $\overline{\bm{H_1}}$-path $\overline{P_1}$ has endpoints which are isolated. The other vertices of $\overline{P_1}$
%a $(\overline{\bm{H_1}},2)$-path 
lie in copies of $K_k$ and these copies are  disjoint from each other except for a single pair of $K_k$s that 
 meet at a singular vertex. See Figure \ref{fig:Hpaths} for an example. We now turn to Case 2, as described in Section~\ref{sec:overview}. Here we can use the graph $H_0$ from Definition \ref{Hzero}. We also use a slight variant of $H_0$ where we redefine the distinguished vertices.

\begin{proposition} \label{case2reachability}
Suppose 
%we are in case 2, and so 
$\gamma>0$, $n,r, k\in\mathbb{N}$, such that $n$ is sufficiently large, $2\le k <r/2$ and ${r}/{k}\notin \mathbb{N}$. Further, let $r^*,q$ and $H_0$ be as defined in Definition~\ref{Hzero} and let $H_0'=(K_{k,\ldots,k,q+1}^{r*},1,2)$ be the same graph as $H_0$ with distinguished vertices in distinct\footnote{Note that this is possible as we are in the case where the number of parts, $r^*$ of $H_0$ is at least 3.} parts of size $k$ (see Figure \ref{fig:H2}). 

Then there exists $\beta_2=\beta_2(r,k,\gamma)>0$ such that for any $n$-vertex $G$  of minimum degree $\delta(G)\ge (1-\frac{k}{r}+\gamma)n$, 
  every pair of distinct vertices $x,y$ in $V(G)$ are $(\bm{H_2};\beta_2)$-reachable where $\bm{H_2}:=(H_0,H_0',H_0',H_0)$. 
  %with $H_0$ as in \Cref{Hzero} and  $H_2:=K^{r^*}_{k,\ldots,k,k,q+1}$, with identified vertices $w_{21}$ and $w_{22}$ in distinct\footnote{Note that this is possible as we are in the case where the number of parts, $r^*$ of $H_0$ is at least 3.} parts of size $k$.
\end{proposition}

\begin{figure}
    \centering
  \includegraphics[scale=0.9]{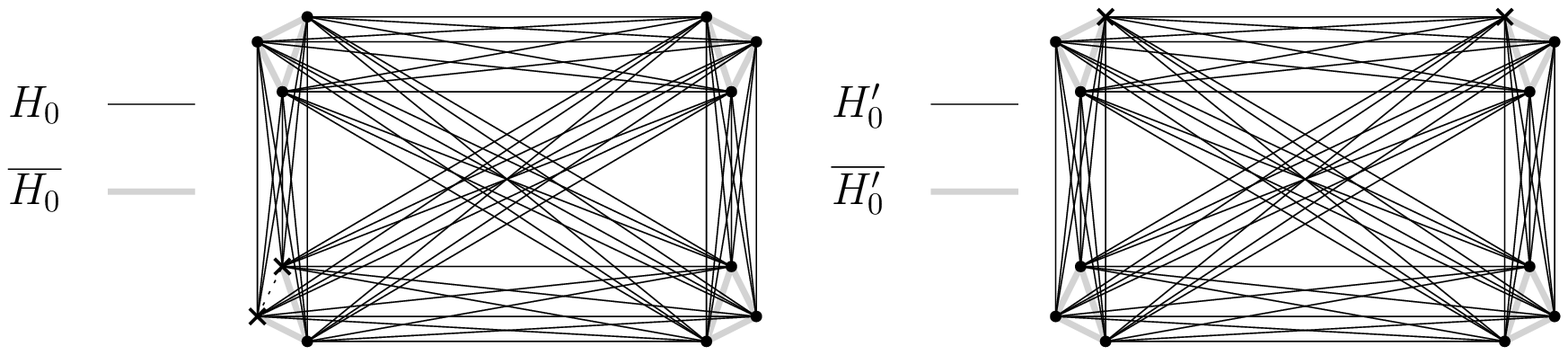}
    \caption{    \label{fig:H2} An example of $H_0$, $\overline{H_0}$, $H_0'$  and $\overline{H'_0}$  for Case 2 (b)  $r=11, k=3, q=2, r^*=4$ (see Definitions \ref{def:H}--\ref{Hzero} and Proposition \ref{case2reachability} for the definitions).}
\end{figure}

\begin{proof}
%The covering graphs we will use, as in the notation of \Cref{Reachable sets}, are $H_0=K^{r^*}_{m,\ldots,m,q+1}$ with $w^{*}_1$ and $w_{02}$ in the same part of size $q+1$ and $H_2=K^{r^*}_{m,\ldots,m,q+1}$ with $w_{21}$ and $w_{22}$ in distinct parts of size $m$. %Firstly, let us show that a given $x \in V(G)$ lies in many copies of $H_0$ with many $z$ playing the role of $w^{*}_2$. 
%\begin{claim}
%There is some $\beta'>0$ such that every $x \in V(G)$ is $(H_{\text{det}}, \beta',1)$-reachable (with respect to $H_0$) to at least $\beta'n$ vertices. 
%\end{claim}
%\begin{claimproof}
We know that $\delta(G)\geq (1-\frac{k}{r}+\gamma)n \geq (1-\frac{1}{r^*-1}+\gamma)n$. Therefore every set of at most $r^*-1$ vertices has a common neighbourhood of size at least $\gamma n$. We will appeal to \Cref{lem:supersat} to give us the whole $\bm{H_2}$-path in one fell swoop. Let $J$ be a graph with vertex set
\[
V(J)=\{x_1,\ldots,x_{r^*-1},z_1,w_1,\ldots,w_{r^*-2},z_2,u_1,\ldots,u_{r^*-2},z_3,y_1,\ldots,y_{r^*-1}\},
\]
and $E(J)$ consisting of $r^*$-cliques on $\{x_1,\ldots,x_{r^*-1},z_1\}$, $\{z_1,w_1,\ldots,w_{r^*-2},z_2\},$  $\{z_2,u_1,\ldots,u_{r^*-2},z_3\}$ \\ and $\{z_3,y_1,\ldots,y_{r^*-1}\}$. We claim that if we can find $(\gamma n)^{4r^*-3}/2$ copies of $J$ in $G$ such that $x_i \in N_G(x)$ and $y_i\in N_G(y)$ for $i=1,\ldots,r^*-1$, then we are done. 
Indeed, consider a blow-up $J'$ of $J$ with parts \begin{eqnarray*}\{X^{(k)}_1,\ldots,X^{(k)}_{r^*-1},Z^{(k+q-1)}_1,W^{(q+1)}_1,W^{(k)}_2,\ldots,W^{(k)}_{r^*-2},Z^{(2k-1)}_2,\\U^{(q+1)}_1,U^{(k)}_2,\ldots,U^{(k)}_{r^*-2},Z^{(k+q-1)}_3,Y^{(k)}_1,\ldots,Y^{(k)}_{r^*-1}\},\end{eqnarray*}
where the parts correspond to the vertices of $J$ in the obvious way and the size of each part is indicated by the superscript. Now if we have a copy of $J'$ in $G\setminus\{x,y\}$ with $X_i \subset N_G(x)$ and $Y_i\subset N_G(y)$ for all $i=1,\ldots, r^{*}-1$, then this gives us an embedding of an ${\bm{H_2}}$-path. Indeed for $i=1,3$, arbitrarily partition $Z_i:=\{z'_i\}\sqcup Z'_i \sqcup Z''_i$ with $|Z'_i|=q-1$ and $|Z''_i|=k-1$ and partition $Z_2:=\{z'_2\}\sqcup Z'_2 \sqcup Z''_2$ with $|Z'_2|=|Z''_2|=k-1$.
Then $\{x,z'_1\}\cup Z'_1\cup^{r^*-1}_{i=1} X_i$ and $\{y,z'_3\}\cup Z'_3 \cup^{r^*-1}_{i=1} Y_i$ both give copies of $H_0$ whilst $\{z'_1,z_2'\}\cup Z''_1 \cup Z'_2 \cup^{r^*-2}_{i=1}W_i$ and $\{z_2',z'_3\}\cup Z''_3 \cup Z''_2 \cup^{r^*-2}_{i=1}U_i$ both give copies of $H_0'$ where in all cases the distinguished vertices appear in the first set of the union.

It suffices then, by \Cref{lem:supersat},  to find $(\gamma n)^{4r^*-3}/2$ embeddings of $J$ in $G$ with the $x_i \in N_G(x)$ and $y_i \in N_G(y)$. We can do this greedily. Indeed if we choose the $x_i$ and $y_i$ first, followed by $z_1$ and $z_3$, then $z_2 \in N_G(z_1) \cap N_G(z_3)$ and then the remaining vertices, we are always seeking to choose a vertex in $G$ which has at most $r^*-1$ neighbours which have already been chosen. Thus, by our degree condition, we have at least $\gamma n$ choices for each vertex with the right adjacencies. To ensure that these choices actually give an embedding of $J$  we then discard any set of choices with repeated vertices, of which there are $O(n^{4r^*-4})$, and thus the conclusion holds as $n$ is sufficiently large. 
\end{proof}

Consider an $\overline{\bm{H_2}}$-path which we denote $\overline{P_2}$ (see Figure \ref{fig:Hpaths} for an example). It is formed by copies of $K_k$ and $K_{q+1}^-$ which intersect in at most one vertex and such that the endpoints of $\overline{P_2}$ lie in copies of $K_{q+1}^-$. Furthermore, note that the endpoints of $\overline{P_2}$ are in distinct connected components.   This  will be an important feature when we start to address the random edges of our absorbing structure as it will allow us to use Lemma~\ref{PhiFobservations} to conclude certain statements about the likelihood of finding our desired random subgraph in $G(n,p)$. This motivated the introduction of $H_0'$ in the previous proposition.

In Case 3, we cannot hope to prove reachability between every pair of vertices. Indeed our minimum degree in this case is $\delta(G)\geq\left(1-\frac{k}{r}+\gamma\right)n$ and $k>r/2$ and so it is possible that $\delta(G)<n/2$ and $G$ is disconnected. Thus, as in \cites{hansolo, han2016complexity}, we use a partition of the vertices into `closed' parts, where we can guarantee that two vertices in the same part are reachable, with some set of parameters. We adopt the following notation which also allows us to consider different possibilities for what vectors we use for reachability. 

\begin{definition} \label{def:closed}
Let $\bm{\cH}$ be a set of vectors, such that the entry of each vector in $\bm{\cH}$ is an $(r+1)$-vertex graph endowed with a tuple of distinguished vertices. We say that two vertices in $G$ are $(\bm{\cH};\beta)$-reachable if they are $(\bm{H};\beta)$-reachable for some $\bm{H}\in\bm{\cH}$. 

We  say that a subset $V$ of vertices in a graph $G$ is $(\bm{\cH};\beta)$-closed if every pair of vertices\footnote{Note that we do not require the vertices of the $\bm{H}$-paths  which give the reachability to lie in $V$.} in $V$ is $(\bm{\cH};\beta)$-reachable. We  denote\footnote{If $\bm{\cH}$ consists of just one vector $\bm{H}$, we simply refer to sets being $(\bm{H};\beta)$-closed and use $N_{\bm{H},\beta}(v)$ to denote the closed neighbourhood of a vertex.} by $N_{\bm{\cH},\beta}(v)$ the set of vertices in $G$ that are $(\bm{\cH};\beta)$-reachable to $v$. 
\end{definition}

Thus, in this notation, the conclusion of Proposition~\ref{case1reachability} states that $V(G)$ is $(\bm{H_1};\beta_1)$-closed for all $G$ satisfying the given hypothesis (and similarly for Proposition~\ref{case2reachability}). 
Notice that if a set $V$ is $(\bm{\cH};\beta)$-closed in a graph $G$ it may be the case that two vertices $x,y \in V$ are $(\bm{H};\beta)$-reachable whilst two other vertices $z,w \in V$ are $(\bm{H}';\beta)$-reachable for some 
distinct $\bm{H}$, $\bm{H}' \in \bm{\cH}$ of different lengths.

It will be useful for us to consider the following notion. 

\begin{definition} \label{def:concatenation}
Let $\bm{\cH},\bm{\tilde{\cH}}$ be two sets of vectors as in Definition~\ref{def:closed}. Then 
\[\bm{\cH}+\bm{\tilde{\cH}}:=\bm{\cH}\cup \bm{\tilde{\cH}}\cup (\bm{\cH}\cdot \bm{\tilde{\cH}}),\]
where $\bm{\cH}\cdot \bm{\tilde{\cH}}$ is defined to be the set\[\bm{\cH}\cdot \bm{\tilde{\cH}}:=\{(\bm{H},\bm{\tilde{H}}):=(H^1,\ldots,H^t,\tilde{H}^1,\ldots,\tilde{H}^{\tilde{t}}):\bm{H}:=(H^1,\ldots,H^t)\in \bm{\cH},\bm{\tilde{H}}:=( \tilde{H}^1,\ldots,\tilde{H}^{\tilde{t}})\in\bm{\tilde{\cH}}\}.\]
That is, $\bm{\cH}+\bm{\tilde{\cH}}$ comprises of all vectors that lie in $\bm{\cH}$, $\bm{\tilde{\cH}}$, or that can be obtained by a concatenation of a vector from $\bm{\cH}$ with a vector from $\bm{\tilde{\cH}}$.
\end{definition}

As an important example, defining $\bm{\cH}(H,\leq t):=\{(H,s):1\leq s\leq t\}$, we have that 
\[\bm{\cH}(H,\leq t_1)+\bm{\cH}(H,\leq t_2)= \bm{\cH}(H,\leq t_1+t_2).\] %\COMMENT{AT: I replaced $\subseteq$ with equality here, as unless I've misunderstood, I think this is the case. PM: So technically it is not quite equality as the definition stands because we can't have a vector of length 1, eg a single copy of $H$ after the $+$ operation. We could adjust the definition of +, or of $\bm{\cH}(H,\leq t)$ and allow `empty' vectors, or leave it as it is and maybe comment on why it is not equality?}
In what follows we will apply the following simple 
lemma repeatedly.
%We will say that two vertices are $(H,\beta, \leq t)$- reachable if they are $(H,\beta, t')$- reachable for some $t'\leq t$. We will say that a subset of vertices is $(H,\beta, \leq t)$-closed if every pair of vertices in the subset is $(H,\beta, \leq t)$-reachable and we will denote by $N_{H,\beta,\leq t}(v)$ the set of vertices in $G$ that are $(H,\beta,\leq t)$-reachable to $v$. We will need the following easy lemma repeatedly. 

\begin{lemma} \label{concatenation lemma}

Let $r\in\mathbb{N}$ and let $\bm{\cH_x}, \bm{\cH_y}$ be two sets of vectors of $(r+1)$-vertex graphs, each of which is endowed with a tuple of distinguished vertices and suppose that $t_x:=|\bm{\cH_x}|$ and $ t_y:=|\bm{\cH_y}|$ are both finite. 
Suppose $G$ is a sufficiently large $n$-vertex graph and $x,y \in V(G)$.
 Suppose there exist $\beta_x,\beta_y,\epsilon >0$, 
%$t_x,t_y\in \mathbb{N}$ 
and some subset $U\subseteq V(G)$ with $|U| \geq \epsilon n$ such that for every $z\in U$, $x$ and $z$ are $(\bm{\cH_x};\beta_x)$-reachable and $z$ and $y$ are $(\bm{\cH_y};\beta_y)$-reachable. Then $x$ and $y$ are $(\bm{\cH_x}+\bm{\cH_y};\beta)$-reachable for $\beta:=\frac{\epsilon \beta_x\beta_y}{2t_xt_y}>0$.
\end{lemma}

\begin{proof}
By the pigeonhole principle, there exists some $U'\subseteq U$ such that $|U'|\geq \frac{\epsilon n}{t_x t_y}$ and some $\bm{H}_x\in\bm{\cH_x}$, $\bm{H}_y\in\bm{\cH_y} $ such that for every $z\in U'$, $z$ and $x$ are $(\bm{H}_x;\beta_x)$-reachable and $z$ and $y$ are $(\bm{H}_y;\beta_y)$-reachable.
%for fixed $t'_x\leq t_x$ and $t'_y\leq t_y$. 
Suppose $\bm{H}_x$ has length $s_x$ and $\bm{H}_y$ has length $s_y$. Thus, fixing $z \in U'$, there are at least $\beta_x\beta_yn^{(s_x+s_y)r-2}$ pairs of labelled vertex sets $S_x$ and $S_y$ in $G$ such that there is an embedding  of an $\bm{H}_x$-path on $S_x\cup \{x,z\}$ mapping endpoints to $\{x,z\}$ and an embedding of a $\bm{H}_y$-path on the vertices $S_y\cup\{y,z\}$ which maps the endpoints to  $\{y,z\}$. Of these pairs, at most 
\[
s_xs_yr^2n^{(s_x+s_y)r-3}
\]
are \emph{not} vertex disjoint or they intersect $\{x,y\}$. Hence, as $n$ is sufficiently large we have at least $\frac{\beta_x\beta_y}{2}n^{(s_x+s_y)r-2}$ vertex disjoint pairs which together form an embedding of an $(\bm{H_x},\bm{H}_y)$-path. As we have
at least $\frac{\epsilon}{t_xt_y} n$ choices for $z$, this gives that $x$ and $y$ are $((\bm{H}_x,\bm{H}_y);\beta)$-reachable and 
$(\bm{H}_x,\bm{H}_y)\in \bm{\cH_x}+\bm{\cH_y}$.
%$t'_x+t'_y\leq t_x+t_y$.
\end{proof}

We now turn to proving reachability in Case 3.
%The following lemma is very close to~\cite[Lemma 6.3]{han2016complexity}.
The following two lemmas together find the partition we will work on.
Similar ideas have been used in~\cites{hansolo, han2016complexity}.

\begin{lemma} \label{case3lemma}
Suppose $\gamma>0$ and $n,r,k,q\in \mathbb{N}$ such that $r/2< k\le r-1$, $r=k+q$ and $n$ is sufficiently large. Let $c:=\ceil{r/q}$ and for $t\in \mathbb{N}$ define $\bm{\cH}^t:=\bm{\cH}(H_0,\leq 2^{t})=\{(H_0,s):1\leq s\leq 2^t\}$, where $H_0=(K^2_{k,q+1},2,2)$ is as defined in Definition~\ref{Hzero} with distinguished vertices $w_1$ and $w_2$.

Then there exists constants $0<\beta_3'=\beta_3'(r,k,\gamma)\ll\alpha=\alpha(r,k,\gamma)$ such that any $n$-vertex graph $G$ of minimum degree $\delta(G)\ge (1-\frac{k}{r}+\gamma)n$ can be partitioned into at most $c-1$ parts, each of which is $(\bm{\cH}^{c};\beta_3')$-closed and of size at least $\alpha n$.%\COMMENT{AT: looking at the proof, actually it seems we can partition into at most $c-1$ parts}

%Assume we are in case 3, and so we have $\gamma>0$, $n,r,k\in \mathbb{N}$ such that $\frac{r}{2}< k\le r-1$ and $n$ is sufficiently large. Let $G$ be a fixed graph with $n$ vertices and of minimum degree $\delta(G)\ge (1-\frac{k}{r}+\gamma)n$.
%Then there exists constants $\delta=\delta(r,k,\gamma)\gg\beta_3'=\beta_3'(r,k,\gamma)>0$, and there is a partition $\mathcal{P}$ of the vertex set into $V_1,\ldots, V_{s'}$ with $s'\leq c:=\ceil{\frac{r}{r-k}}$ such that $|V_i|\geq \delta n$ for each $i$ and each $V_i$ is $(H_0,\beta_3',\leq 2^{c-2})$-closed, recalling that $H_0=K_{k,q+1}$ with identified vertices $w_{01},w_{02}$ in the part of size $q+1$.
\end{lemma}

\begin{proof}
%\COMMENT{AT: made quite a lot of small changes in this proof}
%We will show $(H_{\text{det}},\beta,t)$-reachability with respect to $H_0=K_{m,q+1}$ with $w_1$ and $w_2$ in the same part of size $q+1$. 
Firstly, observe that there is some $\eta=\eta (r,k,\gamma)>0$ such that in every set of at least $c$ vertices, there are two vertices which are $((H_0,1);\eta)$-reachable. Indeed, fix some arbitrary set of vertices $S=\{v_1,\ldots,v_c\}\subset V(G)$, and for $v\in V=V(G)$, define $d_S(v):=|\{i\in [c]:vv_i\in E(G)\}|$. Let $\bar{d}_S:=\sum_{v\in V}d_S(v)/n$ be the average.  Then we have that 
\[\sum_{v\in V}d_S(v)=   \sum_{i\in[c]}d_G(v_i) \geq c\left(1-\frac{k}{r}+\gamma\right)n\geq (1+c\gamma)n.\]
Thus $\sum_{v\in V}\binom{d_S(v)}{2}\geq n \binom{\bar{d}_S}{2}\geq \frac{(c\gamma)^2}{2}n$ by Jensen's inequality. 
By averaging over all pairs we have that there exists a pair $i\neq j\in [c]$ so that  both $v_i$ and $v_j$ are in the neighbourhood of at least $\gamma^2n$ vertices. That is, $|N_G(v_i)\cap N_G(v_j)|\geq \gamma^2 n$. 

Therefore there are at least $\gamma^3 n^2$ edges  in $G$ with one endpoint in $N_G(v_i)\cap N_G(v_j)$. Applying \Cref{lem:supersat} this ensures that there is $\eta=\eta (r,k,\gamma)>0$ 
so that there are $\eta n^{r-1}$
copies of $K^2_{k,q-1}$ where the first vertex class lies in 
$N_G(v_i)\cap N_G(v_j)$. Thus together they form copies of $H_0$ with distinguished vertices $v_i$ and $v_j$; so $v_i$ and $v_j$ are 
 $((H_0,1);\eta)$-reachable in $G$.

%there must be a pair $v_i$, $v_j$ such that $|N(v_i)\cap N(v_j)|\geq \frac{2\gamma}{c} n$ as given $c$ vertex subsets of size $\left(\frac{r-k}{r}+\gamma\right) n$, there are at least $c\gamma n$ vertices which feature in at least two such subsets and there are less than $\frac{c^2}{2}$ pairs of subsets. Given such a pair $v_i$ and $v_j$, \Cref{lem:supersat} then gives $(H_0,\eta,1)$-reachability for an appropriate $\eta>0$, as there are at least $\frac{2\gamma^2}{c}n^2$ edges in $G$ with one endpoint in $N(v_i)\cap N(v_j)$. 

\smallskip

Note also that  there is some fixed $\alpha'=\alpha' (r,k,\gamma)>0$ such that $|N_{(H_0,1),\alpha'}(v)|\geq \alpha' n$ for every $v\in V(G)$. 
Indeed,
this follows as there are at least $(\gamma n)^2/2$ edges in $G$ with one endpoint in $N_G(v)$. So, by \Cref{lem:supersat}, there is a fixed $\alpha''=\alpha'' (r,k,\gamma)>0$ such that there are
 at least $\alpha''n^r$ embeddings of $H_0$ which map $w_{1}$ to $v$. Setting $\alpha' :=\alpha''/3$, this implies  that there are at least $\alpha' n$ vertices which are $((H_0,1);\alpha')$-reachable to $v$. 
%Indeed otherwise we get that there are too few embeddings of $H_0$ which map $w_{1}$ to $v$ by counting both the copies with $w_{2}$ mapped to a vertex in  $N_{(H_0,1),\alpha}(v)$ and those which map $w_{2}$ to another vertex.
%N_{H_0,\alpha,1}(v)$.

\smallskip

Now let $0<\epsilon\ll \alpha'$, $\eta=:\eta_0$ and $\eta_i:=\frac{\epsilon^4}{2^{2i+1}}\eta_{i-1}$ 
%\COMMENT{AT: changed definition from $\eta_i:=\frac{\epsilon}{2^{2(i-1)}}\eta_{i-1}^2$ here as I think we need the slightly smaller value of $\eta _i$ in one application of \Cref{concatenation lemma} }
for $i=1,\ldots, c$. 
Set $\beta '_3:= \eta _c$.
As in the statement of the lemma, define $\bm{\cH}^t:=\bm{\cH}(H_0,\leq 2^{t})$ 
for values of $t\leq c$ and note that $\bm{\cH}^t+\bm{\cH}^t= \bm{\cH}^{t+1}$. %\COMMENT{AT: again I think it is equality here. PM: Same as before.} 
We will be interested in $(\bm{\cH}^t;\eta_t)$-reachability and so we will use the shorthand notation 
$\tilde{N}_t(v):=N_{\bm{\cH}^t,\eta_{t}}(v)$. Let $\ell$ be the maximal integer such that there exists a set of $\ell$ vertices, $v_1,\ldots, v_\ell$ with $v_j$ and $v_{j'}$ \emph{not} 
$(\bm{\cH}^{c-\ell};\eta_{c-\ell})$-reachable for any pair $j\neq j'\in [\ell]$. 

Suppose $\ell=1$. Then $V(G)$ is $(\bm{\cH}^{c-1};\eta_{c-1})$-closed. 
As $\bm{\cH}^{c-1} \subseteq \bm{\cH}^{c}$ and $\eta _{c-1} >\beta '_3$, the lemma holds in this case.

We also have that $\ell\leq c-1$ from our observations above, so we can assume $2\leq \ell\leq c-1$. Now fix such a set of $\ell$ vertices, $v_1,\ldots,v_\ell$. We make the following two observations: 

\begin{enumerate}[label=\roman*]
\item[(i)] Any $v\in V(G)\setminus \{v_1,\ldots,v_\ell\}$ is in $\tilde{N}_{c-\ell-1}(v_j)$ for some $j\in[\ell]$ from our definition of $\ell$, as otherwise $v$ could be added to give a larger family contradicting the maximality of $\ell$. Indeed, this follows because two vertices that are not  $(\bm{\cH}^{c-\ell};\eta_{c-\ell})$-reachable  are certainly not  $(\bm{\cH}^{c-\ell-1};\eta_{c-\ell-1})$-reachable by definition. 
\item[(ii)] $|\tilde{N}_{c-\ell-1}(v_j)\cap \tilde{N}_{c-\ell-1}(v_{j'})|\leq \epsilon n$ for every pair $j \neq j' \in [\ell]$. This follows from \Cref{concatenation lemma} as otherwise we would have that $v_j$ and $v_{j'}$ are $(\bm{\cH}^{c-\ell};\eta_{c-\ell})$-reachable, a contradiction.
%This follows, again from our definition of the family $v_1,\ldots,v_k$. If this was not the case, say for a fixed $v_j$ and $v_{j'}$, then for any $v\in\tilde{N}_{c-k-1}(v_j)\cap \tilde{N}_{c-k-1}(v_{j'})$, we have some $t,t'\leq 2^{c-k-1}$, such that there are $\beta_{c-k-1}^2n^{(t+t')r-2}$ pairs of labelled sets $S_j$ and $S_{j'}$ such that $S_j$ is a $(H_{\text{det}},t)$-reachable set between $v_j$ and $v$
%and $S_{j'}$ is a $(H_{\text{det}},t)$-reachable set between $v_{j'}$ and $v$. Of these pairs, at most $$\beta_{c-k-1}(tr-1)(t'r-1)n^{(t+t')r-3}$$ are \emph{not} vertex disjoint. Thus, for large enough $n$, we get at least $\frac{\beta_{c-k-1}^2}{2}n^{(t+t')r-2}$ vertex disjoint pairs $S_j$ and $S_{j'}$ which together form a $(H_{\text{det}},t+t')$-reachable set. As we have $\epsilon n$ choices for $v$, this gives at least $\beta_{c-k}n^{(t+t')r-1}$ $(H_{\text{det}},t+t')$-reachable sets between $v_j$ and $v_{j'}$, a contradiction as $v_j$ and $v_{j'}$ are not $(H_{\text{det}},\beta_{c-k},\leq 2^{c-k})$-reachable and $t+t'\leq 2^{c-k}$.
\end{enumerate}
We define $U_j:=\left(\tilde{N}_{c-\ell-1}(v_j) \cup \{v_j\} \right)\setminus \left ( \bigcup_{j'\in[\ell]\setminus \{j\}}\tilde{N}_{c-\ell-1}(v_{j'}) \right )$ for $j\in [\ell]$, and 
$U_0:=V(G) \setminus \cup_{j\in[\ell]}U_j$. Now for $j\in [\ell]$, we have that $U_j$ is $(\bm{\cH}^{c-\ell-1};\eta_{c-\ell-1})$-closed. 
Indeed, if there was a $j\in[\ell]$ and $u_1,u_2 \in U_j$ not reachable, then $\{u_1,u_2\}\cup\{v_1,\ldots,v_\ell\}\setminus\{v_j\}$, is a 
larger family contradicting the definition of $\ell$. Thus, the $U_j$ almost form the partition we are looking for except that it remains to consider the vertices in $U_0$. For these, we greedily add them to  the other $U_j$: We have that for each $u\in U_0$, 
\begin{equation}
    \label{eq2}
|N_{(H_0,1),\alpha'}(u)\setminus U_0|\geq \alpha ' n - |U_0| %\stackrel{(ii)}
{\geq} \alpha ' n -\binom{\ell}{2}\epsilon n\geq \ell \epsilon n.\end{equation}
Here the second inequality holds due to (i), (ii) and the definition of the $U_j$; the final inequality holds by our choice of $\epsilon$. 
Thus, there is a $j$ such that $|N_{(H_0,1),\alpha'}(u)\cap U_j|\geq \epsilon n$, and we add $u$ to this $U_j$, arbitrarily choosing such a $j$ if there are multiple choices. 
Let $V_1,\ldots, V_\ell$ be the resulting partition. 

Applications of \Cref{concatenation lemma} show that
each $V_j$ is $(\bm{\cH}^{c};\eta_{c})$-closed.  Indeed suppose,  for example, that $w_1$ and $w_2$ are two vertices that lie in $U_0$ and are added to $U_j$ in the process of defining $V_j$. Then for each $i=1,2$, taking $W_i=N_{(H_0,1),\alpha'}(w_i)\cap U_j$,  an application of \Cref{concatenation lemma} with $U=W_1$ gives that for any $x\in U_j$, $w_1$ and $x$ are $(\bm{\cH}';\eta')$-reachable  where $\bm{\cH}'=\bm{\cH}(H_0,\leq 2^{c-\ell-1}+1)$ and $\eta'=\eps \alpha' \eta_{c-\ell-1}/2^{c-\ell}$. Another application of \Cref{concatenation lemma}, this time with $U=W_2$ then gives that $w_1$ and $w_2$ are $(\bm{\cH}'';\eta'')$-reachable with $\bm{\cH}''=\bm{\cH}(H_0,\leq 2^{c-\ell-1}+2)\subseteq \bm{\cH}^{c}$ and $\eta''=\eps \eta'\alpha'/(2^{c-\ell}+2)>\eta_{c-\ell}>\beta_3'$.  Showing other cases of reachability within each $V_j$ are similar.  
%\COMMENT{AT: replaced $(\bm{\cH}^{c-\ell+1};\eta_{c-\ell+1})$-closed with $(\bm{\cH}^{c-\ell};\eta_{c-\ell})$-closed. Think this is correct!}
%Note that $\bm{\cH}^{c-\ell} \subseteq \bm{\cH}^{c}$ and $\eta _{c-\ell} >\beta '_3$, so each $V_j$ is $(\bm{\cH}^{c};\beta '_3)$-closed.
We are now done since for each $j\in [\ell]$,
\[|V_j|\geq |U_j|\geq |N_{(H_0,1),\alpha'}(v_j)\setminus \left(N_{(H_0,1),\alpha'}(v_j)\cap U_0\right)|\stackrel{(ii)}{\geq} \alpha' n- \ell\epsilon n = \alpha n,\]
where $\alpha:=\alpha'-\ell\epsilon \gg \eta_{c-\ell}\geq \eta_c=\beta_3'$.
\end{proof}

The rough idea for how to handle Case 3 is  to run the same proof as in the other cases on each \emph{part} of the partition given by Lemma \ref{case3lemma}. The point of Lemma \ref{case3lemma} is that we recover the reachability within each part, albeit at the expense of allowing a family of possible paths used for reachability. However, in the process, we lose the minimum degree condition within each part. The purpose of the next proposition is to fix this, by adjusting parameters and making the partition coarser. Thus, we recover a minimum degree condition which is not quite as strong as what we had previously but good enough to work with in what follows.

\begin{figure}
    \centering
  \includegraphics[scale=0.9]{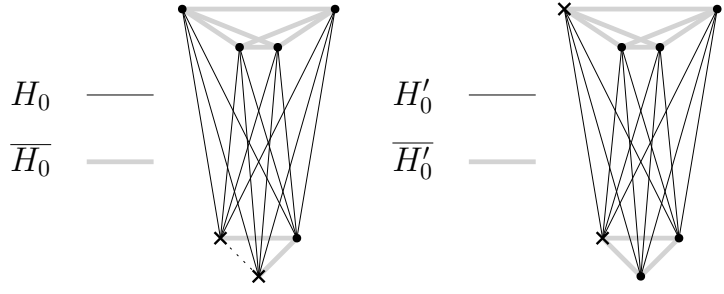}
    \caption{     \label{fig:H3} An example of $H_0$, $\overline{H_0}$, $H_0'$  and $\overline{H'_0}$  for Case 3  (c) $r=6, k=4, q=r^*=2$ (see Definitions \ref{def:H}--\ref{Hzero} and Proposition \ref{case3reachability} for the definitions). }
\end{figure}

\begin{proposition} \label{case3reachability}
Suppose $\gamma>0$ and $n,r,k,q\in \mathbb{N}$ such that $r/2< k\le r-1$, $r=k+q$ and $n$ is sufficiently large. Let $c:=\ceil{r/q}$ and let $H_0=(K^2_{k,q+1},2,2)$ as defined in Definition \ref{Hzero} and $H'_0=(K^2_{k,q+1},1,2)$ be the same graph with distinguished vertices in distinct parts of the bipartition\footnote{This is analogous to the graph $H_0'$ defined in Proposition \ref{case2reachability}.} (see Figure \ref{fig:H3}). We define the following family of vectors of $(r+1)$-vertex graphs (endowed with tuples of vertices):
\[\bm{\cH_3}:= \bigcup^{c(2^{c+1}+1)}_{t=3}\left\{\bm{H}\in\{H_0,H_0'\}^t:\bm{H}[1]=\bm{H}[t]=H_0  \mbox{ and } \bm{H}[i]=H_0' \mbox{ for some } 2 \leq i\leq t-1 \right\}, \]
where $\bm{H}[i]$ denotes the $i^{th}$ entry of $\bm{H}$.

Then for all $\eps>0$, there exists constants  $0<\beta_3(r,k,\gamma,\eps)\ll \alpha(r,k,\gamma)$ such that for any $n$-vertex graph $G$ with minimum degree $\delta(G)\geq (1-\frac{k}{r}+\gamma)n$ there is a partition $\cP$ of $V(G)$ into at most $c-1$ parts such that each part $U\in \cP$ satisfies the following:
\begin{enumerate}[label=(\roman*)]
%\item $s\leq c:=\ceil{\frac{r}{r-k}}$;
\item $|U|\geq \alpha n$; %for each $i\in [s]$;
\item  All but at most $\eps n$ vertices $v\in U$ satisfy
%For all $i\in [s]$, all but at most $\eps n$ vertices $v \in U_i$ have
$|N_G(v)\cap U|\geq (1-\frac{k}{r}+\frac{\gamma}{2})|U|$;
\item $U$ is $(\bm{\cH_3};\beta_3)$-closed.
%Defining $H_3:=K_{\ceil{r/2},\floor{r/2}+1}$ with $w_{31}$ and $w_{32}$ in distinct parts, and $t:=c(2^{c-2}+1)$, we have that for any $i\in[s]$ and choice of  $x,y\in U_i$,  $x$ and  $y$ are $({\bm{H}},\beta_3,t')$-reachable, where ${\bm{H}}$ is a family consisting only of the graphs $H_0$ and $H_3$ and $t' \leq t$;
%\item Every $U_i$ contains at least $\rho n^r$ copies of $H_3^-:=K_{\ceil{r/2},\floor{r/2}}$ and $\rho n^{r+1}$ copies of $H_3$.
\end{enumerate}

%Let $\gamma, n,k,r$ and $G$ be as in \Cref{case3lemma}. Then there exists constants $\rho,\delta,\eps,\beta_3> 0$ such that  $\gamma, \rho,\delta\gg \epsilon \gg \beta_3>0$ and a partition $\mathcal{P'}$ of the vertex set into $U_1,\ldots,U_{s}$ such that \begin{enumerate}[label=(\roman*)]
%\item $s\leq c:=\ceil{\frac{r}{r-k}}$;
%\item $|U_i|\geq \delta n$ for each $i\in [s]$;
%\item For all $i\in [s]$, all but at most $\eps n$ vertices $v \in U_i$ have $|N(v)\cap U_i|\geq (1-\frac{k}{r}+\frac{\gamma}{2})|U_i|$;
%\item Defining $H_3:=K_{\ceil{r/2},\floor{r/2}+1}$ with $w_{31}$ and $w_{32}$ in distinct parts, and $t:=c(2^{c-2}+1)$, we have that for any $i\in[s]$ and choice of  $x,y\in U_i$,  $x$ and  $y$ are $({\bm{H}},\beta_3,t')$-reachable, where ${\bm{H}}$ is a family consisting only of the graphs $H_0$ and $H_3$ and $t' \leq t$;
%\item Every $U_i$ contains at least $\rho n^r$ copies of $H_3^-:=K_{\ceil{r/2},\floor{r/2}}$ and $\rho n^{r+1}$ copies of $H_3$.
%\end{enumerate} 
\end{proposition}
\begin{proof}
This is a simple case of adjusting the partition already obtained after applying \Cref{case3lemma}. Let $\alpha,\beta_3'$ be defined as in the outcome of Lemma \ref{case3lemma} and let $\mathcal{P'}$ be the partition of $V(G)$ obtained, with vertex parts denoted $V_1,\ldots,V_{s}$. Fix $\mu:= \eps\alpha\gamma/2c^3$. We create an auxiliary graph $J$ on vertex set $\{V_1,\ldots,V_{s}\}$ where for $i\neq j\in [s]$ we have an edge $V_iV_j$ in $J$ if and only if there are at least $\mu n^2$ edges in $G$ with one endpoint in $V_i$ and one in $V_j$. Then our new partition $\cP$ in $G$ will come from the connected components of $J$. That is, if $C_1,\ldots,C_{t}$ are the components of $J$, then for $i\in[t]$, we define $U_i:=\cup_{j:V_j\in C_i}V_j$ and let $\cP$ consist of the $U_i$ with $i\in[t]$.  Then certainly point $(i)$  of the hypothesis is satisfied for all $U_i$. Also $(ii)$ is satisfied. %with $\eps:= 2\mu c^3/ \gamma$. 
Indeed, suppose there exists $i\in[t]$, with $d_{U_i}(v):=|N_G(v)\cap U_i|<(1-\frac{k}{r}+\frac{\gamma}{2})|U_i|$ for at least $\eps n$ vertices of $U_i$. Thus, for such vertices $|N_G(v)\cap (V(G)\setminus U_i)|\geq \frac{\gamma\alpha}{2}n$ and by averaging there exists some $V_{j(v)}\in \cP '$ such that $V_{j(v)}\cap U_i=\emptyset$ and $d_{V_{j(v)}}(v)\geq \frac{\gamma\alpha}{2c}n$. 
We average again to conclude that there is some $j,j'\in[s]$ such that $V_j\subset U_i$, $V_{j'}\cap U_i=\emptyset$ and $V_j$ contains at least ${\eps} n/{c^2}$ vertices $v$ 
%\COMMENT{AT: changed $\frac{\eps}{c} n$ to $\frac{\eps}{c^2} n$ and $\frac{\gamma\alpha}{2c^2}n$ to $\frac{\gamma\alpha}{2c}n$. Think that is correct}
which have degree into $V_{j'}$ $d_{V_{j'}}(v) \geq \frac{\gamma\alpha}{2c}n$. This contradicts our definition of $J$ as then $V_j V_{j'}$ should be an edge of $J$ and thus in the same part of $\cP$.

%ndeed, suppose we have $\eps n$ vertices in some $U_i$ with degree at least $\frac{\gamma}{2} n$ outside of $U_i$. Then at least $\frac{\eps}{c}n $ of these points were belonging to the same $V_j\subseteq U_i$ and all of these vertices have degree at least $\frac{\gamma}{2c}n$ to some part $V_{j'}$ such that $V_{j'}\nsubseteq U_i$. Thus there exists some $j^*$ such that $V_{j^*}\nsubseteq U_i$, and there are at least $\frac{\eps}{c^2}n$ vertices in $V_j\subseteq U_i$ with  at least $\frac{\gamma}{2c}n$ neighbours in $V_{j^*}$. But then there are $\mu n^2$ edges between $V_j$ and $V_{j^*}$, a contradiction.  For point $(v)$, we appeal to \Cref{lem:supersat}. 
%Indeed, there are at least $\rho' n^2$ edges in each $U_i$ where $\rho'=\delta\gamma-s'\mu>0,$ counting the degrees of vertices in $U_i$ and discarding edges that leave $U_i$. Thus \Cref{lem:supersat} gives at least $\rho n^r$ copies of $H_3^-$ in $U_i$ and similarly for $H_3$ for some suitably defined $\rho$. 

Thus it only remains to establish reachability. We begin by proving the following claim which is a slight variation of Lemma~\ref{concatenation lemma}. 
\begin{claim}\label{cc}
 Let $\bm{\cH}^c$ be as defined in Lemma \ref{case3lemma}. Suppose $x,y\in V(G)$ and that there exist (not necessarily disjoint)
 sets $S_x, S_y\subset V(G)$ such that for any $z_x\in S_x$, $x$ and $z_x$ are  $(\bm{\cH}^c; \beta_3')$-reachable and for any $z_y\in S_y$, $y$ and $z_y$ are  $(\bm{\cH}^c; \beta_3')$-reachable. If there exists at least $\mu n^2$ edges with one endpoint in $S_x$ and one endpoint in $S_y$, then $x$ and $y$  are $(\bm{H};\beta_3'')$-reachable for some $\beta_3''=\beta_3 ''(\mu, \beta '_3,c)>0$ and $\bm{H}\in \bm{\cH_3}$ of length at most $2^{c+1}+1$. 
\end{claim}
Indeed letting $w'_1, w'_2$ be the distinguished vertices of $H_0'$, we have, by Lemma \ref{lem:supersat}, that there are at least $\mu'n^{r+1}$ embeddings of $H_0'$ into $G$ which map $w'_1$ to $S_x$ and $w_2'$ to $S_y$ for some $\mu'=\mu'(\mu)>0$. 
By averaging, there exists $\bm{H}_x,\bm{H}_y\in \bm{\cH}^c$ such that there are $\frac{\mu'}{2^{2c}}n^{r+1}$ embeddings of $H'_0$ such that the image of $w_1'$ and $x$ are $(\bm{H}_x;\beta_3')$-reachable and the image of $w_2'$ and $y$ are $(\bm{H}_y;\beta_3')$-reachable. By considering the embeddings of $\bm{H}_x$, $\bm{H}_y$ and $H_0'$  which join to give an embedding of an $(\bm{H}_x,H_0',\bm{H}_y)$-path (that is, ignoring choices of embeddings which are not vertex disjoint), we see that $x$ and $y$ are $((\bm{H}_x,H_0',\bm{H}_y),\beta_3'')$-reachable with $\beta_3'':=\frac{\mu'\beta_3'^2}{2^{2c+1}}$. This completes the proof of the claim.

\medskip

Recall the partition $\cP'=\{V_1,V_2,\ldots,V_s\}$. Further consider any part $U \in \mathcal P$.
First suppose $U=V_j$ for some $j$.  Now given any 
 $x, y\in U$, by Lemma \ref{case3lemma},  $x$ and $y$ are already $(H_0,s)$-reachable for some $s\leq 2^c$. However, $(H_0,s)$ 
does not contain a copy of $H'_0$ and so is not a valid vector in the family $\bm{\cH_3}$. We therefore apply Claim~\ref{cc} with $S_x=S_y=V_j\setminus \{x, y\}$, 
to conclude that $x$ and $y$ are  $(\bm{H};\beta_3'')$-reachable for some $\bm{H}\in \bm{\cH_3}$ of length at most $2^{c+1}+1$.  Indeed since $V_j$  sends fewer than $\mu n^2$ edges out to any other part $V_i$ of $\cP'$ and $|V_j|\geq \alpha n$, the minimum degree condition on $G$ ensures that
there are at least $2\mu n^2$ edges in $G[V_j]$ and hence $\mu n^2$ edges in $S_x=S_y$ allowing Claim~\ref{cc} to be applied. 

Next suppose $U$ is the union of more than one part from $\cP'$.
If $x\in V_i\subseteq U$ and $y\in V_j \subseteq U$, for $i\neq j\in [s]$ and $V_iV_j\in E(J)$ as defined above, we can again  apply Claim~\ref{cc} to conclude $x$ and $y$ are  $(\bm{H};\beta_3'')$-reachable for some $\bm{H}\in \bm{\cH_3}$ of length at most $2^{c+1}+1$. Therefore, we just need to establish reachability for vertices $x,y$ such that $x\in V_i$, $y\in V_j$ with $V_iV_j\notin E(J)$ but such that $V_i$ and $V_j$ are in the same component of $J$. 
If $i\not =j$, there is a path of (at most $c$) edges  from $V_i$ to $V_j$ in $J$; if $i=j$ there is a walk of length $2\leq c$ in $J$ that starts and ends at $V_i=V_j$ (i.e. traverse a single edge in $J$).
In both cases
 we can repeatedly apply Lemma \ref{concatenation lemma} to derive that $x$ and $y$ are $(\bm{\cH_3};\beta_3)$-reachable with $\beta_3:=\left(\frac{\alpha \beta_3''}{2^{(2^{c+1}+2)}}\right)^c.$ It is crucial here that we apply Claim~\ref{cc} in all cases to establish the reachability here (even when $i=j$) in order to guarantee that the  vectors witnessing the reachability contain a copy of $H_0'$ and hence indeed lie in $\bm{\cH_3}$.
\end{proof}

We remark that the reason for the introduction of $H_0'$ in Proposition~\ref{case3reachability} 
%\COMMENT{AT: changed from Proposition \ref{case1reachability}... I guess that's what you meant?}
is two-fold. Firstly, it allows us to establish reachability  between parts from Lemma \ref{case3lemma} which have many edges between them. Moreover, as in Proposition \ref{case2reachability}, we have that for every $\bm{H}\in\bm{\cH_3}$, if $\overline{P}$ is an $\overline{\bm{H}}$-path, then the endpoints of $\overline{P}$ are in distinct connected components on $\overline{P}$ (see Figure \ref{fig:Hpaths} for an example), which is something that we will require later.

%%%%%%%%%%
\subsubsection{Absorbing gadgets}

%Recall the $(r+1)$-vertex graphs we have used so far for reachability. We defined $H_0$ to be $K^{r^*}_{k,\ldots,k,q+1}$ in \Cref{Hzero}, where $q,r^*\in \mathbb{N}$ were defined so that $k(r^*-1)+q=r$ and $0\leq q \leq k$. We also endowed this graph with two identified vertices $\{w_{01},w_{02}\}$ in the part of size $q+1$. In case 1 (\Cref{case1reachability}), we then adapted this graph, defining $H_1$ to be $H_0$ with edges from $w_{01}$ to all other vetrices in the part of size $q+1$ apart from $w_{02}$. 
%In case 2 (\Cref{case2reachability}), we defined $H_2$ which is the same graph as $H_0$, except we now identify $w_{21}$ and $w_{22}$ in distinct parts of size $k$. 
%Finally, in case 3 (\Cref{case3reachability}), we used the graph $H_3=K_{\ceil{r/2},\floor{r/2}+1}$ with identified vertices $w_{31}$ and $w_{32}$ in distinct parts of this bipartite graph. We now look to build reachable gadgets using these reachable paths and we will also use the $r$-vertex graphs, $H_{\text{det}}=K^{r^*}_{k,\ldots,k,q}$ as defined in \Cref{Hzerominus}, and for case 3, the graph $H_3^-=K_{\ceil{r/2},\floor{r/2}}$ as defined in \Cref{case3reachability}. 

In this section, we will focus on larger subgraphs which we look to embed in our graph and which will be used as part of an absorbing structure. These are formed by piecing together the $\bm{H}$-paths of the previous section and the aim will be to obtain subgraphs with even more flexibility, in that they will be able to contribute to a tiling in many ways. The key definition is a graph which we call an absorbing gadget.

\begin{figure}
    \centering
  \includegraphics[scale=1]{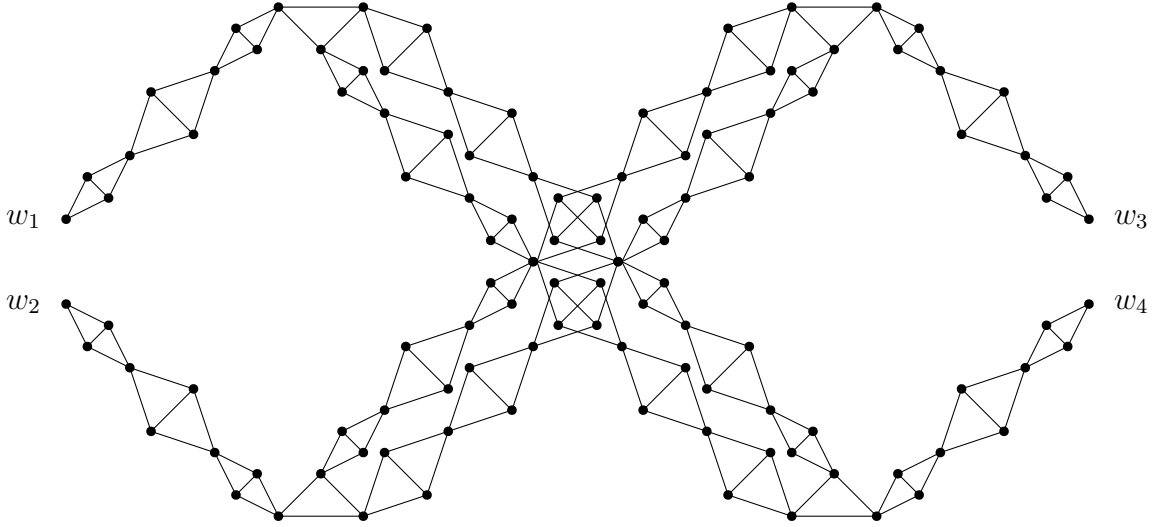}
    \caption{An $(\bm{\underline{H}},K_3)$-absorbing gadget  with $\bm{\underline{H}}:=\{\bm{H}_{i,j}:i\in[3],j\in[4]\}$ such that each $\bm{H}_{i,j}=(K_4^-,3)$. The base set of the absorbing gadget is $W:=\{w_1,w_2,w_3,w_4\}$. }
    \label{fig:Absorbing Gadget}
\end{figure}

\begin{definition} \label{absorbing gadgets}
Let $r,s\in \mathbb{N}$, let $H$ be an $r$-vertex graph and let  $\bm{\underline{H}}:=\{\bm{H}_{i,j}:i\in[r],j\in[s]\}$ be a labelled family of vectors of $(r+1)$-vertex graphs (with tuples of distinguished vertices). Then an $(\bm{\underline{H}},H)$-absorbing gadget is a graph obtained by the following procedure. Take disjoint $\bm{H}_{i,j}$-paths for $1\leq i \leq r$ and $1\leq j\leq s$ and denote their endpoints by $u_{i,j}$ and $v_{i,j}$. Place a copy of $H$ on $\{v_{i,j}:i\in [r]\}$ for each $j\in [s]$. For $2\leq i\leq r$, identify all vertices $\{u_{i,j}: 1\leq j\leq s\}$ and relabel this vertex $u_i$. Finally relabel $u_{1,j}$ as $w_j$ for $j\in [s]$ and let $W:=\{w_1,w_2,\ldots,w_s\}$, which we refer to as the \emph{base set} of vertices for the absorbing gadget.

%Let $X=\{x_1,\ldots,x_t\}\subset V(G)$ be a subset of vertices such that $|X|=t\leq 40$ and let $\cH$ be a set of $(r+1)$-vertex graphs and $\mathcal{J}$ a set of $r$-vertex graphs and $\ell\in \mathbb{N}$. A $(\cH,\mathcal{J},\ell)$-\emph{absorbing gadget} $R\subseteq V(G)$ for $X$ is a subset of vertices of $G$, disjoint from $X$ such that there is some labelling of $R$ as 
%$$R=\{y_2,y_3,\ldots,y_{r}\}\sqcup\{z_{1s},\ldots,z_{rs}:1\leq s\leq t\}\sqcup \{Y_{is}:1\leq i\leq r, 1\leq s\leq t\},$$
%such that 
%\begin{itemize}
%\item For each $1\leq s\leq t$, there is some $H^s\in \mathcal{J}$ such that $\{z_{1s},\ldots,z_{rs}\}$ hosts a copy of $H^s$.
%\item For each $1\leq s\leq t$ and $1\leq i\leq r$, there is some vector ${\bm{H}}_{is}$ of $(r+1)$-vertex graphs whose entries lie in $\cH$ such that $Y_{is}$ is a $({\bm{H}}_{is},\ell_{is})$-absorbing path between $y_{is}$ and $z_{is}$ for some $\ell_{is}\leq \ell$, where we have $y_{1s}:=x_s$ for $s\in [t]$ and $y_{is}:=y_i$ for all $i\geq 2$.
%\end{itemize}
%We say two $(\cH,\mathcal{J},\ell)$-absorbing gadgets for $X$ are isomorphic if the choices of (labelled) $H^s\in \mathcal{J}$, $\ell_{is}\in \mathbb{N}$ and ${\bm{H}}_{is}\in \cH^{\ell_{is}}$ in the above two bullet points are the same for both absorbing gadgets. We say that a vertex set $X$ is $(\cH,\mathcal{J},\ell,\beta)$-\emph{reachably compatible} if there exists some $q\in \mathbb{N}$ such that $X$ has at least $\beta n^q$ distinct isomorphic $(\cH,\mathcal{J},\ell)$-absorbing gadgets of size $q$.
\end{definition}

An example of an absorbing gadget is given in Figure \ref{fig:Absorbing Gadget}. 
Recall that we always consider $K_{r+1}^-$ to have two distinguished vertices which form the only non-edge of the graph. 
In the previous section we commented on how a $(K_{r+1}^-,t)$-path $P$ with endpoints $x$ and $y$ has two $K_r$-tilings 
covering all but one vertex; the first misses $x$, the other misses $y$.
 The point of the absorbing gadget is to generalise this property, giving a graph which can use any one of a number of vertices (the base set) in a $K_r$-tiling. In more detail, suppose $s,t^*\in\mathbb{N}$ and $\bm{\cH}=\bm{\cH}(K_{r+1}^-,\leq t^*)=\{(K_{r+1}^-,t):1\leq t\leq t^*\}$ as defined in the previous subsection. 
Let $\bm{\underline{H}}:=\{\bm{H}_{i,j}:i\in[r],j\in[s]\}$ where each $\bm{H}_{i,j}$ is an element from $\bm \cH$.
Then an $(\bm{\underline{H}},K_r)$-absorbing gadget $F$ with base set $W=\{w_1,\ldots,w_s\}$ 
%\COMMENT{AT: made the change here}
has the property that for \emph{any} $j\in[s]$, there is a $K_r$-tiling covering precisely $(F\setminus W)\cup \{w_j\}$. Indeed, %adopting the notation of Definition~\ref{absorbing gadgets}, 
we have that for all $j'\neq j$ and $i\in[r]$, there is a $K_r$-tiling of the $\bm{H}_{i,j'}$-path $P_{i,j'}$ which uses\footnote{We label all vertices in this discussion as in Definition \ref{absorbing gadgets}.} $v_{i,j'}$ and not the other endpoint of $P_{i,j'}$. Then there is a tiling of the $\bm{H}_{1,j}$-path which uses $w_j$, a tiling of the $\bm{H}_{i,j}$-path for $2\leq i\leq r$ which uses $u_i$ and a copy of $K_r$ on $\{v_{i,j}:i\in [r]\}$ which completes the desired $K_r$-tiling.  

As in the previous subsection, we begin by showing that there are many absorbing gadgets in the deterministic graph. Again, although we are interested in $(\bm{\underline{H}},K_r)$-absorbing gadgets for some $\bm{\underline{H}}$ consisting of vectors, all of whose entries are $K_{r+1}^-$, we split the edges of our absorbing gadget and rely on the deterministic graph to provide many copies of a subgraph of the gadget. In particular, we will use here our paradigm $H_{\text{det}}$, defined in Definition \ref{Hzerominus}. The following general proposition allows us to show that we can find many absorbing gadgets if all the vertices which we hope to map the base set to, are reachable to each other. 

%\begin{definition} \label{absorbing gadget}
%Let $K_{r+1}^-$ denote the complete $r+1$ vertex graph with an edge missing between two identified vertices (in the sense of \Cref{reachable paths}). Then a $(\{K_{r+1}^-\},\{K_r\},\ell)$-absorbing gadget, say $A$, for a set $X$ is called an \emph{($\ell$-bounded) absorbing gadget} for $X$. It has the following key property: For any choice of one $x\in X$, there is a $K_r$-tiling of $A\cup\{x\}$. Indeed, using the notation of \Cref{reachable gadgets}, one can tile $\cup_{i\in[r]} Y_{is}\cup\{z_{is}\}$ for each $i$ such that $x\neq x_i$ and for $j$ such that $x=x_j$, there is a tiling of $x_j\cup Y_{1j}$, of $y_{i}\cup Y_{ij}$ for $2 \leq i\leq r$ and a copy of $K_r$ on $\{z_{1j},\ldots,z_{rj}\}$. 
%\end{definition}

\begin{definition}\label{def1} Let $r,s \in \mathbb N$.
Let $\bm{\cH}$ be a finite set of vectors, such that each entry of each vector in $\bm{\cH}$ is an $(r+1)$-vertex graph with a tuple of distinguished vertices. 
We write $\bm{\cH}(r \times s)$ for the collection of all ordered labelled sets $\bm{\underline{H}}:=\{\bm{H}_{i,j}:i\in[r],j\in[s]\}$ where each $\bm{H}_{i,j}$ is an element from $\bm \cH$.
If $\bm{\cH}$ consists of a single vector $\bm H$ we write $\bm H(r \times s):=\bm{\cH}(r\times s)$. That is, $\bm H(r \times s)$ is the ordered labelled (multi)set with each element a copy of $\bm H$.
\end{definition}
%\COMMENT{AT: added definition}

\begin{proposition} \label{prop:manyabsorbinggadgets}
Let $\alpha,\gamma,\beta'>0$, $q,k,r \in \mathbb{N}$ and let $\bm{\cH}$ be a finite set of vectors, such that each entry of each vector in $\bm{\cH}$ is an $(r+1)$-vertex graph with a tuple of distinguished vertices. 

Then there exists $\beta=\beta(\alpha,\gamma,\beta',q,k,r,\bm{\cH})>0$, such that for sufficiently large $n$, if $G$ is an $n$-vertex graph with vertex subset $U\subseteq V(G)$ such that $U$ is  $(\bm{\cH};\beta')$-closed, $|U|\geq \alpha n$ and $\delta(G[U])\geq (1-\frac{k}{r}+\gamma)|U|$, then for any set $X=\{x_1,\ldots,x_s\}\subset U$ with $|X|\leq q$, there exists some $\bm{\underline{H}}\in \bm{\cH}(r \times s)$ and some $(\bm{\underline{H}},H_{\text{det}})$-absorbing gadget $F$ with base set $W=\{w_1,\ldots,w_{s}\}$ such that there are at least $\beta n^{v(F)-s}$ embeddings of $F$ in $G$ which map $w_i$ to $x_i$ for $i\in [s]$.
\end{proposition}
\begin{proof}
Firstly notice that for a fixed $s\leq q$, there is a finite number (i.e.~$|\bm \cH|^{rs}$) of $(\bm{\underline{H}},H_{\text{det}})$-absorbing gadgets $F$ such that $\bm{\underline{H}}\in \bm{\cH}(r\times s)$ and $F$ has a base set of size $s$. 
Let $\cF_s$ be the set of all such absorbing gadgets, let $f:=|\cF_s|$ and set $Q:=\max\{|F|-s: F\in \cF_s\}$. We claim that there is some $\beta''=\beta''(\alpha,\gamma,\beta',q,k,r,\bm{\cH})>0$ such that with $G$ and $U$ as in the statement of the proposition and $X\subset U$ of size $s$, there are at least $\beta''n^Q$ subsets $S\subseteq V(G)\setminus X$ of $Q$ ordered vertices such that there is an embedding of some $F\in \cF_s$ in $G$ which maps the base set of $F$ to $X$ and the other vertices to a subset\footnote{In particular, if $|F|<Q$ then not all of the vertices of $S$ are used in this embedding.} of $S$.  
%\COMMENT{PM:footnoted}
Given this claim, the conclusion of the proposition follows easily. Indeed, by averaging we get that there is some $F\in \cF_s$ and at least $({\beta''}/{f}) n^Q$ ordered subsets $S$ of $Q$ vertices in $V(G)$ as above, that correspond to an embedding of $F$.
%such that for each such $S$ there is an embedding of $F$ which maps $W$ to $X$ and the remaining vertices of $F$ into $S$. 
Then setting $\beta:={\beta''}/({Q!f})$, we get that there must be at least $\beta n^{|F|-s}$ embeddings of $F$ in $G$ which map the base set to $X$. Indeed 
for each such embedding $F'$ of $F$, the vertex set $V(F')\setminus X$ lies in at most $Q!n^{Q-(|F|-s)}$ different ordered sets of vertices $S \subseteq V(G)$.

%if there were fewer such embeddings of $F$, one could not get so many ordered subsets as above after considering all possible extensions of embeddings with isolated vertices and all possible orderings of the resulting $Q$-subsets.

\smallskip

So it remains to find these $\beta''n^{Q}$ ordered subsets $S$. We will show that $S$ can be generated in a series of steps so that every time we choose some $a$ vertices, we have $\Omega(n^a)$ choices. We will use the notation of Definition \ref{absorbing gadgets}. Firstly we select $r-1$ vertices $Y=\{y_2,y_3,\ldots,y_r\}$  in $U\setminus X$ which we can do in $\binom{|U\setminus X|}{r-1}=\Omega(n^{r-1})$ many ways. Now repeatedly find disjoint copies of $H_{\text{det}}$ in $U\setminus (X\cup Y)$ and label these $\{z_{i,j}:1\leq i\leq r, 1\leq j\leq s\}$ such that $\{z_{i,j}:1\leq i\leq r\}$ comprise a copy of $H_{\text{det}}$ for each $j\in [s]$. In order to do this we repeatedly apply Lemma \ref{lem:supersat} and the degree condition which we can take to be $\delta(G[U])\geq (1-\frac{k}{r}+\frac{\gamma}{2})|U|$ (ignoring any neighbours of vertices that have already been chosen in $S$). 
Hence there are $\Omega (n^{rs})$ choices for these copies of $H_{\text{det}}$.

Now for $2\leq i\leq r$ and $1\leq j\leq s$, we have that $y_i$ and $z_{i,j}$ are $(\bm{H}_{i,j},\beta')$-reachable for some $\bm{H}_{i,j}\in\bm{\cH}$ of length $t_{i,j}$ say. 
Thus there are $\beta' n^{rt_{i,j}-1}$ embeddings of an  $\bm{H}_{i,j}$-path $P$ in $G$ which map the endpoints of $P$ to $\{y_i,z_{i,j}\}$. We ignore those choices of embeddings of $P$ which use previously chosen vertices of $S$, of which there are $O(n^{rt_{i,j}-2})$. Similarly, for $1\leq j\leq s$, $x_{j}$ and $z_{1,j}$ are $(\bm{H}_{1,j},\beta')$-reachable for some $\bm{H}_{1,j}\in\bm{\cH}$, so select an embedding of an $\bm{H}_{1,j}$-path in $G$ which maps the endpoints to $\{x_{j},z_{1,j}\}$ and has all other vertices disjoint from previously chosen vertices. This gives an embedding of an $(\bm{\underline{H}}, H_{\text{det}})$-absorbing gadget in $G$ which maps the base set $W$ to $X$, $u_{i,j}$ to $z_{i,j}$ for $i\in [r]$, $j\in[s]$ and maps $u_i$ to $y_i$ for $i\in [r]$. Choosing unused vertices arbitrarily until we have a set $S$ of $Q$ vertices, the claim and hence the proof of the proposition are settled. 
\end{proof}

\subsection{The absorbing structure - random edges}

In this section, we will introduce the edges of $G(n,p)$ and show that $G\cup G(n,p)$ contains the absorbing structure we desire. The absorbing structure will be formed by choosing absorbing gadgets rooted on certain prescribed sets of vertices. The absorbing gadgets will be $ (\bm{\underline{H}},K_r)$-absorbing gadgets $F^*$ for some $\bm{\underline{H}}$ consisting of vectors whose entries are all $K_{r+1}^-$. In order to obtain these absorbing gadgets, we consider the absorbing gadgets of just deterministic edges which we looked at in the previous section and show that with high probability, one of these matches up with random edges to get the required subgraph $F^*$.
We begin by investigating the absorbing gadgets that we look for in the random graph. 

\subsubsection{Absorbing gadgets in the random graph}%\COMMENT{PM: Added this subsection} 
\label{Random AGs}
Recalling Definitions~\ref{Hzerominus} and \ref{absorbing gadgets}, let $\bm{\underline{H}}:=\{\bm{H}_{i,j}:i\in[r],j\in[s]\}$ be a labelled family of vectors of $(r+1)$-vertex graphs and suppose that there is an embedding $\phi$ of an $(\bm{\underline{H}},H_{\text{det}})$-absorbing gadget $F'$  in $G$ which maps the base set of the gadget to some $U\subset V(G)$, with $|U|=s$. Recalling Definition~\ref{reachable paths}, define $\bm{\overline{\underline{H}}}:=\{\overline{\bm{H}_{i,j}}:i\in[r],j\in[s]\}$. Now in order to complete this absorbing gadget $F'$ into one which has the form that we require,  we have to find a labelled embedding of an $(\bm{\overline{\underline{H}}},\overline{H_{\text{det}}})$-absorbing gadget $F$ onto the ordered vertex set $\phi(V(F'))$ in $G(n,p)$. The following lemma will be used to show that there are sufficiently many embeddings in $G(n,p)$  of the necessary $F$s defined as above. It is worth noting that as $F$ is uniquely defined by $F'$, it is in fact the way that we chose our deterministic absorbing gadgets, that guarantees the following conclusions.

\begin{figure}    
    \centering
  \includegraphics[scale=1.1]{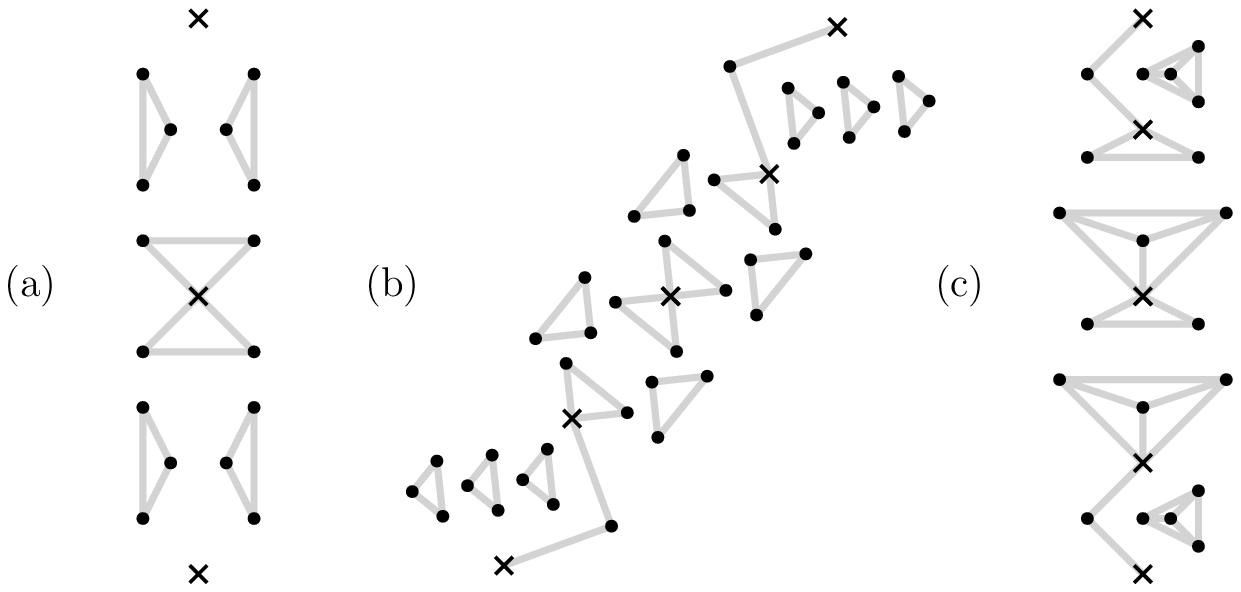}
  \captionsetup{singlelinecheck=off}
    \caption[bla]{% 
    \label{fig:Hpaths}
    Some examples of $\bm{H}$-paths for various $\bm{H}$ (see Definition \ref{reachable paths} and Propositions \ref{case1reachability}, \ref{case2reachability} and \ref{case3reachability} for relevant definitions):
    \begin{enumerate}
  \item    (a) An $\bm{\overline{H_1}}$ path with $r=9$ and $k=q=r^*=3$. 
   \item    (b) An  $\bm{\overline{H_2}}$ path with $r=11, k=3, q=2$ and $r^*=4$.
 \item      (c) An  $\bm{\overline{H_3}}$ -path  where $\bm{H_3}= (H_0,H'_0,H'_0,H_0)\in \bm{\cH_3}$ and we have  $r=6, k=4$ and  $q=r^*=2$. 
 \end{enumerate}
 }
 \end{figure}

\begin{lemma} \label{AG Phis}
Let $k,r,s\in \mathbb{N}$ and $C>1$, with $2\leq k \leq r$ and suppose $p=p(n)\geq Cn^{-2/k}$. Suppose $\bm{\underline{H}}$ is such that:
\begin{enumerate}
\item $\bm{\underline{H}}\in \bm{H_1}(r\times s)$ if $r/k\in \mathbb{N}$, recalling the definition of $\bm{H_1}$ from Proposition \ref{case1reachability};
\item $\bm{\underline{H}}\in \bm{H_2}(r\times s)$ if $r/k\notin \mathbb{N}$ and $k<r/2$, recalling the definition of $\bm{H_2}$ from Proposition \ref{case2reachability};
\item $\bm{\underline{H}}\in \bm{\cH_3}(r\times s)$ if $k>r/2$, recalling the definition of $\bm{\cH_3}$ from Proposition \ref{case3reachability}.
\end{enumerate}
Then if $F$ is an $(\bm{\overline{\underline{H}}},\overline{H_{\text{det}}})$-absorbing gadget with base set $W$ such that $|W|=s$, we have that $\Phi_{F\setminus W}\geq Cn$ and $\Phi_{F,W}\geq Cn^{1/k}$.
\end{lemma}
\begin{proof} 
We recommend that the reader refers to the examples in Figure \ref{fig:Hpaths} to help visualise some of the ideas in this proof. 
Note that as the endpoints of an  $\overline{\bm{H_1}}$-path are isolated, we have that the base set of an $(\overline{\bm{H_1}}(r\times s),\overline{H_{\text{det}}})$ absorbing gadget $F$ is also an isolated set of vertices and so $\Phi_{F,W}=\Phi_{F\setminus W}$. Defining $K_k+K_k$ as two copies of $K_k$ which meet in a singular vertex, we have that  $F\setminus W$ consists of disjoint copies of $K_k$ and $r \times s$ disjoint copies of $K_k+K_k$, one for each  $\overline{\bm{H_1}}$-path used in $F$. Therefore Lemma \ref{PhiFobservations} (1) shows that $\Phi_{K_k}\geq C n$, and repeated applications of Lemma \ref{PhiFobservations} (3)  show that $\Phi_{K_k+K_k}\geq Cn$ and in turn $\Phi_{F\setminus W}\geq Cn$ as required. 

Case 2 is similar. Here we have that  $q=r-k\lfloor r/k \rfloor <k$ and each of the base vertices $w$ of $F$ lie in a copy, say $F_w$, of the graph defined as follows.  Take a copy of $K^-_{q+1}$ and a copy of $K_k$  that meet in exactly one vertex, which is one of the vertices of the nonedge in $K^-_{q+1}$. Furthermore, we have that the base vertex $w$ is the other vertex in the nonedge  of this copy of $K_{q+1}^-$. We have that each of the $F_w$ is disconnected from the rest of $F$ and an application of Lemma \ref{PhiFobservations}~(1), (2) and $(3)$ gives that $\Phi_{F_w\setminus w}\geq Cn$ and $\Phi_{F_w,w}\geq Cn^{1/k}$ if  $q\geq 2$. If $q=1$, then $F_w$ is an isolated vertex $w$ and a copy of $K_k$ so we have $\Phi_{F_w,w}=\Phi_{F_w\setminus w}\geq Cn$. 
Now note that $F\setminus (\cup_{w\in W}F_w)$ consists of copies of $K_k$, $K^-_{q+1}$, $K_{q+1}$ and a copy of $K_q$ (in the copy of $\overline{H_{\text{det}}}$ in $F$) which intersect each other in at most one vertex. 
Furthermore, one can view  $F\setminus (\cup_{w\in W}F_w)$ as being `built up' from these copies in the following way:
there is an ordering (starting with $\overline{H_{\text{det}}}$) on these copies of $K_k$, $K^-_{q+1}$, $K_{q+1}$  and $K_q$ such that, starting with the empty graph and adding these copies 
 in this order, each new copy shares at most one vertex with the previous copies already added, and at the end of the process we obtain $F\setminus (\cup_{w\in W}F_w)$. Each time we add a copy, we can  apply  Lemma~\ref{PhiFobservations} (3) and then again to add in the $F_w$ (to obtain $F$). This leads us to conclude that $\Phi_{F\setminus W}\geq Cn$ and $\Phi_{F,W}\geq Cn^{1/k}$ as required. 

In Case 3, let $q=r-k<k$ and let us fix some $\underline{\bm{H}}\in \bm{\cH_3}(r\times s)$ which then defines our $F$. For each $w\in W$, let $F_w$ be the connected component of $F$ which contains $w$. Due to the definition of $\bm{\cH_3}$, and in particular the fact that each $\bm{H}\in \bm{\cH_3}$ contains a copy of $H_0'$ as defined in Proposition~\ref{case3reachability}, we have that $F_w\neq F_{w'}$ for all $w\neq w'\in W$. 
Also, for $q\ge 2$, it can be seen that $F_w$ is a graph obtained by sequentially `gluing' copies of $K_{q+1}^-$ to vertices of degree $q-1$ and that $w$ is a vertex of degree $q-1$ in the resulting graph. 
Similarly to the previous case, applications of Lemma \ref{PhiFobservations} (2) and $(3)$ imply that $\Phi_{F_w\setminus w}\geq Cn$ and $\Phi_{F_w,w}\geq Cn^{1/k}$ if  $q\geq 2$ and if $q=1$, we see that $F_w$ is an isolated vertex, namely $w$ itself. 
Also as before, we have that  $F\setminus (\cup_{w\in W}F_w)$ consists of copies of $K_k$, $K^-_{q+1}$, $K_{q+1}$ and a copy of $K_q$ which intersect each other in at most one vertex.
Thus, introducing the ordering of these copies as in Case 2, we can apply Lemma~\ref{PhiFobservations} repeatedly to obtain the desired conclusion. 
\end{proof}

We will use Lemma~\ref{AG Phis}  to prove the existence of our desired absorbing gadgets in $G'=G\cup G(n,p)$.
Before embarking on this however, we need to know how we wish our absorbing gadgets (in particular their base sets) to intersect in $G'$. This is given by the notion of a \emph{template} in the following subsection.

\subsubsection{Defining an absorbing structure}

A \emph{template} $T$ with \emph{flexibility} $m\in \mathbb{N}$ is a
bipartite graph on $7m$ vertices with vertex classes $I$ and
$J_1\sqcup J_2$, such that $|I|=3m$, $|J_1|=|J_2|=2m$, and for any
$\overline{J}\subset J_1$, with $|\overline{J}|=m$, the induced graph
$T[V(T)\setminus\overline{J}]$ has a perfect matching. We call $J_1$ the
\textit{flexible} set of vertices for the template. Montgomery first
introduced the use of such templates when applying the absorbing
method in his work on spanning trees in random graphs~\cite{M14a,M19}.
There, he used a sparse template of maximum degree~$40$, which we will
also use. It is not difficult to prove the existence of such
templates for large enough~$m$ probabilistically; see 
e.g.~\cite[Lemma~2.8]{M14a} . The idea has since be used by various authors in different settings \cite{FKL16,FN17,Kwan16,HKMP18,nenadov2018ramsey,han2019finding}. %\COMMENT{PM: added the recent paper of Jie , Yoshi, Yury and me}. 
We will use a template here as an auxiliary graph in order to build an absorbing structure for our purposes. 

\begin{definition} \label{def: absorbing structure}
Let $m, t^*\in \mathbb{N}$ and $T=(I=\{1,\ldots,3m\},J_1\sqcup J_2=\{1,\ldots,2m\}\sqcup\{2m+1,\ldots,4m\}, E(T))$ 
be a bipartite template with maximum degree $\Delta(T)\le 40$ and flexibility $m$ as defined above.  Further, let 
\[\bm{\cH}:=\bm{\cH}(K_{r+1}^-,\leq t^*)=\{(K_{r+1}^-,t):1\leq t\leq t^*\}\] 
be the set of vectors of length at most $t^*$ whose entries are all $K_{r+1}^-$. 

A ($t^*$-bounded) \emph{absorbing structure} $\cA=(\Phi,Z,Z_1)$ of flexibility $m$ in a graph $G'$ consists of a vertex set $Z=Z_1\sqcup Z_2\subset V(G')$ which we label $Z_1:=\{z_1,\ldots,z_{2m}\}$ and $Z_2:=\{z_{2m+1},\ldots,z_{4m}\}$ %such that $|Z_1|=|Z_2|=2m$ 
and a set $\Phi:=\{\phi_1,\ldots,\phi_{3m}\}$ of  embeddings of absorbing gadgets into $G'$. 
We require the following properties:
\begin{itemize}
    \item For $i\in [3m]$, setting $N(i):=\{j:(i,j)\in E(T) \subset I\times J\}$ and $n(i):=|N(i)|$, we have that $\phi_i$ is an embedding of some $(\bm{\underline{H}},K_r)$-absorbing gadget $F_i$ such that 
		$\bm{\underline{H}}\in \bm{\cH}(r \times n(i))$ and the base set of $F_i$, which we denote $W_i$, is mapped to $\{z_j:j\in N(i)\}\subseteq Z$ by $\phi_i$.
    \item The embeddings of the absorbing gadgets are vertex disjoint other than the images of the base sets. That is, for all $i\in[3m]$, $\phi_i(V(F_i)\setminus W_i)\subseteq V(G')\setminus Z$ and $\phi_i(V(F_i)\setminus W_i)\cap \phi_{i'}(V(F_{i'})\setminus W_{i'})=\emptyset$ for all $i\neq i'\in [3m]$.
\end{itemize}
We call $Z_1$ the \emph{flexible set} of the absorbing structure.
\end{definition}
Thus the absorbing structure is an embedding of a larger graph which is formed of $3m$ disjoint absorbing gadgets whose base vertices are then identified according to a template of  flexibility $m$. We will refer to the vertices of $\cA$ which are the vertices which feature in the embedding of this larger graph. That is,  \[V(\cA):=\bigsqcup_{i\in [3m]}\phi(V(F_i)\setminus W_i)\bigsqcup Z.\]

 %we have that for $i\in[3m]$, that $Y_i$ is an embedding of an $(\bm{\underline{H}}(i),K_r)$-absorbing gadget such that $\bm{\underline{H}}(i)\in \bm{\cH}$

%. We have that for each $i\in[3m]$, $Y_i$ is an absorbing gadget for
%$N_T(x_i)\subseteq Z$ in $G'$ as defined in \Cref{absorbing gadgets}. We also require that $\sqcup_{Y\in\mathcal{W}}V(Y)$ is disjoint from $Z$.   
\begin{rem} \label{boundy}
If $\mathcal{A}$ is a $t^*$-bounded absorbing structure of flexibility $m$, then it has less than $125t^* r^2m$ vertices in total. 
\end{rem}
%\COMMENT{AT: added remark so easy to cite}
 
 In our proof, we will bound $t^*$ by a constant and look for an absorbing structure on a small linear number of vertices. The key property of the absorbing structure is that it inherits the flexibility of the template that defines it, but in the context of $K_r$-tilings, as detailed in the following remark.
 \begin{rem} \label{keyabsorbingproperty}
 If $G'$ contains an absorbing structure $\cA=(\Phi,Z,Z_1)$ of flexibility $m$, then for any subset of vertices $\overline{Z}\subset Z_1$ such that $|\overline{Z}|=m$, there is a $K_r$-tiling in $G'$ covering precisely  $V(\cA)\setminus \overline{Z}$.
 \end{rem}
 %has a similar property to the template that defines it. Namely, for any set $\overline{Z} \subset Z_1\subset V(G')$ such that $|\overline{Z}|=m$, there is a perfect $K_r$-tiling of $V(\mathcal{A})\setminus \overline{Z}$ in $G'$. 
 Indeed given such a $\overline{Z}$, letting $\overline{J}$ be the corresponding indices from $J$, we have that  $T\setminus \overline{J}$ has a  perfect matching. The matching then indicates, for each $i\in[3m]$, which vertex $z_{j_i}$ of $Z$ to use in a tiling of the corresponding absorbing gadget. That is, for each $i$, if $\phi_i$ is `matched' to $z_{j_i}$ by the perfect matching, then we take the $K_r$-tiling covering $\phi_i(F_i\setminus W_i)\cup \{z_{j_i}\}$ 
%\COMMENT{AT: added $\setminus W_i$ here}
(which exists by the key property of the absorbing gadget mentioned after Definition~\ref{absorbing gadgets}) and then take their union.
 
 \subsubsection{The existence of an absorbing structure}
 
 In order to prove the existence of an absorbing structure, we must find embeddings of absorbing gadgets in our graph. In the previous section we found many embeddings of certain absorbing gadgets with deterministic edges and thus it remains to find embeddings 
of complementary absorbing gadgets, using only random edges. Therefore we will turn to Lemma~\ref{lm:gnpfindingsubgraph}, which is a general result regarding embeddings in random graphs. However, there is still some work to do in the application of this lemma and the following proposition shows how we can use Lemma~\ref{lm:gnpfindingsubgraph} repeatedly in order to embed a larger graph. 
We state the proposition in a more general form than just for showing the existence of absorbing structures as we will also use the result at other points in the proof. As the statement of the proposition is somewhat technical, we recommend that the reader sees how it is applied in Corollaries \ref{cor:case1+2absorbingstructure}, \ref{cor:case3absorbingstructure} and 
%of Theorem~\ref{thm:Upper}.
\ref{cor:embeddingK_rs} to help with digesting it.

\begin{proposition} \label{prop:embeddinglargegraphs}
%\COMMENT{AT: TO DO. We are applying Lemma~\ref{PhiFobservations} inside the proof of this proposition. So I guess we will need $k \geq 2$ and e.g. $p=p(n)> 2n^{-2/k}$ added to the hypothesis of the proposition. PM: Hmmm, yea it certainly doesn't harm to put that condition in but I don't think we actually use it. We only use part (iii) of Lemma~\ref{PhiFobservations} whereas those conditions are only needed for part $(i)$ and $(ii)$, right?}
%Let $\kappa_d,\kappa_w,\kappa_e,\kappa_v, k\in \mathbb{N}$, $\eta,\beta,c>0$ and $p=p(n),t=\eta n$ such that $2^{(\kappa_v+2\kappa_e+9)}(\kappa_v+2\kappa_e+\kappa_w)(\kappa_v)!\eta\leq \beta^2 c$ 
Let $\kappa_d,\kappa_w,\kappa_e,\kappa_v, k, \in \mathbb{N}$ and $\beta>0$. % such that $2^{(\kappa_v+2\kappa_e+9)}(\kappa_v+2\kappa_e+\kappa_w)(\kappa_v)!\eta\leq \beta^2 c$ 
%\COMMENT{AT: doesn't look quite right when applying Lemma~\ref{lm:gnpfindingsubgraph} later on. Probably $2^{(K_v+2K_e+7)}(K_v+2K_e+K_w)(K_v)!\eta\leq \beta^2 c$ works}
Then there exists $\eta_0>0$ and $C>0$ such that the following holds for any $0<\eta<\eta_0$, $n\in \mathbb{N}$ and $t=\eta n\in \mathbb{N}$. 

Suppose that $F_1,\ldots,F_t$ are labelled graphs with distinguished base vertex sets $W_i\subset V(F_i)$ such that $|W_i|\leq \kappa_w$, $v_i:=|V(F_i)\setminus W_i|\leq \kappa_v$, $e(F_i[W_i])=0$ and $e(F_i)\leq \kappa_e$ for all $i\in[t]$. 
Suppose that $p=p(n)$ such that $\Phi_{F_i\setminus W_i}=\Phi_{F_i\setminus W_i}(n,p)\geq Cn$ and $\Phi_{F_i,W_i}=\Phi_{F_i,W_i}(n,p)\geq Cn^{\frac{1}{k}}$ for all $i \in [t]$. Let $V$ be an $n$-vertex set, and $U_1,\ldots,U_t\subset V$ be subsets such that $|U_i|=|W_i|$ for each $i\in[t]$, and defining \[d(i):=|\{j\in[t]:U_i\cap U_j\neq \emptyset\}|,\]
we have that $d(i)\leq \kappa_d$. Finally, suppose that $\cF_1,\cF_2,\ldots, \cF_t$ are families of vertex sets such that each $\cF_i$ contains $\beta n^{v_i}$ ordered subsets of $V$ of size $v_i$. 

Then 
a.a.s.
 there is a set of embeddings $\phi_1,\phi_2,\ldots,\phi_t$ such that each $\phi_i$ embeds a copy of $F_i$ into $G(n,p)$ on $V$ with $W_i$ being mapped to $U_i$ and $V(F_i)\setminus W_i$ being mapped to a set in $\cF_i$ which does not intersect $\cup_{i\in[t]}U_i$. Furthermore for $i\neq i'$, we have that $\phi_i(V(F_i)\setminus W_i)\cap\phi_{i'}(V(F_{i'})\setminus W_{i'})=\emptyset$.
\end{proposition}

\begin{proof}
Fix $\kappa_v':=\kappa_v+2\kappa_e$ and let $\eta_0<\beta^22^{-(\kappa'_v+2\kappa_e+9)}{(\kappa'_v!(\kappa'_v+\kappa_e+\kappa_w))^{-1}}$. The idea here is to greedily extract the desired embeddings, finding them one at a time in $G(n,p)$. %\COMMENT{JH: I commented out a few sentences. In particular, I wouldn't say the multi-round exposure is necessary (or, equivalently, one-round is not possible) -- maybe some more careful approach e.g. the greedy embedding as in the blow-up lemma would work.}
%In order to do we would like a statement that guarantees that no matter which indices are left to embed, there exists a valid embedding amongst this set of indices. 
%It turns out that such a statement is not possible. 
%That is to say, we cannot naively apply a concentration result to guarantee that an embedding exists at all points in this process. 
To achieve this, we use the multi-round exposure trick, having a constant number of phases such that in each phase we find a collection of embeddings. At the beginning of each phase we `reveal' another copy of $G(n,p)$ on the same vertex set and focus \emph{only} on the indices for which we have not yet found a suitable embedding, showing that in any sufficiently large subset of these indices there is an index for which we can find a suitable embedding. 
At each phase, we will apply Lemma~\ref{lm:gnpfindingsubgraph} and so we first need to slightly adjust the sets we are considering in order to be in the setting of that lemma.   

Firstly let us adjust each $F_i$ so that it has $\kappa_v'$ non-base vertices and $\kappa_e$ edges. To each $F_i$ add $\kappa_e-e(F_i)$ isolated edges. Then 
%let $K_v'=:\max \{v(F_i')- v(W_i):i\in[t]\}$ and to each $F_i'$ 
add isolated vertices until the resulting graph has $\kappa_{v}'+|W_i|$ vertices and redefine $F_i$ as the resulting graph. Note that if $p=p(n)$ is such that  $\Phi_{F_i\setminus W_i}(n,p)\geq Cn$ and $\Phi_{F_i,W_i}(n,p)\geq Cn^{\frac{1}{k}}$ for the original $F_i$ as in the statement of the proposition, then these conditions are preserved under the above changes to $F_i$ for each $i$, by Lemma~\ref{PhiFobservations}. We also arbitrarily extend each set in each $\cF_i$ to get sets of size $\kappa_v'$. As we can extend with any vertices not already in the set, it can be seen that we can have families $\cF_i$ of size at least $\beta'n^{\kappa'_v}$ for some  $\beta'>\beta/(2^{\kappa_v'})$ which we now fix. Clearly, a set of valid embeddings of these new $F_i$ 
(where the new vertices of $F_i$ are mapped to the new vertices from a set in $\cF_i$\footnote{This will be guaranteed in applications of Lemma~\ref{lm:gnpfindingsubgraph} as the lemma is concerned with \emph{labelled} embeddings.})
will also yield a set of embeddings of the original graphs we were interested in. 

Now let us turn to the phases of our algorithm. We will generate $G(n,p)$ in $k+1$  rounds so that $G(n,p)=\cup_{j=1}^{k+1}G_j$ with each $G_j$ an independent copy of  $G(n,p')$, where $p'$ is such that $(1-p)=(1-p')^{k+1}$. Note that for any graph $F$,  vertex subset $W\subset V(F)$, constant  $c>0$ and probability $p$, one has that   $\Phi_{F\setminus W}(n,cp)=c'\Phi_{F\setminus W}(n,p)$ for some constant $c'=c'$ between $1$ and $c^{e(F)}$. Likewise, multiplication of the probability by some constant $c>0$ results in multiplication of $\Phi_{F,W}$ by some constant factor. Hence, choosing $C>0$ sufficiently large, we can guarantee that if $\Phi_{F_i\setminus W_i}(n,p)\geq Cn$ and $\Phi_{F_i,W_i}(n,p)\geq Cn^{\frac{1}{k}}$ as in the statement of the proposition, then $\Phi_{F_i\setminus W_i}(n,p')\geq C'n$ and $\Phi_{F_i,W_i}(n,p')\geq C'n^{\frac{1}{k}}$ with $C'$   such that   $C'\geq\frac{2^{k_v'+9}\kappa_v'!\kappa_v'}{\beta'^2}$. % and we can express $G(n,Cp)=\cup_{j=1}^{k+1}G_j$ with each $G_j$ a copy of  $G(n,C'p)$. 
We fix such a $C>0$ and for $j=1,\ldots,k$, we define $t_j:=\eta n^{1-\frac{j-1}{k}}((\kappa_d+1)\log n)^{j-1}$ and $s_j:=t_j n^{-\frac{1}{k}}\log n=\eta n^{1-\frac{j}{k}}(\kappa_d+1)^{j-1}(\log n)^{j}$. We also define $t_{k+1}:=(\kappa_d+1)s_k=\eta ((\kappa_d+1)\log n)^k$, $s_{k+1}:=1$ and $t_{k+2}:=0$.

Now, as discussed, we look to choose embeddings one by one in order to reach the desired conclusion. Therefore, for the sake of brevity, at any point in the argument let us say that an embedding $\phi_i$ of $F_i$ is \emph{valid} if it maps $W_i$ to $U_i$ and maps $V(F_i\setminus W_i)$ to a set in $\cF_i$ which is disjoint from $U:=\cup_{i\in [t]}U_i$ and also disjoint from $\phi_{i'}(V(F_{i'}\setminus W_{i'}))$ for all indices $i'\in[t]$ for which we have already chosen an embedding.  Our claim is that a.a.s. (with respect to $G(n,p)$) we can repeatedly choose valid embeddings until we have found embeddings for all $t$ indices in $T:=\{1,\ldots,t\}$. We therefore need to show that we never get stuck and that this greedy algorithm always finds a valid embedding. In order to do this, we split the algorithm into $k+1$ phases and rely on the edges of $G_j$ in the $j^{th}$ phase where we will find $t_j-t_{j+1}$ valid embeddings. We will show that for all $j\in[k+1]$, conditioned on the fact that the algorithm has succeeded so far, we have that a.a.s (with respect to $G_j=G(n,p')$) the algorithm will succeed for a further phase. The conclusion then follows easily as there are constantly many phases.

So let us analyse the $j^{th}$ phase and  condition on the fact that the process has been successful so far and so there are $t_j$ indices that remain for us to find embeddings for. Let us further fix a specific set\footnote{Note that when $j=1$ we must have that $T_1=T$.} of $t_j$ indices $T_j\subseteq T$ that remain and some set of already chosen valid embeddings $\{\phi_{i}:i\in R_j\}$ where  $R_j:=T\setminus T_j$. %We also have. 
%such that $\phi_i$ is an embedding of $F_i$ in $\cup_{j'=1}^{j-1}G_{j'}$  with $W_i$  mapped to $U_i$ and $V(F_i)\setminus W_i$ mapped to a set in $\cF_i$ as required. 
%Let $U:=\cup_{i\in[t]}U_i$ and 
By the law of total probability,  it suffices to condition on this fixed set of embeddings so far and show that a.a.s. (with respect to $G_j$) we can repeatedly find valid embeddings, each time removing the corresponding index from $T_j$, until there are $t_{j+1}$ indices remaining.  So let $V''_j:=\cup_{i\in R_j}\phi_i(V(F_i))\cup U$ and for $i\notin R_j$, define $\cF^{(j)}_i:=\{S\in \cF_i:S\cap V''_j=\emptyset\}$. We have that $|\cF_i^{(j)}|\geq \frac{\beta'}{2}n^{\kappa'_v}$ as $| V''_j|<\frac{\beta'}{2}n$ due to our condition on $\eta_0$. We then apply Lemma \ref{lm:gnpfindingsubgraph} to the sets $\cF_{i}^{(j)}$ such that $i\in T_j$, 
and where $t_j,s_j, \frac{\beta '}{2},\kappa'_v,\kappa'_v,\kappa_w,\kappa_e$ and $p'$ play the role of $t,s, \beta, L,v,w,e$  and $p$ respectively. Let us check that the conditions needed for the lemma are  satisfied. Indeed, we certainly have that $\kappa'_v t_j\leq\kappa'_v t \leq \frac{\beta'n}{8\kappa_v'}$, $s_j\kappa_w \leq \frac{\beta'n}{8\kappa_v'}$ and  $\binom{t_j}{s_j}\leq \binom{n}{s_j}\leq 2^n$.  Moreover, when $1\leq j \leq k$,  
\[C's_jn^{\frac{1}{k}}=C't_j\log n \geq \left(\frac{2^{k_v'+9}\kappa_v'!\kappa_v'}{\beta'^2}\right)t_j\log n \mbox{ and }C'n \ge \left(\frac{2^{\kappa_v'+9}\kappa_v'!}{\beta'^2}\right)n, \]
by our definition of $C'$ whilst for $j=k+1$, $C's_{k+1}n^{\frac{1}{k}}=\omega(t_{k+1}\log n)$. This verifies the conditions in \eqref{eq1} in all cases and so we conclude that a.a.s., given any set $V'_j$ of at most $\kappa_v't_j$ vertices and any set $S_j$ of $s_j$ indices in $T_j$ such that the sets $U_i$ with $i\in S_j$ are pairwise disjoint, there is an index $i^*\in S_j$ and a valid embedding of $F_{i^*}$ in $G_j$ which avoids $V'_j$. This then implies that the greedy process will suceed throughout this phase. Indeed, we can now  initiate with $V_j'=\emptyset$ and repeatedly find indices $i\in T_j$ for which we have a valid embedding $\phi_i$. We add this embedding to our chosen embeddings, add the vertices of it to $V_j'$ and delete the index $i$ from $T_j$. The conclusion that we drew from Lemma \ref{lm:gnpfindingsubgraph} above asserts that we continue this process until we have $t_{j+1}$ indices remaining in $T_j$, which is precisely what we need. Indeed, for $1\leq j\leq k$, if we have more than $t_{j+1}$ indices in $T_j$ left then by the upper bound on $d(i)$ for $i$ in $T_j$ taking a maximal set $S\subset T_j$ such that $U_i$ are all pairwise disjoint for $i\in S$, we have that $|S|\geq t_{j+1}/(\kappa_d+1)\geq s_j$. In the final phase when $j=k+1$ we can simply find embeddings one at a time as $s_{k+1}=1$. This concludes the proof. %  and we can use the conclusion of the lemma in $G_j$ to find a valid embedding for an index in $S$.
\end{proof}

As corollaries, we can conclude the existence of absorbing structures in $G\cup G(n,p)$. We split the cases here as Case~1 and 2 are much simpler. 

\begin{cor} \label{cor:case1+2absorbingstructure}
Let $k,r\in\NN$ such that either $2\leq k \leq r/2$ or $k=r$ and let $\gamma>0$. There exists $\eta_0>0$ and $C>0$ such that if $p\geq C n^{-2/k}$ and $G$ is an $n$-vertex graph with minimum degree $\delta(G)\geq (1-\frac{k}{r}+\gamma) n$, then for any $0<\eta<\eta_0$ and any set of $2\eta n$ vertices $X_1\subseteq V(G)$,  a.a.s. %\COMMENT{AT: so given a fixed $X_1$ we can find an absorbing structure a.a.s. However, this doesn't say a.a.s we can find an absorbing structure for any such set $X_1$.}
there exists a $4$-bounded absorbing structure $\cA=(\Phi,Z,Z_1)$ in $G':=G\cup G(n,p)$ of flexibility $m:=\eta n$, which has flexible set $Z_1=X_1$.
\end{cor}

\begin{proof}%\COMMENT{AT: quite a few  changes in this proof}
%The proof is fairly immediate from previous results. 
We look to apply Proposition~\ref{prop:embeddinglargegraphs} and simply need to establish the hypothesis of the proposition. 
Consider a bipartite template $T=(I=\{1,\ldots,3m\},J_1\sqcup J_2=\{1,\ldots,2m\}\sqcup\{2m+1,\ldots,4m\}, E(T))$  as in Definition~\ref{def: absorbing structure}; recall such a template exists~\cite{M14a}.
Fix $Z_1=X_1=\{z_1,\ldots, z_{2m}\}$ and choose an arbitrary set of $2m$ vertices $Z_2\subset V(G)\setminus Z_1$ which we label $\{z_{2m+1},\ldots,z_{4m}\}$. Now towards applying Proposition~\ref{prop:embeddinglargegraphs}, we set $t:=3m$ and for $i \in [t]$ we define the sets $U_i:=\{z_j:j\in N(i)\}$ where $N(i)$ is as in Definition~\ref{def: absorbing structure}. Note that we can set $\kappa_d:=1600$ as we start with a template $T$ with  $\Delta(T) \leq 40$, so for any set $N(i)\subset J$ (of at most $40$ vertices), there are at most $1600$ indices $i'\in I=[3m]$ such that $N(i')\cap N(i)\neq \emptyset$. 

Now, fixing $i$, the collection $\cF_i$, which we will use when applying Proposition~\ref{prop:embeddinglargegraphs}, will be obtained from Proposition~\ref{prop:manyabsorbinggadgets}. Indeed, this proposition implies, along with Propositions~\ref{case1reachability} and~\ref{case2reachability}, that there is some $\beta>0$ such that the following holds with $a=1$ if $r/k\in\mathbb{N}$ (Case 1) and $a=2$ otherwise (Case 2). 
\begin{claim}\label{C3P0}
For any set $U$ of at most $40$ vertices, there is an $(\bm{H}_a (r \times |U|),H_{\text{det}})$-absorbing gadget\footnote{Recall here the definition of $\bm{H}_1$ from Proposition~\ref{case1reachability}, of $\bm{H}_2$ from 
Proposition~\ref{case2reachability} and $H_{\text{det}}$ from Definition~\ref{Hzerominus}.
The notation $\bm{H}_a (r \times |U|)$ is also defined as in Definition~\ref{def1}.}
 $F'$ such that there are at least $\beta n^{|F'|-|U|}$ embeddings of $F'$ in $G$ which map the base set of the absorbing gadget to $U$. 
\end{claim}
For each $i$, apply Claim~\ref{C3P0} with $U_i$ playing the role of $U$ to obtain a collection $\mathcal F_i$ of ordered vertex sets from $V(G)$ that combined with $U_i$ each span such an absorbing gadget $F'_i=F'$.
For each such embedding of $F'_i$, if we have an ordered  $(\overline{\bm{H}_a}(r \times |U_i|),\overline{H_{\text{det}}})$-absorbing gadget $F_i$ (in $G(n,p)$), 
 on the same vertex set, then we obtain the desired embedding $\phi_i$ of a 
$(\bm{K}_a (r\times |U_i|),K_r)$-absorbing gadget
 in $G\cup G(n,p)$, where $\bm{K}_a$ is a $(K_{r+1}^-,2a)$-path. Applying Proposition~\ref{prop:embeddinglargegraphs} with small enough $\eta>0$  thus gives us the absorbing structure, upon noticing that the conditions on $\Phi_{F_i,W_i}$ and $\Phi_{F_i\setminus W_i}$ are satisfied by Lemma \ref{AG Phis}. %repeated applications of Lemma~\ref{PhiFobservations}. 
%\COMMENT{TO DO: AT: perhaps we should have a separate lemma that says that the conditions on $\Phi_{F_i,W_i}$ and $\Phi_{F_i\setminus W_i}$ hold? We could explain in each of the three cases that we define the absorbing structures in $G$ differently so that their
%`complementary' structures are plentiful in $G(n,p)$. I know things are mentioned earlier in the paper, but I still think at the moment the reader may not really get a feel for why we use different structures in the 3 cases. PM: added, see Section \ref{Random AGs}.}
\end{proof}

The third case, when $r/2<k\leq r-1$, follows the exact same method of proof. %and so we omit the details. 
The main difference comes from the fact that we do not have many absorbing gadgets for \emph{all} small sets of vertices in the deterministic graph but only for sets which lie in one part of the partition dictated by Lemma \ref{case3reachability}. Therefore  we look to find an absorbing structure in each part of the partition. Thus when we apply Proposition \ref{prop:embeddinglargegraphs}, we do so to find all these absorbing structures at once, in order to guarantee that these absorbing structures are disjoint. The conclusion is as follows. 

\begin{cor} \label{cor:case3absorbingstructure}
Let $r/2< k \leq r-1$ be integers, and define $q:=r-k$, $c:=\ceil{r/q}$ and $\gamma>0$. Then there exists $\alpha>0$ such that the following holds for all $0<\eps<\alpha\gamma/4$. There exists $C=C(r,k,\gamma,\eps)>0$ and $\eta_0=\eta_0(r,k,\gamma,\eps)>0$ such that if $p\geq C n^{-2/k}$ and $G$ is an $n$-vertex graph with minimum degree $\delta(G)\geq (1-\frac{k}{r}+\gamma) n$, then for any $0<\eta<\eta_0$ there is a partition $\cP=\{V_1,V_2,\ldots,V_\rho, W\}$ of $V(G)$ into at most $c$ parts with the following properties:
\begin{itemize}
    \item $|V_i|\geq \alpha n$ for $i\in [\rho]$;
    \item $|W|\leq \eps n$;
    \item $\delta({G[V_i]})\geq (1-\frac{k}{r}+\frac{\gamma}{4})|V_i|$ for each $i\in [\rho]$;
    \item For any collection of subsets $X_i\subset V_i$ such that $1.8 \eta |V_i|\leq |X_i| \leq 2\eta |V_i|$ for all $i\in[\rho]$, there a.a.s. exists a set of  $c(2^{c+1}+1)$-bounded absorbing structures $\{\cA_i=(\Phi_i,Z_i,Z_{i1}):i\in[\rho]\}$ in $G':=G\cup G(n,p)$ such that each $\cA_i$ has  flexibility $m_i:=|X_i|/2$ and has flexible set $Z_{i1}=X_i$. Furthermore $V(\cA_i)\cap V(\cA_{i'})=\emptyset$ for all $i\neq i'\in [\rho]$.
\end{itemize}
\end{cor}
\begin{proof}
We begin by applying Proposition \ref{case3reachability}  to get a vertex partition $\cP$ with at most $c-1$ parts and in each part $U\in \cP$ we remove any vertex $v$ which has internal degree $d_U(v)=|N_G(v)\cap U|<(1-\frac{k}{r}+\frac{\gamma}{2})|U|$, and add $v$ to $W$. The resulting partition is the partition we will use. Choosing  $\eps_{\ref{case3reachability}}$ in the application of Proposition~\ref{case3reachability} to be less than $\eps/c$, we have that the first three bullet points are satisfied. 
Below we show the last bullet point, and to aid readability we temporarily fix $i=1$.

Now given a set of $X_1\subset V_1$ we choose a set $Z_2\subset V_1\setminus X_1$ such that $|Z_2|=2m_1$. 
Further, according to some template $T=(I=\{1,\ldots,3m_1\},J_1\sqcup J_2=\{1,\ldots,2m_1\}\sqcup\{2m_1+1,\ldots,4m_1\}, E(T))$  as in Definition~\ref{def: absorbing structure}, we label $X_1$ according to $J_1$ and $Z_2$ according to $J_2$ and identify sets $U_{i'}\subseteq X_1$ for each $i'\in [3m_1]$ according to the neighbourhood of $i'$ in $T$. 
As in Corollary \ref{cor:case1+2absorbingstructure}, by Propositions~\ref{case3reachability} and~\ref{prop:manyabsorbinggadgets} there exists some $\beta>0$ such that for each $i'\in [3m_1]$, fixing $s_{i'}=|U_{i'}|$ the following holds. 
There is some $\underline{\bm{H}}_{i'}\in \bm{\cH_3}(r \times s_{i'})$ and some $(\underline{\bm{H}}_{i'}, H_{\text{det}})$-absorbing gadget $F'_{i'}$ such that there are at least $\beta n^{|F'_{i'}|-s_{i'}}$ embeddings of $F'_{i'}$ in $G$ which map the base set of $F'_{i'}$ to $U_{i'}$. 
Each of these embeddings gives a candidate vertex set for which we could embed an $(\overline{\underline{\bm{H}}_{i'}}, \overline{H_{\text{det}}})$-absorbing gadget, say $F_{i'}$ %with base set $W_{i,i'}$
to get a copy of a $(\underline{\bm{K}},K_r)$-absorbing gadget in $G'$, with base set $U_{i'}$, where $\underline{\bm{K}}\in \bm{\cH}(r\times s_{i'})$ and $\bm{\cH}=\bm{\cH}(K_{r+1}^-,\leq c(2^{c+1}+1))$. Using Lemma \ref{AG Phis},  
%to establish that $\Phi_{F_{i,i'}, W_{i,i'}}\geq Cn^{1/k}$ and $\Phi_{i,i'}_{F_{i,i'}\setminus W_{i,i'}}\geq Cn$ for each $i$ and $i'$, 
we can now apply Proposition \ref{prop:embeddinglargegraphs} (provided $\eta>0$ is sufficiently small) to get the desired embeddings of all the $F_{i'}$ which an absorbing structure $\cA_1$ as in the statement of the corollary. We in fact apply Proposition \ref{prop:embeddinglargegraphs} for all $i\in[\rho]$ at once which gives the collection of absorbing structures as required.   
\end{proof}
%\COMMENT{TO DO: I think there should at least be a proof sketch of this corollary so we can state which results we combine to prove this. PM: Added}
Before proving the upper bound in our main result, Theorem~\ref{thm:Main}, we give one last consequence of Proposition~\ref{prop:embeddinglargegraphs} which will be useful for us.

\begin{cor} \label{cor:embeddingK_rs}
Suppose that $2 \leq k\leq r$ and $\gamma, \beta>0$. Then there exists $\alpha=\alpha(r,k,\gamma,\beta)>0$ and $C>0$ such that the following holds. Suppose $G$ is an $n$-vertex graph with disjoint vertex sets $U, W$ such that $|U|\leq \alpha n$, 
%\COMMENT{AT: since for our applications we often have that $|U|\leq \alpha n$, might be clearly for the reader if we state that explicitly here}
$|W|\geq \beta n$ and for all $v\in U\cup W$, $|N_G(v)\cap W|\geq (1-\frac{k}{r}+\gamma)|W|$ and $p=p(n)$ is such that  $p\geq C n^{-2/k}$. Then a.a.s. in $G\cup G(n,p)$ there is a set of $|U|$ disjoint copies $K_r$ so that each copy of $K_r$ contains a vertex of $U$ and $r-1$ vertices of $W$.  
\end{cor}
\begin{proof}
Firstly, let $r^*:=\ceil{r/k}$.  By the fact that $|N_G(v)\cap W|\geq (1-\frac{k}{r}+\gamma)|W|$ for all $v \in U \cup W$, we have that each vertex $u\in U$ is in at least $(\frac{\beta\gamma}{2}n)^{r^*}$ distinct copies of $K_{r^*+1}^-$ in $G$ such that the  other vertices of each copy lie in $W$, and $u$ is contained in the nonedge of each $K_{r^*+1}^-$.  Thus by Lemma \ref{lem:supersat}, there exists some $\beta'>0$ such that each $u\in U$ is in $\beta'n^{r-1}$ copies of $H_{\text{det}}$ with the other vertices of each copy in $W$, and $u$ in the part of size $q:=r-(r^*-1)k$ in $H_{\text{det}}$. 
Let $\cF_u$ be the collection of $(r-1)$-sets of vertices in $W$ that, together with $u$, give rise to these copies of 
$H_{\text{det}}$ containing $u$. Set $F_u:=\overline{H_{\text{det}}}=K_r-E({H_{\text{det}}})$ with an identified  vertex $w_u$ in the clique of size $q$ in $\overline{H_{\text{det}}}$. Thus an ordered embedding in $G(n,p)$ of $F_u$ which maps $w_u$ to $u$ and $V(F_u)\setminus \{w_u\}$ to an ordered set in $\cF_u$ will give an embedding of $K_r$ in $G\cup G(n,p)$ containing $u$ and vertices of $W$. By Lemma \ref{PhiFobservations} we have that $\Phi_{F_u,{w_u}}\geq Cn^{1/k}$ and $\Phi_{F_u\setminus w_u}\geq Cn$. Thus, provided $\alpha>0$ is sufficiently small, an application of Proposition \ref{prop:embeddinglargegraphs} gives the desired set of embeddings of $K_r$ in $G\cup G(n,p)$.
\end{proof}

\section{Proof of the upper bound of Theorem~\ref{thm:Main}}\label{sec:proof}
In this section we prove the upper bound of Theorem \ref{thm:Main}. %which %we restate here in the following form. 
%\COMMENT{again, I think we want to actually prove this theorem in $k=r$ case too}
%\begin{theorem} \label{thm:Upper}
%Let $2\leq k\leq r$ be integers and $\gamma>0$. Then there exists $C=C(\gamma,k,r)>0$ 
%\COMMENT{AT: changed to $C=C(\gamma,k,r)>0$  as $C$ does depend on $k$ and $r$. Also deleted `$n_0\in \mathbb{N}$ such that if $ n\geq n_0$' as it is an a.a.s. result}
%such that if   $p\geq Cn^{-2/k}$ and  $G$ is an $n$-vertex graph where $r$ divides $n$ and $\delta(G)\geq (1-\frac{k}{r}+\gamma)n$, then $G\cup G(n,p)$  a.a.s. contains a perfect $K_r$-tiling. 
%covering all but at most $\alpha n$ vertices. 
%\end{theorem}
%Here we use $a.a.s.$ to mean that the probability that there is a $K_r$-tiling in $G\cup G(n,p)$ is a function of $n$ which tends to $1$ as $n$ tends to infinity. Thus, the above theorem implies the $(i)$-statement of Definition \ref{def:threshold} for Theorem \ref{thm:Main}. %\COMMENT{AT: not sure if we  need this paragraph here. Perhaps in notation section we should formally define a.a.s.? PM: Yea, so what I was trying to do here is point out the slight change of viewpoint from looking at a sequence of graphs to a single $n$-vertex graph. When we have a sequence of graphs, we have a sequence of probabilities which tend to 1, whereas now we just have one probability but it is a function of $n$. Probably I'm just overcomplicating things and this is not necessary though. }
%\subsection{Proof of \Cref{thm:Upper}}
Fix some sufficiently large $n\in r \mathbb N $  and let $G$ be an  $n$-vertex graph with $\delta(G)\geq (1-\frac{k}{r}+\gamma)n$. We will show that there exists $C=C(\gamma,k,r)>0$ 
such that if   $p\geq Cn^{-2/k}$, then $G':=G\cup G(n,p)$  a.a.s. contains a perfect $K_r$-tiling. Again, we split the proof according to the parameters. We first treat Cases~1 and 2 together (i.e. when $2\leq k\leq r/2$ or $k=r$). Here we avoid many of the technicalities which occur in Case~3 and the main scheme of the proof is clear. 

\smallskip

{\noindent \it Proof of Cases 1 and 2.} %\COMMENT{AT: quite a number of changes in this proof}
Suppose $2\leq k\leq r/2$ or $k=r$, and let $C,C'>0$ be chosen such that we can express $G(n,p)=\cup_{j=1}^4G_j$ with each $G_j$ a copy of $G(n,p')$ where $p'\geq C'n^{-2/k}$ and $C'>0$ is large enough to be able to draw the desired conclusions in what follows. 
Now fix $0<\eta<\min\{\frac{\gamma}{2000r^2},\eta_0\}$ where $\eta_0$ is as in Corollary~\ref{cor:case1+2absorbingstructure} and consider $X'\subseteq V(G)$ 
to be the subset generated by taking every vertex in $V(G)$ in $X'$ with probability $1.9\eta$, independently of the other vertices. With high probability, by Chernoff's theorem, 
we have that $1.8\eta n\leq |X'|\leq 2\eta n$ and for every vertex $v\in V(G)$, $|N_G(v)\cap X'|\geq (1-\frac{k}{r}+\frac{3\gamma}{4})|X'|$. 
Take an instance of $X'$ where this is the 
case and let $X:=X'$ if $|X'|$ is even and $X:=X'\cup\{x\}$ for some arbitrary vertex $x\in V(G)\setminus X'$ if $|X'|$ is odd. 
Apply Corollary~\ref{cor:case1+2absorbingstructure} to get  a $4$-bounded absorbing structure $\cA=(\Phi,Z,Z_1)$ in $G\cup G_1$ with flexibility $|X|/2$ and flexible set $Z_1=X$. Remark~\ref{boundy} implies $|V(\cA)|\leq 500r^2\eta n\leq \gamma n/4$. 

Then letting $V':=V(G)\setminus V(\cA)$, we have that  $\delta (G[V'])\geq (1-\frac{k}{r}+\frac{\gamma}{2})|V'|$. Choose $\alpha:= \min\{\alpha_{\ref{cor:embeddingK_rs}},\frac{\gamma\eta}{4r}\}$, 
where $\alpha_{\ref{cor:embeddingK_rs}}$ is the constant obtained when applying Corollary~\ref{cor:embeddingK_rs} with $r,k,\gamma /2, \eta$ playing the role of $r,k,\gamma, \beta$ respectively.

Apply Theorem~\ref{almost tiling thm} to obtain a $K_r$-tiling $\cK_1$ in $(G\cup G_2)[V']$ covering  all but at most $\alpha n$ vertices of $V'$. 
Let $Y$ denote the set of those vertices in $V'$ uncovered by $\cK_1$. Apply Corollary \ref{cor:embeddingK_rs} to obtain a $K_r$-tiling $\cK_2$ in $(G\cup G_3)[X \cup Y]$ which covers $Y$ and covers precisely $(r-1)|Y|\leq \gamma\eta n/2\leq \gamma |X|/2$ vertices of $X$. Let $\tilde{X}$ be the set of  those vertices in $X$ not covered by $\cK_2$. We have that $\delta(G[\tilde{X}])\geq (1-\frac{k}{r}+\frac{\gamma}{4})|\tilde{X}|$ so we can apply Theorem~\ref{almost tiling thm} to obtain a $K_r$-tiling $\cK_3'$  in 
$(G\cup G_4)[\tilde{X}]$ which covers all but at most $|X|/4$ vertices of $\tilde{X}$. 

By Remark~\ref{keyabsorbingproperty} we know that for any subset $X_1$ of $X$ of size $|X|/2$, there is a $K_r$-tiling covering precisely $V(\cA) \setminus X_1$. Thus,
$|V(\cA)|-|X|/2$ is divisible by $r$. Therefore, as  the only vertices in $V(G)$ uncovered by $V(\cK_1\cup \cK_2)$ are those from $(V(\cA)\setminus X) \cup \tilde{X}$,
 there must be a subtiling $\cK_3\subseteq \cK_3'$ which covers all but exactly $|X|/2$ vertices of $\tilde{X}$. 

Let $\overline{X}$ be the set of vertices of $X$ that are covered by cliques in $\cK_2\cup \cK_3$. Thus $|\overline{X}|=|X|/2$ and by Remark~\ref{keyabsorbingproperty} there is a $K_r$-tiling $\cK_4$ in $G\cup G_1$ covering precisely $V(\cA)\setminus \overline{X}$. Hence, $\cK:=\cK_1\cup \cK_2\cup \cK_3\cup \cK_4$ gives a perfect $K_r$-tiling of $G\cup G(n,p)$ as required. 
\qed

\medskip

If $r/2<k\leq r-1$, we have to overcome a few technicalities. The idea is to apply Corollary~\ref{cor:case3absorbingstructure} and to apply the same approach as above in each of the parts of the resulting partition to find a $K_r$-tiling. Of course we also have to incorporate the vertices of the exceptional class $W$ into copies of cliques in our tiling; this is straightforward using Corollary~\ref{cor:embeddingK_rs}.
So we cover these vertices first
 before embarking on tiling the majority of the graph. 

More subtle is a problem that arises from divisibility. That is, when we tile each part according to the scheme above, we cannot guarantee that we are left with a subset of the flexible set of the right size to apply the key property of the absorbing structure. Therefore we embed `crossing' copies of $K_r$ in our flexible sets in order to resolve this divisibility hurdle at the end of our process. We find these copies in the following manner. Consider the graph $F:=K_{\ceil{\frac{r-1}{2}},\floor{\frac{r-1}{2}}}$. Because of our minimum degree condition  and Lemma~\ref{lem:supersat},  every part $V_i$ contains at least $\gamma'n^{r-1}$ copies of $F$ for some $\gamma'>0$. Now let $\overline{F}$ be the graph consisting of a copy of $K_{\ceil{\frac{r-1}{2}}+1}$ and a copy of $K_{\floor{\frac{r-1}{2}}+1}$ joined at a single vertex $x$, say. If we consider $F$ and $\overline{F}-x$ to have the same vertex set so that $\overline{F}-x=K_{r-1}-E(F)$, then
 $F\cup \overline{F}$ is a copy of $K_r$. Also note that it follows from \Cref{PhiFobservations} that $\Phi_{\overline{F}}\geq C' n$ for $p\geq C'n^{-2/k}$. We will look for embeddings of $K_r=F\cup \overline{F}$ in $G\cup G(n,p)$ such that the vertex $x$ is mapped to one part of the partition and the $r-1$ other vertices lie in another part of the partition. 

\medskip

{\noindent \it Proof of Case 3.} %\COMMENT{AT: quite a few little changes in proof}
Suppose $r/2< k\leq r-1$, $q:=r-k$ and $c:=\ceil{r/q}$. Now let $C,C'>0$ be chosen so that we can express $G(n,p)=\cup_{j=1}^{4}G_j\cup_{i=1}^c (G_{i1}\cup G_{i2})$ with each $G_j, G_{i1}$ and $G_{i2}$  a copy of $G(n,p')$ where $p'\geq C'n^{-2/k}$ and $C'>0$ is large enough to be able to draw the desired conclusions in what follows.

We use our first copy of $G(n,p')$ to find the crossing copies of $K_r$ discussed above. Apply Corollary~\ref{cor:case3absorbingstructure}, letting $\alpha_1>0$ be the outcome of the corollary with input $r,k,\gamma$.
Choose $0<\eps<\min\{\alpha_1\gamma/(8r), \alpha_{\ref{cor:embeddingK_rs}}\}$, where $\alpha_{\ref{cor:embeddingK_rs}}$ is the  constant obtained when applying Corollary~\ref{cor:embeddingK_rs} with  $r,k,\gamma/4,1/2$
playing the roles of $r,k,\gamma, \beta$ respectively.
Thus, Corollary~\ref{cor:case3absorbingstructure} yields a partition
 $V_1,\ldots,V_\rho,W$ of  $V(G)$ where $\rho\leq c$ and $|W|\leq \eps n$. 
Note $\delta (G[V_i]) \geq (1-\frac{k}{r}+\frac{\gamma}{4})|V_i|$ for each $i \in [\rho]$.
Thus for each $i \in [\rho-1]$ and every subset $V'\subseteq V(G)$ of at least $n-c(r-1)r$ vertices,  Lemma~\ref{lem:supersat} implies there exists some $\gamma'=\gamma' (r,k,\gamma)>0$ such that there are at least $\gamma'n^r$ choices of pairs $(S,v)\in \binom{V'\cap V_{i+1}}{r-1}\times (V'\cap V_{i})$ such that $S$ hosts a copy of the graph $F:=K_{\ceil{\frac{r-1}{2}},\floor{\frac{r-1}{2}}}$ discussed above. 

Therefore, using that $\Phi_{\overline{F}}\geq C' n$ with the graph $\overline{F}$ as described above, we can apply Corollary~\ref{cor:simplegraphcontainment} to conclude that for any subset $V'$ of vertices of at least $n-c(r-1)r$ vertices and any $i\in[\rho-1]$, there is a copy of $K_r$ in $G\cup G_1$ which has $r-1$ vertices in $V_{i+1}$ and one vertex in $V_{i}$. Therefore, we can greedily choose copies of $K_r$ so that we have a set $\cR:=\cup_{i\in[\rho-1]}\cR_i$ of disjoint copies of $K_r$ in $G\cup G_1$  such that $\cR_i$ contains $r-1$ copies of $K_r$ with one vertex in $V_i$ and $r-1$ vertices of $V_{i+1}$. Let $\cR_\rho:=\emptyset$ and  $R_i:=V(\cR)\cap V_i$ for $i\in[\rho]$, where $V(\cR)$ denotes the vertices which feature in cliques in $\cR$. Note that $|R_1|=r-1$, $|R_2|=|R_3|=\cdots=|R_{\rho-1}|=r(r-1)$ and $|R_\rho|=(r-1)^2$.
We will incorporate these $R_i$ into our flexible sets in order to use the copies of $K_r$ that they define to fix divisibility issues that arise in the final stages of the argument.

\smallskip

Now fix $0<\eta<\min\{\frac{\gamma}{4000c2^cr^2},\eta_0\}$ where $\eta_0$ is as in Corollary~\ref{cor:case3absorbingstructure} and for each $i\in[\rho]$ consider $X'_i\subseteq V(G)$ to be a subset selected by taking every vertex in $V_i\setminus R_i$  with probability $1.9\eta$, independently of the other vertices. With high probability, by Chernoff's theorem, we have that $1.8\eta |V_i|\leq |X_i'|\leq 2\eta |V_i|-r(r-1)$ and for every vertex $v\in V_i$, $|N_G(v)\cap X'_i|\geq (1-\frac{k}{r}+\frac{\gamma}{8})|X_i'|$. Therefore, for each $i$, take an instance of $X'_i$ where this is the case and let $X_i:=X_i'\cup R_i$ if $|X_i'|+|R_i|$ is even and $X_i:=X_i'\cup R_i\cup\{x\}$ for some arbitrary vertex $x\in V_i\setminus (X'_i\cup R_i)$ if $|X'_i|+|R_i|$ is odd. Apply Corollary \ref{cor:case3absorbingstructure} to get  a collection  $\{\cA_i=(\Phi_i,Z_i,Z_{i1}):i\in[\rho]\}$ of absorbing structures in $G\cup G_2$ such that each $\cA_i$ has flexibility $|X_i|/2$ and flexible set $Z_{i1}=X_i$. By 
Remark~\ref{boundy} we have that $A:=\cup_{i\in[\rho]} V(\cA_i)$ is such that $|A|\leq 125c2^{c+2}r^2\eta n\leq \gamma n/8$.

Therefore, setting $V':=V(G)\setminus (W\cup A)$, we have that for every $w\in W\cup V'$, $|N_G(w)\cap V'|\geq (1-\frac{k}{r}+\frac{\gamma}{4})|V'|$ and so an application of 
Corollary~\ref{cor:embeddingK_rs} yields a $K_r$-tiling $\cK_1$ in $G\cup G_3$ of $|W|$ cliques, each using one vertex of $W$ and $r-1$ vertices of $V'$. 
Setting $V'':=V(G)\setminus (A\cup V(\cK_1))$, we have that $\delta(G[V''])\geq (1-\frac{k}{r}+\frac{\gamma}{8})|V''|$. So, as in the previous proof, we let $\alpha_2:= \min\{\alpha_{\ref{cor:embeddingK_rs}},\frac{\gamma\eta}{16r}\}$, where $\alpha_{\ref{cor:embeddingK_rs}}$ is obtained from Corollary~\ref{cor:embeddingK_rs} (where $\gamma/8$ and $\eta$ play the roles of $\gamma$ and $\beta$ respectively),  and we apply Theorem~\ref{almost tiling thm} to obtain
 a $K_r$-tiling $\cK_2$ in $(G\cup G_4)[V'']$ covering  all but at most $\alpha_2 n$ vertices of $V''$. Let $Y$ be the set of vertices from $V''$ uncovered by $\cK_2$ and set $Y_i:=Y\cap V_i$ for each $i\in[\rho]$. 
%Now we also define, for $i\in[\rho]$, the set $U_i:=X_i\setminus R_i$ of vertices in the corresponding flexible set which are not contained in our reserved cliques $\cR$. 
%Now for each $i\in[\rho]$ we have that every $v$ vertex in $Y_i\cup U_i$ has $|N(v)\cap U_i|=$ 

Now for each $i\in[\rho]$ a simple application of Corollary~\ref{cor:embeddingK_rs} yields a $K_r$-tiling $\cK_{i1}$ in $G\cup G_{i1}$ which covers $Y_i$ and 
uses precisely $(r-1)|Y_i|\leq \gamma\eta n/16 \leq \gamma |X'_i|/16$ vertices of $X'_i$. Note that we do not use any vertices of $R=\cup_{i\in[\rho]}R_i$ in these cliques. For each $i\in[\rho]$ let $\tilde{X}_i$ be the vertices of $X_i\setminus R_i$ not involved in copies of $K_r$ in $\cK_{i1}$. As $\delta(G[\tilde{X}_{i}])\geq (1-\frac{k}{r}+\frac{\gamma}{16})|\tilde{X}_{i}|$, we can apply Theorem~\ref{almost tiling thm} to obtain a 
$K_r$-tiling $\cK'_{i}$ in $(G\cup G_{i2})[\tilde{X}_{i}]$ which covers all but at most $|X_i|/4$ vertices of $\tilde{X}_i$, for each $i\in[\rho]$.

Note that we will not use the full tilings $\cK'_{i}$ in our final tiling. 
So (ignoring for now the tilings $\cK'_{i}$), it remains to cover the vertices in $(V(\cA_i)\setminus X_i)\cup R_i\cup \tilde{X}_i$ for each $i\in[\rho]$. 
We do so by means of the following algorithm. We initiate with the $\cK'_{i}, \cR_i$ as above and set $\overline{Z}_i:=V(\cK_{i1})\cap X_i$ and $\cK_{i2}:=\emptyset$ for all $i\in[\rho]$ and $i'=1$. Now whilst $|\overline{Z}_{i'}|\leq |X_{i'}|/2-r+1$, remove a clique from $\cK'_{i'}$, add it to $\cK_{i'2}$ and add its vertices to $\overline{Z}_{i'}$. Once this process stops, add $|X_{i'}|/2-|\overline{Z}_{i'}|$ copies of $K_r$ in $\cR_{i'}$ to $\cK_{i'2}$, and add all their vertices in $X_{j}$ to $\overline{Z}_j$ for $j=i',i'+1$. If $i'\leq \rho -1$, repeat this process, setting $i'=i'+1$. Note that when $i'=\rho$, $\cR_{i'}=\emptyset$ and there are no cliques which we could add in this process. However, setting $\cK_0=\cK_1\cup \cK_2 \cup_{i\in[\rho]} (\cK_{i1}\cup \cK_{i2})$ we have that $|V(\cK_0)|$, $n$, and  $|V(\cA_i)|-|X_i|/2$ are divisible by $r$ for each $i$, so we can deduce that the algorithm takes no cliques from $\cR_\rho$ and terminates with $|\overline{Z}_i|=|X_i|/2$ for all $i\in[\rho]$. 

Finally, by the key property of the absorbing structure (Remark~\ref{keyabsorbingproperty}), we have that for each $i\in[\rho]$, there is a $K_r$-tiling $\cK_{i3}$ in $G\cup G_2$
covering  $V(\cA_i)\setminus \overline{Z}_i$  and thus $\cup_{i\in[\rho]}\cK_{i3}\cup\cK_0$ is the desired perfect $K_r$-tiling in $G\cup G(n,p)$.
\qed

\section{Concluding remarks}\label{sec:conc}
In this paper we have almost completely resolved the perfect $K_r$-tiling problem for randomly perturbed graphs. The only cases that Theorem~\ref{thm:Main} does not resolve is when 
$\alpha =1/r,2/r,\dots, (r-2)/r$. Note however for $\alpha :=1-k/r$ where $2 \leq k \leq r- 1$ we know that
\[n ^{-2/k} \leq p(K_r,\alpha) \leq n^{-2/(k+1)}.\]
In fact, one can slightly improve the lower bound, giving that $p(K_r,\alpha)\geq p(K_k,0)=n^{-2/k}(\log n)^{\frac{2}{k^2-k}}$. Indeed, as in our lower bound construction (Section \ref{sec:lower}),  take $G$ to be complete graph on $n$ vertices with a clique of size $kn/r$ removed. Letting $A$ be the resulting independent set of vertices,  if $p=o(p(K_k,0))$ then a.a.s. we have that the number of copies of $K_{k+1}$ in $G(n,p)$ is less than, say, $(\log n)^3$ by Markov's inequality, whilst the number of vertices in $A$ which do not lie in copies of $K_k$ is at least $n^{1-o(1)}$, as can be seen by a second moment calculation (see e.g. \cite[Theorem 3.22]{jlr}). This precludes the existence of a  perfect $K_r$-tiling in $G \cup G(n,p)$, as the average intersection of a clique in such a tiling with the vertex set $A$ would be $k$ and we cannot tile $G(n,p)[A]$ with a family of cliques whose average size is $k$ given the restrictions above. This leaves a gap between the upper and lower bounds and it would be very interesting to resolve the problem for these `boundary' cases.

It is also of interest to consider the analogous problem for perfect $H$-tilings for arbitrary graphs $H$. Note that whilst the main result from~\cite{bwt2} determines $p(H,\alpha)$ for all graphs $H$ and $0<\alpha<1/|H|$, the problem is still wide open for larger values of $\alpha$. The methods from our paper are likely to be useful for the general problem, though we suspect how $p(H,\alpha)$  `jumps' as $\alpha$ increases will depend heavily on the structure of $H$. Thus we believe it would be a significant 
challenge to prove such a general result.

\section*{Acknowledgments} 
Much of the research in this paper was undertaken whilst the authors were visiting the Instituto Nacional de Matem\'atica Pura e Aplicada in Rio de Janeiro as part of the Graphs@IMPA thematic program in 2018. 
We are grateful to IMPA for the nice working environment. We are also extremely grateful to Louis DeBiasio %\COMMENT{PM:added as I had a lot of conversations with him about little things when writing this.} 
and Shagnik Das for many useful conversations and to the referees for their careful reviews and their suggestions. In particular, we thank one of the referees who brought our attention to the improvement on the lower bound in the concluding remarks.

\bibliographystyle{abbrv}
\bibliography{refs}
\end{document}